\documentclass{amsart}
\usepackage{amsmath}
\usepackage{amsfonts}
\usepackage{amsthm}
\usepackage{amsbsy}
\usepackage{amssymb}
\usepackage{amsthm}
\usepackage{enumerate}
\usepackage[margin=1in]{geometry}
\usepackage{hyperref}
\usepackage{mathrsfs}
\usepackage{mathtools}
\usepackage{microtype}
\usepackage[nobysame]{amsrefs}

\newtheorem{thm}{Theorem}[section]
\newtheorem{lem}[thm]{Lemma}
\newtheorem{prop}[thm]{Proposition}
\newtheorem{cor}[thm]{Corollary}

\theoremstyle{definition}
\newtheorem{defn}[thm]{Definition}
\newtheorem{nota}[thm]{Notation}
\newtheorem{rem}[thm]{Remark}
\newtheorem{exam}[thm]{Example}
\newtheorem{cons}[thm]{Construction}

\newcommand{\bC}{{\mathbb{C}}}
\newcommand{\bE}{{\mathbb{E}}}
\newcommand{\bF}{{\mathbb{F}}}
\newcommand{\bN}{{\mathbb{N}}}
\newcommand{\bR}{{\mathbb{R}}}

\newcommand{\A}{{\mathcal{A}}}
\newcommand{\B}{{\mathcal{B}}}
\newcommand{\C}{{\mathcal{C}}}
\newcommand{\D}{{\mathcal{D}}}
\newcommand{\E}{{\mathcal{E}}}
\newcommand{\F}{{\mathcal{F}}}

\newcommand{\I}{{\mathcal{I}}}
\newcommand{\K}{{\mathcal{K}}}
\renewcommand{\L}{{\mathcal{L}}}

\newcommand{\N}{{\mathcal{N}}}

\newcommand{\R}{{\mathcal{R}}}
\renewcommand{\S}{{\mathcal{S}}}

\newcommand{\X}{{\mathcal{X}}}
\newcommand{\Y}{{\mathcal{Y}}}
\newcommand{\Z}{{\mathcal{Z}}}

\newcommand{\qand}{\quad\text{and}\quad}

\newcommand{\alg}{\mathrm{alg}}

\renewcommand{\th}{\mathrm{th}}

\newcommand{\lat}{\mathrm{lat}}
\newcommand{\capp}{\mathrm{cap}}

\usepackage{tikz}
\usetikzlibrary{shapes,snakes,calc,arrows}
\usetikzlibrary{decorations.pathreplacing,shapes.geometric}
\usetikzlibrary{calc,positioning}
\tikzset{Box/.style={very thick, rounded corners}}
\tikzset{marked/.style={star, star point height = .75mm, star points =5, fill=black,minimum size=2mm, inner sep=0mm} }
\tikzset{verythickline/.style = {line width=7pt}}
\tikzset{thickline/.style = {line width=5pt}}
\tikzset{medthick/.style = {line width=3pt}}
\tikzset{med/.style = {line width=2pt}}
\tikzset{count/.style = {fill=white,circle,draw,thin, inner sep=2pt}}
\tikzset{rcount/.style = {fill=white,rectangle,draw,thin,inner sep=2pt, rounded corners}}
\tikzset{cpr/.style = {draw,fill=white,rectangle,thin, rounded corners}}

\definecolor{ggreen}{HTML}{00BB33}

\begin{document}

\nocite{*}

\title[Conditionally bi-free independence with amalgamation]{Conditionally bi-free independence with amalgamation}

\author{Yinzheng Gu and Paul Skoufranis}

\address{Department of Mathematics and Statistics, Queen's University, Jeffrey Hall, Kingston, Ontario, K7L 3N6, Canada}
\email{gu.y@queensu.ca}

\address{Department of Mathematics and Statistics, York University, 4700 Keele Street, Toronto, Ontario, M3J 1P3, Canada}
\email{pskoufra@yorku.ca}

\date{\today}

\begin{abstract}
In this paper, we introduce the notion of conditionally bi-free independence in an amalgamated setting. We define operator-valued conditionally bi-multiplicative pairs of functions  and construct operator-valued conditionally bi-free moment and cumulant functions. It is demonstrated that conditionally bi-free independence with amalgamation is equivalent to the vanishing of mixed operator-valued bi-free and conditionally bi-free cumulants. Furthermore, an operator-valued conditionally bi-free partial $\mathcal{R}$-transform is constructed and various operator-valued conditionally bi-free limit theorems are studied.
\end{abstract}

\maketitle

\section{Introduction}

The notion of conditionally free (c-free for short) independence was introduced in \cite{BLS1996} as a generalization of the notion of free independence to two-state systems.  In our previous paper \cite{GS2016} we introduced the notion of conditionally bi-free (c-bi-free for short) independence in order to study the non-commutative left and right actions of algebras on a reduced c-free product simultaneously.  Thus conditional bi-freeness is an extension of the notion of bi-free independence \cite{V2014} to two-state systems. Moreover \cite{GS2016} introduced c-$(\ell, r)$-cumulants and demonstrated that a family of pairs of algebras in a two-state non-commutative probability space is conditionally bi-free if and only if mixed $(\ell, r)$- and c-$(\ell, r)$-cumulants vanish.

In \cite{V1995} Voiculescu generalized his own notion of free independence by replacing the scalars with an arbitrary algebra thereby obtaining the notion of free independence with amalgamation (see also \cite{S1998} for the combinatorial aspects).  For c-free independence, the generalization to an amalgamated setting over a pair of algebras was done by Popa in \cite{P2008} (see also \cite{M2002}). On the other hand, the framework for generalizing bi-free independence to an amalgamated setting was suggested by Voiculescu in \cite{V2014}*{Section 8} and the theory was fully developed in \cite{CNS2015-2}.

The main goal of this paper is to extend the notion of c-bi-free independence to an amalgamated setting over a pair of algebras.  Furthermore, we demonstrate that the combinatorics of conditionally bi-free probability and bi-free probability with amalgamation, which are governed by the lattice of bi-non-crossing partitions, are specific instances of more general combinatorial structures.

Including this introduction this paper contains nine sections which are structured as follows. Section \ref{sec:prelims} briefly reviews some of the background material pertaining to conditionally bi-free probability and bi-free probability with amalgamation from \cites{CNS2015-1, CNS2015-2, GS2016}. In particular, the notions bi-non-crossing partitions and diagrams, their lateral refinements and cappings, interior and exterior blocks, $\B$-$\B$-non-commutative probability spaces, operator-valued bi-multiplicative functions, and the operator-valued bi-free moment and cumulant functions are recalled.

Section \ref{sec:c-bi-free-defn} introduces the structures studied within conditionally bi-free independence with amalgamation. We define the notion of a $\B$-$\B$-non-commutative probability space with a pair of $(\B, \D)$-valued expectations $(\A, \bE, \bF, \varepsilon)$ (see Definition \ref{BBncpsBD}), demonstrate a representation of $\A$ as linear operators on a $\B$-$\B$-bimodule with a pair of specified $(\B, \D)$-valued states (see Theorem \ref{Embedding}), and define the notion of conditionally bi-free independence with amalgamation over $(\B, \D)$ thereby generalizing conditionally bi-free independence to the operator-valued setting and bi-free independence with amalgamation to the two-state setting.

Section \ref{sec:pairs-of-fns} introduces the notion of an operator-valued conditionally bi-multiplicative pair of functions (see Definition \ref{CondBiMulti}). Each such pair consists of two functions where the first function is operator-valued bi-multiplicative (see \cite{CNS2015-2}*{Definition 4.2.1}) and the second function is defined via a certain rule using the first function. Furthermore, operator-valued conditionally bi-free moment and cumulant pairs (see Definitions \ref{CBFMomentPair} and \ref{OpVCBFCumulants}) are introduced and shown to be operator-valued conditionally bi-multiplicative.

Sections \ref{sec:moment-express} and \ref{sec:additivity} provide alternate characterizations of conditionally bi-free independence with amalgamation. More precisely, Section \ref{sec:moment-express} demonstrates through Theorem \ref{MomentFormulae} that a family of pairs of $\B$-algebras in a $\B$-$\B$-non-commutative probability space with a pair of $(\B, \D)$-valued expectations $(\A, \bE, \bF, \varepsilon)$ is c-bi-free over $(\B, \D)$ if and only if certain moment expressions with respect to $\bE$ and $\bF$ are satisfied. On the other hand, Section \ref{sec:additivity} demonstrates through Theorem \ref{VanishingEquiv} that a family of pairs of $\B$-algebras is c-bi-free over $(\B, \D)$ if and only if their mixed operator-valued bi-free and conditionally bi-free cumulants vanish.

Section \ref{sec:additional} provides additional properties such as the vanishing of operator-valued conditionally bi-free cumulants when a left or right $\B$-operator is input, how c-bi-free independence over $(\B, \D)$ can be deduced from c-free independence over $(\B, \D)$ under certain conditions, and how operator-valued conditionally bi-free cumulants involving products of operators may be computed.

In Section \ref{sec:R-transform}, an operator-valued conditionally bi-free partial $\mathcal{R}$-transform is constructed as the operator-valued analogue of the conditionally bi-free partial $\mathcal{R}$-transform (see \cite{GS2016}*{Definition 5.3}). As with the operator-valued bi-free partial $\mathcal{R}$-transform (see \cite{S2015}*{Section 5}), the said transform is also a function of three $\B$-variables, and a formula relating it to the moment series is proved using combinatorics.
Finally, in Section \ref{sec:limit-thms}, operator-valued c-bi-free distributions are discussed and various operator-valued c-bi-free limit theorems are studied.

\section{Preliminaries}\label{sec:prelims}

In this section, we review the necessary background on conditionally bi-free probability and operator-valued bi-free probability required for this paper.

\subsection{Conditionally bi-free probability}

We recall several definitions and results relating to conditionally bi-free probability. For more precision, see \cite{GS2016}.

\begin{defn}
Let $(\A, \varphi, \psi)$ be a two-state non-commutative probability space; that is, $\A$ is a unital algebra and $\varphi, \psi: \A \to \bC$ are unital linear functionals. A \textit{pair of algebras} in $\A$ is an ordered pair $(A_\ell, A_r)$ of unital subalgebras of $\A$.
\end{defn}

\begin{defn}
A family $\{(A_{k, \ell}, A_{k, r})\}_{k \in K}$ of pairs of algebras in a two-state non-commutative probability space $(\A, \varphi, \psi)$ is said to be \textit{conditionally bi-freely independent} (or \textit{c-bi-free} for short) with respect to $(\varphi, \psi)$ if there is a family of two-state vector spaces with specified state-vectors $\{(\X_k, \X_k^\circ, \xi_k, \varphi_k)\}_{k \in K}$ and unital homomorphisms
\[\ell_k: A_{k, \ell} \to \L(\X_k) \qand r_k: A_{k, r} \to \L(\X_k)\]
such that the joint distribution of $\{(A_{k, \ell}, A_{k, r})\}_{k \in K}$ with respect to $(\varphi, \psi)$ is equal to the joint distribution of the family
\[\{(\lambda_k \circ \ell_k(A_{k, \ell}), \rho_k \circ r_k(A_{k, r}))\}_{k \in K}\]
in $\L(\X)$ with respect to $(\varphi_\xi, \psi_\xi)$, where $(\X, \X^\circ, \xi, \varphi) = *_{k \in K}(\X_k, \X_k^\circ, \xi_k, \varphi_k)$.
\end{defn}

In general, a map $\chi: \{1, \dots, n\} \to \{\ell, r\}$ is used to designate whether the $k^{\mathrm{th}}$ operator in a sequence of $n$ operators is a left operator (when $\chi(k) = \ell$) or a right operator (when $\chi(k) = r$), a map $\omega: \{1, \dots, n\} \to I \sqcup J$ is used to designate the index of the $k^\th$ operator, and a map $\omega: \{1, \dots, n\} \to K$ is used to designate from which collection of operators the $k^\th$ operator hails from. 

Given $\omega: \{1, \dots, n\} \to I \sqcup J$ for non-empty disjoint index sets $I$ and $J$, we define the corresponding map $\chi_\omega: \{1, \dots, n\} \to \{\ell, r\}$ by
\[\chi_\omega(k) = \begin{cases}
\ell &\text{if } \omega(k) \in I\\
r &\text{if } \omega(k) \in J
\end{cases}.\]
Given a map $\omega: \{1, \dots, n\} \to K$, we may view $\omega$ as a partition of $\{1, \dots, n\}$ with blocks $\{\omega^{-1}(\{k\})\}_{k \in K}$. Thus $\pi \leq \omega$ denotes $\pi$ is a refinement of the partition induced by $\omega$.

For the basic definitions and combinatorics of bi-free probability that will be used in this paper, we refer the reader to \cites{CNS2015-1, CNS2015-2, MN2015, V2014} or the summary given in \cite{GS2016}*{Section 2}. Particular attention should be paid to:
\begin{itemize}
\item the set $\B\N\C(\chi)$ of bi-non-crossing partitions with respect to $\chi: \{1, \dots, n\} \to \{\ell, r\}$, and the minimal and maximal elements $0_\chi$ and $1_\chi$ of $\B\N\C(\chi)$ (see \cite{CNS2015-2}*{Definition 2.1.1});

\item  for $m, n \geq 0$ with $m + n \geq 1$, $1_{m, n}$ denotes $1_{\chi_{m, n}}$ where $\chi_{m, n}: \{1, \dots, m + n\} \to \{\ell, r\}$ is such that $\chi_{m, n}(k) = \ell$ if $k \leq m$ and $\chi_{m, n}(k) = r$ if $k > m$;

\item the M\"{o}bius function $\mu_{\B\N\C}$ on the lattice of bi-non-crossing partitions (see \cite{CNS2015-1}*{Remark 3.1.4});

\item the total ordering $\prec_\chi$ on $\{1, \dots, n\}$ and the notion of $\chi$-interval induced by $\chi: \{1, \dots, n\} \to \{\ell, r\}$ (see \cite{CNS2015-2}*{Definition 4.1.1});

\item the set $\L\R(\chi, \omega)$ of shaded $\L\R$-diagrams corresponding to $\chi: \{1, \dots, n\} \to \{\ell, r\}$ and $\omega: \{1, \dots, n\} \to K$, and the subsets $\L\R_k(\chi, \omega)$ ($1 \leq k \leq n$) of $\L\R(\chi, \omega)$ with exactly $k$ spines reaching the top (see \cite{CNS2015-1}*{Section 2.5});

\item the notion $\leq_\lat$ of lateral refinement (see \cite{CNS2015-1}*{Definition 2.5.5});

\item the family $\{\kappa_\chi: \A^n \to \bC\}_{n \geq 1, \chi: \{1, \dots, n\} \to \{\ell, r\}}$ of $(\ell, r)$-cumulants (see \cite{MN2015}*{Definition 5.2}).
\end{itemize}

Inspired by the `vanishing of mixed $(\ell, r)$-cumulants' characterization of bi-free independence and the `vanishing of mixed free and c-free cumulants' characterization of c-free independence, we introduced in \cite{GS2016}*{Subsection 3.3} the family of c-$(\ell, r)$-cumulants using bi-non-crossing partitions that are divided into two types. More precisely, a block $V$ of a bi-non-crossing partition $\pi \in \B\N\C(\chi)$ is said to be \textit{interior} if there exists another block $W$ of $\pi$ such that $\min_{\prec_\chi}(W) \prec_\chi \min_{\prec_\chi}(V)$ and $\max_{\prec_\chi}(V) \prec_\chi \max_{\prec_\chi}(W)$, where $\min_{\prec_\chi}$ and $\max_{\prec_\chi}$ denote the minimum and maximum elements with respect to $\prec_\chi$. A block of $\pi$ is said to be \textit{exterior} if it is not interior. The family 
\[
\{\K_\chi: \A^n \to \bC\}_{n \geq 1, \chi: \{1, \dots, n\} \to \{\ell, r\}}
\]
of \textit{c-$(\ell, r)$-cumulants} of a two-state non-commutative probability space $(\A, \varphi, \psi)$ is recursively defined by
\[\varphi(a_1\cdots a_n) = \sum_{\pi \in \B\N\C(\chi)}\K_{\pi}(a_1, \dots, a_n),\]
where
\[\K_{\pi}(a_1, \dots, a_n) = \left(\prod_{\substack{V \in \pi\\V\,\mathrm{interior}}}\kappa_{\chi|_V}((a_1, \dots, a_n)|_V)\right)\left(\prod_{\substack{V \in \pi\\V\,\mathrm{exterior}}}\K_{\chi|_V}((a_1, \dots, a_n)|_V)\right),\]
for all $n \geq 1$, $\chi: \{1, \dots, n\} \to \{\ell, r\}$, and $a_1, \dots, a_n \in \A$.

Furthermore, as noticed in \cite{GS2016}*{Section 4}, in order to obtain a moment formula for conditionally bi-free independence, additional sets of shaded diagrams and terminology are required.

\begin{defn}\label{ShadedDiagrams}
Let $n \geq 1$, $\chi: \{1, \dots, n\} \to \{\ell, r\}$, and $\omega: \{1, \dots, n\} \to K$ be given.
\begin{enumerate}[$\qquad(1)$]
\item For $0 \leq k \leq n$, let $\L\R_k^\lat(\chi, \omega)$ denote the set of all diagrams that can be obtained from $\L\R_k(\chi, \omega)$ under later refinement (i.e., cutting spines that do not reach the top). For $D' \in \L\R_k^\lat(\chi, \omega)$ and $D \in \L\R_k(\chi, \omega)$, write $D \geq_\lat D'$ if $D'$ can be obtained by laterally refining $D$. Moreover, let
\[\L\R^\lat(\chi, \omega) = \bigcup_{k = 0}^n\L\R_k^\lat(\chi, \omega).\]

\item Let $0 \leq k \leq n$ and $D \in \L\R_k^\lat(\chi, \omega)$. A diagram $D'$ is said to be a \textit{capping} of $D$, denoted $D \geq_\capp D'$, if $D' = D$ or $D'$ can be obtained by removing spines from $D$ that reach the top. Let $\L\R_m^{\lat\capp}(\chi, \omega)$ denote the set of all diagrams with $m$ spines reaching the top that can be obtained by capping some $D \in \L\R_k^\lat(\chi, \omega)$ with $k \geq m$. Moreover, let
\[\L\R^{\lat\capp}(\chi, \omega) = \bigcup_{m = 0}^n\L\R_m^{\lat\capp}(\chi, \omega).\]

\item For $D \in \L\R_m^{\lat\capp}(\chi, \omega)$, let $|D| = (\text{number of blocks of } D) + m$.

\item Let $0 \leq m \leq n$, $k \geq m$, $D \in \L\R_k(\chi, \omega)$, and $D' \in \L\R_m^{\lat\capp}(\chi, \omega)$. We say that $D$ \textit{laterally caps} to $D'$, denoted $D \geq_{\lat\capp} D'$, if there exists $D'' \in \L\R_k^\lat(\chi, \omega)$ such that $D \geq_\lat D''$ and $D'' \geq_\capp D'$.
\end{enumerate}
\end{defn}

Suppose $a_1, \dots, a_n$ are elements in a two-state non-commutative probability space $(\A, \varphi, \psi)$, and $D \in \L\R^{\lat\capp}(\chi, \omega)$ with blocks $V_1, \dots, V_p$ whose spines do not reach the top and $W_1, \dots, W_q$ whose spines reach the top. Writing $V_i = \{r_{i, 1} < \cdots < r_{i, s_i}\}$ and $W_j = \{r_{j, 1} < \cdots < r_{j, t_j}\}$, we define
\[\varphi_D(a_1, \dots, a_n) = \prod_{i = 1}^p\psi(a_{r_{i, 1}}\cdots a_{r_{i, s_i}})\prod_{j = 1}^q\varphi(a_{r_{j, 1}}\cdots a_{r_{j, t_j}}).\]
Under the above notation, the following moment type characterization and vanishing of mixed cumulants characterization were established in \cite{GS2016}*{Theorems 4.1 and 4.8}.

\begin{thm}\label{CBFMoments}
A family $\{(A_{k, \ell}, A_{k, r})\}_{k \in K}$ of pairs of algebras in a two-state non-commutative probability space $(\A, \varphi, \psi)$ is c-bi-free with respect to $(\varphi, \psi)$ if and only if
\begin{equation}\label{psi-Moment}
\psi(a_1\cdots a_n) = \sum_{\pi \in \B\N\C(\chi)}\left[\sum_{\substack{\sigma \in \B\N\C(\chi)\\\pi \leq \sigma \leq \omega}}\mu_{\B\N\C}(\pi, \sigma)\right]\psi_\pi(a_1, \dots, a_n)
\end{equation}
and
\begin{equation}\label{phi-Moment}
\varphi(a_1\cdots a_n) = \sum_{D \in \L\R^{\lat\capp}(\chi, \omega)}\left[\sum_{\substack{D' \in \L\R(\chi, \omega)\\D' \geq_{\lat\capp} D}}(-1)^{|D| - |D'|}\right]\varphi_D(a_1, \dots, a_n)
\end{equation}
for all $n \geq 1$, $\chi: \{1, \dots, n\} \to \{\ell, r\}$, $\omega: \{1, \dots, n\} \to K$, and $a_1, \dots, a_n \in \A$ with $a_k \in A_{\omega(k), \chi(k)}$.

Equivalently, for all $n \geq 2$, $\chi: \{1, \dots, n\} \to \{\ell, r\}$, $\omega: \{1, \dots, n\} \to K$, and $a_k$ as above, we have
\[\kappa_\chi(a_1, \dots, a_n) = \K_\chi(a_1, \dots, a_n) = 0\]
whenever $\omega$ is not constant.
\end{thm}

\subsection{Bi-free probability with amalgamation}

Now we recall bi-free probability in an amalgamated setting. Since our constructions for operator-valued conditionally bi-free independence in Section \ref{sec:c-bi-free-defn} are very similar, we shall only present the essential concepts.  Please refer to \cite{CNS2015-2}*{Section 3} or the summary given in \cite{S2015}*{Section 2} for complete details. In particular, the following definitions and results will be generalized:

\begin{itemize}
\item a $\B$-$\B$-bimodule with a specified $\B$-valued state $(\X, \X^\circ, \mathfrak{p})$ (see \cite{CNS2015-2}*{Definition 3.1.1});

\item the free product with amalgamation over $\B$ of a family $\{(\X_k, \X_k^\circ, \mathfrak{p}_k)\}_{k \in K}$ of $\B$-$\B$-bimodules with specified $\B$-valued states (see \cite{CNS2015-2}*{Construction 3.1.7});

\item a $\B$-$\B$-non-commutative probability space $(\A, \bE, \varepsilon)$ with left and right algebras $\A_\ell$ and $\A_r$ (see \cite{CNS2015-2}*{Definition 3.2.1});

\item any $\B$-$\B$-non-commutative probability can be represented on a $\B$-$\B$-bimodule with a specified $\B$-valued state (see \cite{CNS2015-2}*{Theorem 3.2.4}).
\end{itemize}

Furthermore, in order to discuss operator-valued bi-free probability, one needs the correct notions for moment and cumulant functions, which we now review in greater depth.

\begin{defn}\label{BiMulti}
Let $(\A, \mathbb{E}, \varepsilon)$ be a $\mathcal{B}$-$\mathcal{B}$-non-commutative probability space and let
\[\Psi: \bigcup_{n \geq 1}\bigcup_{\chi: \{1, \dots, n\} \to \{\ell, r\}}\B\N\C(\chi) \times \A_{\chi(1)} \times \cdots \times \A_{\chi(n)} \to \B\]
be a function that is linear in each $\A_{\chi(k)}$. We say that $\Psi$ is \textit{operator-valued bi-multiplicative} if for every $\chi: \{1, \dots, n\} \to \{\ell, r\}$, $Z_k \in \A_{\chi(k)}$, $b \in \B$, and $\pi \in \B\N\C(\chi)$, the following four conditions hold.
\begin{enumerate}[$\qquad(1)$]
\item Let
\[q = \max\{k \in \{1, \dots, n\} \, \mid \, \chi(k) \neq \chi(n)\}.\]
If $\chi(n) = \ell$, then
\[\Psi_{1_\chi}(Z_1, \dots, Z_{n - 1}, Z_nL_b) = \begin{cases}
\Psi_{1_\chi}(Z_1, \dots, Z_{q - 1}, Z_qR_b, Z_{q + 1}, \dots, Z_n) &\text{if } q \neq -\infty\\
\Psi_{1_\chi}(Z_1, \dots, Z_{n - 1}, Z_n)b &\text{if } q = -\infty
\end{cases}.\]
If $\chi(n) = r$, then
\[\Psi_{1_\chi}(Z_1, \dots, Z_{n - 1}, Z_nR_b) = \begin{cases}
\Psi_{1_\chi}(Z_1, \dots, Z_{q - 1}, Z_qL_b, Z_{q + 1}, \dots, Z_n) &\text{if } q \neq -\infty\\
b\Psi_{1_\chi}(Z_1, \dots, Z_{n - 1}, Z_n) &\text{if } q = -\infty
\end{cases}.\]

\item Let $p \in \{1, \dots, n\}$, and let
\[q = \max\{k \in \{1, \dots, n\} \, \mid \, \chi(k) = \chi(p), k < p\}.\]
If $\chi(p) = \ell$, then
\[\Psi_{1_\chi}(Z_1, \dots, Z_{p - 1}, L_bZ_p, Z_{p + 1}, \dots, Z_n) = \begin{cases}
\Psi_{1_\chi}(Z_1, \dots, Z_{q - 1}, Z_qL_b, Z_{q + 1}, \dots, Z_n) &\text{if } q \neq -\infty\\
b\Psi_{1_\chi}(Z_1, Z_2, \dots, Z_n) &\text{if } q = -\infty
\end{cases}.\]
If $\chi(p) = r$, then
\[\Psi_{1_\chi}(Z_1, \dots, Z_{p - 1}, R_bZ_p, Z_{p + 1}, \dots, Z_n) = \begin{cases}
\Psi_{1_\chi}(Z_1, \dots, Z_{q - 1}, Z_qR_b, Z_{q + 1}, \dots, Z_n) &\text{if } q \neq -\infty\\
\Psi_{1_\chi}(Z_1, Z_2, \dots, Z_n)b &\text{if } q = -\infty
\end{cases}.\]

\item Suppose that $V_1, \dots, V_m$ are $\chi$-intervals ordered by $\prec_\chi$ which partition $\{1, \dots, n\}$, each a union of blocks of $\pi$. Then
\[\Psi_\pi(Z_1, \dots, Z_n) = \Psi_{\pi|_{V_1}}((Z_1, \dots, Z_n)|_{V_1})\cdots\Psi_{\pi|_{V_m}}((Z_1, \dots, Z_n)|_{V_m}).\]

\item Suppose that $V$ and $W$ partition $\{1, \dots, n\}$, each a union of blocks of $\pi$, $V$ is a $\chi$-interval, and
\[\min_{\prec_\chi}(\{1, \dots, n\}), \max_{\prec_\chi}(\{1, \dots, n\}) \in W.\]
Let
\[p = \max_{\prec_\chi}\left(\left\{k \in W \, \mid \, k \prec_\chi \min_{\prec_\chi}(V)\right\}\right) \qand q = \min_{\prec_\chi}\left(\left\{k \in W \, \mid \, \max_{\prec_\chi}(V) \prec_\chi k\right\}\right).\]
Then
\begin{align*}
\Psi_\pi(Z_1, \dots, Z_n) &= \begin{cases}
\Psi_{\pi|_{W}}\left(\left(Z_1, \dots, Z_{p - 1}, Z_pL_{\Psi_{\pi|_{V}}\left((Z_1, \dots, Z_n)|_{V}\right)}, Z_{p + 1}, \dots, Z_n\right)|_{W}\right) &\text{if } \chi(p) = \ell\\
\Psi_{\pi|_{W}}\left(\left(Z_1, \dots, Z_{p - 1}, R_{\Psi_{\pi|_{V}}\left((Z_1, \dots, Z_n)|_{V}\right)}Z_p, Z_{p + 1}, \dots, Z_n\right)|_{W}\right) &\text{if } \chi(p) = r
\end{cases}\\
&= \begin{cases}
\Psi_{\pi|_{W}}\left(\left(Z_1, \dots, Z_{q - 1}, L_{\Psi_{\pi|_{V}}\left((Z_1, \dots, Z_n)|_{V}\right)}Z_q, Z_{q + 1}, \dots, Z_n\right)|_{W}\right) &\text{if } \chi(q) = \ell\\
\Psi_{\pi|_{W}}\left(\left(Z_1, \dots, Z_{q - 1}, Z_qR_{\Psi_{\pi|_{V}}\left((Z_1, \dots, Z_n)|_{V}\right)}, Z_{q + 1}, \dots, Z_n\right)|_{W}\right) &\text{if } \chi(q) = r
\end{cases}.
\end{align*}
\end{enumerate}
\end{defn}

Given an operator-valued bi-multiplicative function, conditions $(1)$ to $(4)$ above are reduction properties which allows one to move $\B$-operators around and, more importantly, to compute the values on arbitrary bi-non-crossing partitions based on its values on full non-crossing partitions.

Finally, the two most important operator-valued bi-multiplicative functions in the theory, called operator-valued bi-free moment and cumulant functions, are defined as follows.

\begin{defn}\label{E-pi}
Let $(\A, \mathbb{E}, \varepsilon)$ be a $\mathcal{B}$-$\mathcal{B}$-non-commutative probability space. For $\chi: \{1, \dots, n\} \to \{\ell, r\}$, $\pi \in \B\N\C(\chi)$, and $Z_1, \dots, Z_n \in \A$, define $\bE_\pi(Z_1, \dots, Z_n) \in \B$ recursively as follows: Let $V$ be the block of $\pi$ that terminates closest to the bottom, so $\min(V)$ is largest among all blocks of $\pi$.
\begin{enumerate}[$\qquad(1)$]
\item If $\pi$ contains exactly one block (that is, $\pi = 1_\chi$), define $\bE_{1_\chi}(Z_1, \dots, Z_n) = \bE(Z_1\cdots Z_n)$.

\item If $V = \{k + 1, \dots, n\}$ for some $k \in \{1, \dots, n - 1\}$ (so $\min(V)$ is not adjacent to any spine of $\pi$), define
\[\bE_\pi(Z_1, \dots, Z_n) = \begin{cases}
\bE_{\pi|_{V^\complement}}(Z_1, \dots, Z_kL_{\bE_{\pi|_V}(Z_{k + 1}, \dots, Z_n)}) &\text{if } \chi(\min(V)) = \ell\\
\bE_{\pi|_{V^\complement}}(Z_1, \dots, Z_kR_{\bE_{\pi|_V}(Z_{k + 1}, \dots, Z_n)}) &\text{if } \chi(\min(V)) = r\\
\end{cases}.\]

\item Otherwise, $\min(V)$ is adjacent to a spine. Let $W$ denote the block of $\pi$ corresponding to the spine adjacent to $\min(V)$ and let $k$ be the smallest element of $W$ that is larger than $\min(V)$.  Define
\[\bE_\pi(Z_1, \dots, Z_n) = \begin{cases}
\bE_{\pi|_{V^\complement}}((Z_1, \dots, Z_{k - 1}, L_{\bE_{\pi|_V}((Z_1, \dots, Z_n)|_V)}Z_k, Z_{k + 1}, \dots, Z_n)|_{V^\complement}) &\text{if } \chi(\min(V)) = \ell\\
\bE_{\pi|_{V^\complement}}((Z_1, \dots, Z_{k - 1}, R_{\bE_{\pi|_V}((Z_1, \dots, Z_n)|_V)}Z_k, Z_{k + 1}, \dots, Z_n)|_{V^\complement}) &\text{if } \chi(\min(V)) = r\\
\end{cases}.\]
\end{enumerate}
\end{defn}

\begin{defn}\label{MomentCumulant}
Let $(\A, \mathbb{E}, \varepsilon)$ be a $\mathcal{B}$-$\mathcal{B}$-non-commutative probability space. The \textit{operator-valued bi-free moment and cumulant functions} on $\A$ are
\[\E, \kappa: \bigcup_{n \geq 1}\bigcup_{\chi: \{1, \dots, n\} \to \{\ell, r\}}\B\N\C(\chi) \times \A_{\chi(1)} \times \cdots \times \A_{\chi(n)} \to \B\]
defined by
\[\E_\pi(Z_1, \dots, Z_n) = \bE_\pi(Z_1, \dots, Z_n) \qand \kappa_\pi(Z_1, \dots, Z_n) = \sum_{\substack{\sigma \in \B\N\C(\chi)\\\sigma \leq \pi}}\E_\sigma(Z_1, \dots, Z_n)\mu_{\B\N\C}(\sigma, \pi)\]
for all $\chi: \{1, \dots, n\} \to \{\ell, r\}$, $\pi \in \B\N\C(\chi)$, and $Z_k \in \A_{\chi(k)}$.
\end{defn}

A substantial amount of effort was taken in \cite{CNS2015-2}*{Sections 5 and 6} to show that both $\E$ and $\kappa$ are operator-valued bi-multiplicative.

\section{Conditionally bi-free families with amalgamation}\label{sec:c-bi-free-defn}

In this section, we develop the structures to discuss conditionally bi-free independence with amalgamation.  To begin, we need an analogue of a two-state vector space with a specified state-vector.

\begin{defn}
A \textit{$\mathcal{B}$-$\mathcal{B}$-bimodule with a pair of specified $(\B, \D)$-valued states} is a quadruple $(\mathcal{X}, \mathcal{X}^\circ, \mathfrak{p}, \mathfrak{q})$, where $\mathcal{B}$ and $\mathcal{D}$ are unital algebras such that $1 := 1_\D \in \mathcal{B} \subset \mathcal{D}$, $\mathcal{X}$ is a direct sum of $\mathcal{B}$-$\mathcal{B}$-bimodules $\mathcal{X} = \mathcal{B} \oplus \mathcal{X}^\circ$, $\mathfrak{p}: \mathcal{X} \to \mathcal{B}$ is the linear map $\mathfrak{p}(b \oplus \eta) = b$, and $\mathfrak{q}: \mathcal{X} \to \mathcal{D}$ is a linear $\mathcal{B}$-$\mathcal{B}$-bimodule map such that $\mathfrak{q}(1 \oplus 0) = 1$.
\end{defn}

Given a $\mathcal{B}$-$\mathcal{B}$-bimodule with a pair of specified $(\B, \D)$-valued states $(\mathcal{X}, \mathcal{X}^\circ, \mathfrak{p}, \mathfrak{q})$, we have
\[\mathfrak{p}(b_1\cdot x\cdot b_2) = b_1\mathfrak{p}(x)b_2 \qand \mathfrak{q}(b_1\cdot x\cdot b_2) = b_1\mathfrak{q}(x)b_2\]
for all $b_1, b_2 \in \B$ and $x \in \X$. Moreover, let $\L(\X)$ denote the set of linear operators on $\X$, and recall from \cite{CNS2015-2}*{Definition 3.1.3} that the operators $L_b, R_b \in \L(\X)$ are defined by
\[L_b(x) = b\cdot x \qand R_b(x) = x\cdot b\]
for all $b \in \B$ and $x \in \X$. In addition, the \textit{left and right algebras} of $\L(\X)$ are the unital subalgebras $\L_\ell(\X)$ and $\L_r(X)$ defined by
\[\L_\ell(X) = \{Z \in \L(\X) \, \mid \, ZR_b = R_bZ\text{ for all }b \in \B\}\]
and
\[\L_r(\X) = \{Z \in \L(\X) \, \mid \, ZL_b = L_bZ\text{ for all }b \in \B\},\]
respectively.

As we are interested in conditionally bi-free independence with amalgamation, we need two expectations on $\L(\X)$, one onto $\B$ and one to $\D$.

\begin{defn}
Given a $\B$-$\B$-bimodule with a pair of specified $(\B, \D)$-valued states $(\mathcal{X}, \mathcal{X}^\circ, \mathfrak{p}, \mathfrak{q})$, define the unital linear maps $\mathbb{E}_{\L(\X)}: \L(\X) \to \B$ and $\mathbb{F}_{\L(\X)}: \L(\X) \to \mathcal{D}$ by
\[\mathbb{E}_{\L(\X)}(Z) = \mathfrak{p}(Z(1 \oplus 0))\qand\mathbb{F}_{\L(\X)}(Z) = \mathfrak{q}(Z(1 \oplus 0))\]
for all $Z \in \L(\X)$. We call $\mathbb{E}_{\L(\X)}$ and $\mathbb{F}_{\L(\X)}$ the \textit{expectations} of $\L(\X)$ to $\B$ and $\mathcal{D}$, respectively.
\end{defn}

There are specific properties of these expectations we wish to model.

\begin{prop}\label{Expectations}
Let $(\mathcal{X}, \mathcal{X}^\circ, \mathfrak{p}, \mathfrak{q})$ be a $\B$-$\B$-bimodule with a pair of specified $(\B, \D)$-valued states. We have
\[\mathbb{E}_{\L(\X)}(L_{b_1}R_{b_2}Z) = b_1\mathbb{E}_{\L(\X)}(Z)b_2,\quad\mathbb{E}_{\L(\X)}(ZL_b) = \mathbb{E}_{\L(\X)}(ZR_b)\]
and
\[
\mathbb{F}_{\L(\X)}(L_{b_1}R_{b_2}Z) = b_1\mathbb{F}_{\L(\X)}(Z)b_2,\quad\mathbb{F}_{\L(\X)}(ZL_b) = \mathbb{F}_{\L(\X)}(ZR_b)
\]
for all $b_1, b_2, b \in \mathcal{B}$ and $Z \in \L(\X)$.
\end{prop}

\begin{proof}
The results regarding $\bE_{\L(\X)}$ were shown in \cite{CNS2015-2}*{Proposition 3.1.6}. Moreover, it is immediate that $\mathbb{F}_{\L(\X)}(ZL_b) = \mathbb{F}_{\L(\X)}(ZR_b)$ for all $b \in \B$ and $Z \in \L(\X)$ as $L_b(1 \oplus 0) = R_b(1 \oplus 0)$. Finally, since $\mathfrak{q}$ is a linear $\B$-$\B$-bimodule map, we have
\[\mathbb{F}_{\L(\X)}(L_{b_1}R_{b_2}Z) = \mathfrak{q}(L_{b_1}R_{b_2}Z(1 \oplus 0)) = b_1\mathfrak{q}(Z(1 \oplus 0))b_2 = b_1\bF_{\L(\X)}(Z)b_2\]
for all $b_1, b_2 \in \B$ and $Z \in \L(\X)$.
\end{proof}

Given the above definition and proposition, we extend the notion of a two-state non-commutative probability space $(\A, \varphi, \psi)$ to the operator-valued setting as follows.  Note this is also a natural extension of the notion of a $\B$-$\B$-non-commutative probability space $(\A, \bE, \varepsilon)$ from \cite{CNS2015-2}*{Definition 3.2.1} to the two-state setting.

\begin{defn}\label{BBncpsBD}
A \textit{$\B$-$\B$-non-commutative probability space with a pair of $(\B, \D)$-valued expectations} is a quadruple $(\A, \mathbb{E}, \mathbb{F}, \varepsilon)$, where $\A$, $\mathcal{B}$, and $\mathcal{D}$ are unital algebras such that $1 := 1_\D \in \mathcal{B} \subset \mathcal{D}$, $\varepsilon: \mathcal{B} \otimes \mathcal{B}^{\mathrm{op}} \to \A$ is a unital homomorphism such that $\varepsilon|_{\mathcal{B} \otimes 1}$ and $\varepsilon|_{1 \otimes \mathcal{B}^{\mathrm{op}}}$ are injective, and $\mathbb{E}: \A \to \mathcal{B}$ and $\mathbb{F}: \A \to \mathcal{D}$ are unital linear maps such that
\[\mathbb{E}(\varepsilon(b_1 \otimes b_2)Z) = b_1\mathbb{E}(Z)b_2,\quad\mathbb{E}(Z\varepsilon(b \otimes 1)) = \mathbb{E}(Z\varepsilon(1 \otimes b))\]
and
\[\mathbb{F}(\varepsilon(b_1 \otimes b_2)Z) = b_1\mathbb{F}(Z)b_2,\quad\mathbb{F}(Z\varepsilon(b \otimes 1)) = \mathbb{F}(Z\varepsilon(1 \otimes b))\]
for all $b_1, b_2, b \in \mathcal{B}$ and $Z \in \A$. Moreover, the unital subalgebras $\A_\ell$ and $\A_r$ of $\A$ defined by
\[\A_\ell = \{Z \in \A \, \mid \, Z\varepsilon(1 \otimes b) = \varepsilon(1 \otimes b)Z\text{ for all }b \in \mathcal{B}\}\]
and
\[\A_r = \{Z \in \A \, \mid \, Z\varepsilon(b \otimes 1) = \varepsilon(b \otimes 1)Z\text{ for all }b \in \mathcal{B}\}\]
will be called the \textit{left and right algebras} of $\A$ respectively.
\end{defn}

As with the bi-free case (see \cite{CNS2015-2}*{Remark 3.2.2}), if $(\X, \X^\circ, \mathfrak{p}, \mathfrak{q})$ is a $\B$-$\B$-bimodule with a pair of specified $(\B, \D)$-valued states, then we see via Proposition \ref{Expectations} that $(\L(\X), \bE_{\L(\X)}, \bF_{\L(\X)}, \varepsilon)$ is a $\B$-$\B$-non-commutative probability space with a pair of $(\B, \D)$-valued expectations where $\varepsilon: \B \otimes \B^{\mathrm{op}} \to \L(\X)$ is defined by $\varepsilon(b_1 \otimes b_2) = L_{b_1}R_{b_2}$.  Moreover, the following result demonstrates that any $\B$-$\B$-non-commutative probability space with a pair of $(\B, \D)$-valued expectations can be represented as linear operators on some $(\X, \X^\circ, \mathfrak{p}, \mathfrak{q})$. Hence Definition \ref{BBncpsBD} is the natural extension of \cite{CNS2015-2}*{Definition 3.2.1}. As such, we will write $L_b$ and $R_b$ instead of $\varepsilon(b \otimes 1)$ and $\varepsilon(1 \otimes b)$ and refer to these as left and right $\B$-operators, respectively.

\begin{thm}\label{Embedding}
If $(\A, \mathbb{E}_\A, \mathbb{F}_\A, \varepsilon)$ is a $\mathcal{B}$-$\mathcal{B}$-non-commutative probability space with a pair of $(\B, \D)$-valued expectations, then there exist a $\mathcal{B}$-$\mathcal{B}$-bimodule with a pair of specified $(\B, \D)$-valued states $(\mathcal{X}, \mathcal{X}^\circ, \mathfrak{p}, \mathfrak{q})$ and a unital homomorphism $\theta: \A \to \L(\X)$ such that
\begin{gather*}
\theta(L_{b_1}R_{b_2}) = L_{b_1}R_{b_2},\quad\theta(\A_\ell) \subset \L_\ell(\X),\quad\theta(\A_r) \subset \L_r(\X), \\
\mathbb{E}_{\L(\X)}(\theta(Z)) = \mathbb{E}_\A(Z),\qand\mathbb{F}_{\L(\X)}(\theta(Z)) = \mathbb{F}_\A(Z)
\end{gather*}
for all $b_1, b_2 \in \B$ and $Z \in \A$.
\end{thm}

\begin{proof}
As shown in the proof of \cite{CNS2015-2}*{Theorem 3.2.4}, consider $\X = \B \oplus \Y$ as a vector space over $\mathbb{C}$ where
\[\Y = \ker(\mathbb{E}_\A)/\mathrm{span}\{ZL_b - ZR_b \, \mid \, Z \in \A, b \in \B\}.\]
Define $\theta: \A \to \L(\X)$ by
\[\theta(Z)(b) = \bE_\A(ZL_b) \oplus \pi(ZL_b - L_{\bE_\A(ZL_b)}),\quad b \in \B,\]
and
\[\theta(Z)(\pi(Y)) = \bE_\A(ZY) \oplus \pi(ZY - L_{\bE_\A(ZY)}),\quad Y \in \ker(\bE_\A),\]
where $\pi: \ker(\bE_\A) \to \Y$ denotes the canonical quotient map.  It was shown in \cite{CNS2015-2}*{Theorem 3.2.4} that $\theta$ is a unital homomorphism and $\X$ is a $\B$-$\B$-bimodule via
\[b\cdot\xi = \theta(L_b)(\xi)\qand\xi\cdot b = \theta(R_b)(\xi)\]
for all $b \in \B$ and $\xi \in \X$.
Thus we can define a specified $\B$-valued state $\mathfrak{p}$ on $\X$ by $\mathfrak{p}(b \oplus \pi(Y)) = b$ for all $b \in \B$ and $\pi(Y) \in \Y$.  Using this specified $\B$-valued state, we obtain that  $\theta(\A_\ell) \subset \L_\ell(\X)$, $\theta(\A_r) \subset \L_r(\X)$, and $\mathbb{E}_{\L(\X)}(\theta(Z)) = \mathbb{E}_\A(Z)$. 

On the other hand, since $\bF_\A(ZL_b - ZR_b) = 0$ for all $Z \in \A$ and $b \in \B$, there exists a unique linear map $\widetilde{\mathfrak{q}}: \Y \to \D$ such that $\bF_\A|_{\ker(\bE_\A)} = \widetilde{\mathfrak{q}} \circ \pi$. Let $\mathfrak{q}: \X \to \D$ be the linear map defined by
\[\mathfrak{q}(b \oplus \pi(Y)) = b + \widetilde{\mathfrak{q}} \circ \pi(Y),\quad b \in \B,\quad\pi(Y) \in \Y.\]
Then $\mathfrak{q}(1 \oplus 0) = 1$ and
\begin{align*}
\mathfrak{q}(b_1\cdot(b \oplus \pi(Y))\cdot b_2) &= \mathfrak{q}(\theta(L_{b_1})\theta(R_{b_2})(b \oplus \pi(Y)))\\
&= \mathfrak{q}(\theta(L_{b_1})(bb_2 \oplus \pi(R_{b_2}Y)))\\
&= \mathfrak{q}(b_1bb_2 \oplus \pi(L_{b_2}R_{b_2}Y))\\
&= b_1bb_2 + \bF_\A(L_{b_1}R_{b_2}Y)\\
&= b_1(b + \bF_\A(Y))b_2\\
&= b_1\mathfrak{q}(b \oplus \pi(Y))b_2
\end{align*}
for all $b_1, b_2, b \in \B$ and $\pi(Y) \in \Y$. Therefore, the quadruple $(\X, \Y, \mathfrak{p}, \mathfrak{q})$ is a $\B$-$\B$-bimodule with a pair of specified $(\B, \D)$-valued states. Finally, we have
\[\bF_{\L(\X)}(\theta(Z)) = \mathfrak{q}(\theta(Z)(1 \oplus 0)) = \mathfrak{q}(\bE_\A(Z) \oplus \pi(Z - L_{\bE_\A(Z)})) = \bE_\A(Z) + \bF_\A(Z - L_{\bE_\A(Z)}) = \bF_\A(Z)\]
for all $Z \in \A$.
\end{proof}

The next step is to extend the construction of the free product with amalgamation over $\B$ of a family $\{(\X_k, \X_k^\circ, \mathfrak{p}_k)\}_{k \in K}$ of $\B$-$\B$-bimodules with specified $\B$-valued states (see \cite{CNS2015-2}*{Construction 3.1.7}) to the current framework.

\begin{cons}\label{Construction}
Let $\{(\X_k, \X_k^\circ, \mathfrak{p}_k, \mathfrak{q}_k)\}_{k \in K}$ be a family of $\B$-$\B$-bimodules with pairs of specified $(\B, \D)$-valued states. The \textit{c-free product of $\{(\X_k, \X_k^\circ, \mathfrak{p}_k, \mathfrak{q}_k)\}_{k \in K}$ with amalgamation over $(\B, \D)$} is defined to be the $\B$-$\B$-bimodule with a pair of specified $(\B, \D)$-valued states $(\X, \X^\circ, \mathfrak{p}, \mathfrak{q})$, where $\X = \B \oplus \X^\circ$, $\X^\circ$ is the $\B$-$\B$-bimodule
\[\X^\circ = \bigoplus_{n \geq 1}\left(\bigoplus_{k_1 \neq \cdots \neq k_n}\X_{k_1}^\circ \otimes_{\B} \cdots \otimes_{\B} \X_{k_n}^\circ\right)\]
with left and right actions of $\B$ on $\X^\circ$ defined by
\[b\cdot(x_1 \otimes \cdots \otimes x_n) = (b \cdot x_1) \otimes x_2 \otimes \cdots \otimes x_n \qand (x_1 \otimes \cdots \otimes x_n)\cdot b = x_1 \otimes \cdots \otimes x_{n - 1} \otimes (x_n \cdot b),\]
respectively, $\mathfrak{p}: \X \to \B$ is the linear map $\mathfrak{p}(b \oplus \eta) = b$, and $\mathfrak{q}: \X \to \D$ is the linear $\B$-$\B$-bimodule map such that $\mathfrak{q}(1 \oplus 0) = 1$ and
\[\mathfrak{q}(x_1 \otimes \cdots \otimes x_n) = \mathfrak{q}_{k_1}(x_1)\cdots\mathfrak{q}_{k_n}(x_n)\]
for $x_1 \otimes \cdots \otimes x_n \in \X_{k_1}^\circ \otimes_{\B} \cdots \otimes_{\B} \X_{k_n}^\circ$  (note $\mathfrak{q}$ is well-defined as each $\mathfrak{q}_k$ is a linear $\B$-$\B$-bimodule map).

For every $k \in K$, let
\[V_k: \X \to \X_k \otimes_{\B} \left(\B \oplus \bigoplus_{n \geq 1}\left(\bigoplus_{\substack{k_1 \neq \cdots \neq k_n\\k_1 \neq k}}\X_{k_1}^\circ \otimes_{\B} \cdots \otimes_{\B} \X_{k_n}^\circ\right)\right)\]
be the standard $\B$-$\B$-bimodule isomorphism, and define the left representation $\lambda_k: \L(\X_k) \to \L(\X)$ by
\[\lambda_k(Z) = V_k^{-1}(Z \otimes I)V_k\]
for $Z \in \L(\X_k)$. Similarly, let
\[W_k: \X \to \left(\B \oplus \bigoplus_{n \geq 1}\left(\bigoplus_{\substack{k_1 \neq \cdots \neq k_n\\k_n \neq k}}\X_{k_1}^\circ \otimes_{\B} \cdots \otimes_{\B} \X_{k_n}^\circ\right)\right) \otimes_{\B} \X_k\]
be the standard $\B$-$\B$-bimodule isomorphism, and define the right representation $\rho_k: \L(\X_k) \to \L(\X)$ by
\[\rho_k(Z) = W_k^{-1}(I \otimes Z)W_k\]
for $Z \in \L(\X_k)$. For the exact formulae of how $\lambda_k(Z)$ and $\rho_k(Z)$ act on $\X$, we refer to \cite{CNS2015-2}*{Construction 3.1.7}. Note also that
\[\bE_{\L(\X)}(\lambda_k(Z)) = \bE_{\L(\X)}(\rho_k(Z)) = \bE_{\L(\X_k)}(Z) \qand \bF_{\L(\X)}(\lambda_k(Z)) = \bF_{\L(\X)}(\rho_k(Z)) = \bF_{\L(\X_k)}(Z)\]
for all $Z \in \L(\X_k)$.
\end{cons}

\begin{rem}
It is clear that that all of the above discussions hold if $\B = \D$. However, the more general setting that $\B \subset \D$ is desired due to a result of Boca \cite{B1991}. Indeed, suppose $\{\A_k\}_{k \in K}$ is a family of unital $C^*$-algebras containing $\B$ as a common $C^*$-subalgebra with $1_{\A_k} \in \B$, $\D$ is a unital $C^*$-algebra with $1_\D \in \B \subset \D$, and each $\A_k$ is endowed with two positive conditional expectations $\Psi_k: \A_k \to \B$ and $\Phi_k: \A_k \to \D$ such that $\A_k = \B \oplus \A_k^\circ$, where $\A_k^\circ = \ker(\Psi_k)$, as a direct sum of $\B$-$\B$-bimodules.

Let $\A = (*_\B)_{k \in K}\A_k$ be the free product of $\{\A_k\}_{k \in K}$ with amalgamation over $\B$ (which can be identified as
\[\A = \B \oplus \bigoplus_{n \geq 1}\left(\bigoplus_{k_1 \neq \cdots \neq k_n}\A_{k_1}^\circ \otimes_{\B} \cdots \otimes_{\B} \A_{k_n}^\circ\right)\]
as $\B$-$\B$-bimodules), let $\Psi = (*_{\B})_{k \in K}\Psi_k$ be the amalgamated free product of $\{\Psi_k\}_{k \in K}$, and let $\Phi: \A \to \D$ be the unital linear $\B$-$\B$-bimodule map defined by
\[\Phi(a_1 \otimes \cdots \otimes a_n) = \Phi_{k_1}(a_1)\cdots\Phi_{k_n}(a_n)\]
for $a_1 \otimes \cdots \otimes a_n \in \A_{k_1}^\circ \otimes_{\B} \cdots \otimes_{\B} \A_{k_n}^\circ$. It is well-known that $\Psi$ is positive (see, e.g., \cite{S1998}*{Theorem 3.5.6}). On the other hand, it follows from \cite{B1991}*{Theorem 3.1} that $\Phi$ is also positive, which is the main motivation for our setting.
\end{rem}

With Definition \ref{BBncpsBD} and Construction \ref{Construction} complete, we can define the notion of conditionally bi-free independence with amalgamation as follows.

\begin{defn}
\label{defn:op-c-bi-free-definition}
Let $(\A, \mathbb{E}, \mathbb{F}, \varepsilon)$ be a $\B$-$\B$-non-commutative probability space with a pair of $(\B, \D)$-valued expectations. A \textit{pair of $\B$-algebras} in $(\A, \mathbb{E}, \mathbb{F}, \varepsilon)$ is a pair $(\C_\ell, \C_r)$ of unital subalgebras of $\A$ such that
\[\varepsilon(\B \otimes 1) \subset \C_\ell \subset \A_\ell \qand \varepsilon(1 \otimes \B^{\mathrm{op}}) \subset \C_r \subset \A_r.\]

A family $\{(\A_{k, \ell}, \A_{k, r})\}_{k \in K}$ of pairs of $\B$-algebras in $(\A, \mathbb{E}, \mathbb{F}, \varepsilon)$ is said to be \textit{conditionally bi-free with amalgamation over $(\B, \D)$} (or \textit{c-bi-free over $(\B, \D)$} for short) if there is a family of $\B$-$\B$-bimodules with pairs of specified $(\B, \D)$-valued states $\{(\mathcal{X}_k, \mathcal{X}^\circ_k, \mathfrak{p}_k, \mathfrak{q}_k)\}_{k \in K}$ and unital homomorphisms
\[\ell_k: \A_{k, \ell} \to \L_\ell(\mathcal{X}_k)\qand r_k: \A_{k, r} \to \L_r(\mathcal{X}_k)\]
such that the joint distribution of $\{(\A_{k, \ell}, \A_{k, r})\}_{k \in K}$ with respect to $(\mathbb{E}, \mathbb{F})$ is equal to the joint distribution of the family
\[\{(\lambda_k \circ \ell_k(\A_{k, \ell}), \rho_k \circ r_k(\A_{k, r}))\}_{k \in K}\]
in $\L(\X)$ with respect to $(\mathbb{E}_{\L(\X)}, \mathbb{F}_{\L(\X)})$, where $(\X, \X^\circ, \mathfrak{p}, \mathfrak{q}) = (*_\B)_{k \in K}(\mathcal{X}_k, \mathcal{X}^\circ_k, \mathfrak{p}_k, \mathfrak{q}_k)$.
\end{defn}

It will be an immediate consequence of Theorem \ref{MomentFormulae} below that the above definition does not depend on a specific choice of representations. Moreover, it follows immediately from the definition that if $\{(\A_{k, \ell}, \A_{k, r})\}_{k \in K}$ is c-bi-free over $(\B, \D)$, then the family $\{\A_{k, \ell}\}_{k \in K}$ (respectively $\{\A_{k, r}\}_{k \in K}$) of left $\B$-algebras (respectively right $\B$-algebras) is c-free over $(\B, \D)$.

\section{Operator-valued conditionally bi-free pairs of functions}\label{sec:pairs-of-fns}

In order to study operator-valued conditional bi-free independence we must extend the notion of operator-valued bi-multiplicative functions to pairs of functions.

\subsection{Operator-valued conditionally bi-multiplicative pairs of functions}

We begin with an observation, which will be useful later.

\begin{rem}\label{Decomposition}
If $\pi \in \B\N\C(\chi)$ is a bi-non-crossing partition for some $\chi: \{1, \dots, n\} \to \{\ell, r\}$, then there exists a unique partition $V_1, \dots, V_m$ of $\{1,\ldots, n\}$ into $\chi$-intervals such that each $V_k$ a union of blocks of $\pi$ and such that $\min_{\prec_\chi}(V_k)$ and $\max_{\prec_\chi}(V_k)$ are in the same block of $\pi$ for each $k \in \{1, \dots, m\}$.   Furthermore, by reordering if necessary, we may assume $\max_{\prec_\chi}(V_k) \prec_\chi \min_{\prec_\chi} (V_{k+1})$ for all $k$. For example, if $\pi$ has the following bi-non-crossing diagram
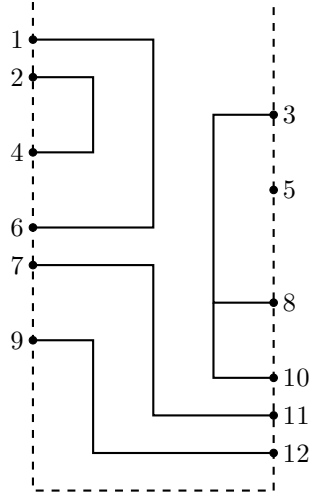
\begin{align*}
\begin{tikzpicture}[baseline]
\node[left] at (0,5.5) {1};
\draw[black,fill=black] (0,5.5) circle (0.05);
\node[left] at (0,5) {2};
\draw[black,fill=black] (0,5) circle (0.05);
\node[right] at (3.2,4.5) {3};
\draw[black,fill=black] (3.2,4.5) circle (0.05);
\node[left] at (0,4) {4};
\draw[black,fill=black] (0,4) circle (0.05);
\node[right] at (3.2,3.5) {5};
\draw[black,fill=black] (3.2,3.5) circle (0.05);
\node[left] at (0,3) {6};
\draw[black,fill=black] (0,3) circle (0.05);
\node[left] at (0,2.5) {7};
\draw[black,fill=black] (0,2.5) circle (0.05);
\node[right] at (3.2,2) {8};
\draw[black,fill=black] (3.2,2) circle (0.05);
\node[left] at (0,1.5) {9};
\draw[black,fill=black] (0,1.5) circle (0.05);
\node[right] at (3.2,1) {10};
\draw[black,fill=black] (3.2,1) circle (0.05);
\node[right] at (3.2,0.5) {11};
\draw[black,fill=black] (3.2,0.5) circle (0.05);
\node[right] at (3.2,0) {12};
\draw[black,fill=black] (3.2,0) circle (0.05);
\draw[thick, black] (0,5.5) -- (1.6,5.5) -- (1.6,3) -- (0,3);
\draw[thick, black] (0,5) -- (0.8,5) -- (0.8,4) -- (0,4);
\draw[thick, black] (3.2,4.5) -- (2.4,4.5) -- (2.4,2) -- (3.2,2);
\draw[thick, black] (2.4,2) -- (2.4,1) -- (3.2,1);
\draw[thick, black] (0,2.5) -- (1.6,2.5) -- (1.6,0.5) -- (3.2,0.5);
\draw[thick, black] (0,1.5) -- (0.8,1.5) -- (0.8,0) -- (3.2,0);
\draw[thick, dashed, black] (0,6) -- (0,-.5) -- (3.2, -.50) -- (3.2,6);
\end{tikzpicture}
\end{align*}
then $V_1 = \{\{1, 6\}, \{2, 4\}\}$, $V_2 = \{\{7, 11\}, \{9, 12\}\}$, and $V_3 = \{\{3, 8, 10\}, \{5\}\}$ where $\min_{\prec_\chi}(V_1)=1$, $\max_{\prec_\chi}(V_1) = 6$, $\min_{\prec_\chi}(V_2)= 7$, $\max_{\prec_\chi}(V_2)=11$, $\min_{\prec_\chi}(V_3) = 10$, and $\max_{\prec_\chi}(V_3)=3$. Note that the blocks $V_k' \subset V_k$ containing $\min_{\prec_\chi}(V_k)$ and $\max_{\prec_\chi}(V_k)$ are the exterior blocks of $\pi$.
\end{rem}

\begin{defn}\label{CondBiMulti}
Let $(\A, \mathbb{E}, \mathbb{F}, \varepsilon)$ be a $\mathcal{B}$-$\mathcal{B}$-non-commutative probability space with a pair of $(\B, \D)$-valued expectations, and let
\[\Psi: \bigcup_{n \geq 1}\bigcup_{\chi: \{1, \dots, n\} \to \{\ell, r\}}\B\N\C(\chi) \times \A_{\chi(1)} \times \cdots \times \A_{\chi(n)} \to \B\]
and
\[\Phi: \bigcup_{n \geq 1}\bigcup_{\chi: \{1, \dots, n\} \to \{\ell, r\}}\B\N\C(\chi) \times \A_{\chi(1)} \times \cdots \times \A_{\chi(n)} \to \D\]
be a pair of functions that are linear in each $\A_{\chi(k)}$. It is said that $(\Psi, \Phi)$ is an \textit{operator-valued conditionally bi-multiplicative pair} if for every $\chi: \{1, \dots, n\} \to \{\ell, r\}$, $Z_k \in \A_{\chi(k)}$, $b \in \B$, and $\pi \in \B\N\C(\chi)$, $\Psi$ satisfies conditions $(1)$ to $(4)$ of Definition \ref{BiMulti} (i.e., $\Psi$ is operator-valued bi-multiplicative), and $\Phi$ satisfies conditions $(1)$ to $(3)$ of Definition \ref{BiMulti} and the following modification of condition $(4)$: Under the same notation with the additional assumption that $\min_{\prec_\chi}(\{1, \dots, n\})$ and $\max_{\prec_\chi}(\{1, \dots, n\})$ are in the same block of $\pi$, we have
\begin{align*}
\Phi_\pi(Z_1, \dots, Z_n) &= \begin{cases}
\Phi_{\pi|_{W}}\left(\left(Z_1, \dots, Z_{p - 1}, Z_pL_{\Psi_{\pi|_{V}}\left((Z_1, \dots, Z_n)|_{V}\right)}, Z_{p + 1}, \dots, Z_n\right)|_{W}\right) &\text{if } \chi(p) = \ell\\
\Phi_{\pi|_{W}}\left(\left(Z_1, \dots, Z_{p - 1}, R_{\Psi_{\pi|_{V}}\left((Z_1, \dots, Z_n)|_{V}\right)}Z_p, Z_{p + 1}, \dots, Z_n\right)|_{W}\right) &\text{if } \chi(p) = r
\end{cases}\\
&= \begin{cases}
\Phi_{\pi|_{W}}\left(\left(Z_1, \dots, Z_{q - 1}, L_{\Psi_{\pi|_{V}}\left((Z_1, \dots, Z_n)|_{V}\right)}Z_q, Z_{q + 1}, \dots, Z_n\right)|_{W}\right) &\text{if } \chi(q) = \ell\\
\Phi_{\pi|_{W}}\left(\left(Z_1, \dots, Z_{q - 1}, Z_qR_{\Psi_{\pi|_{V}}\left((Z_1, \dots, Z_n)|_{V}\right)}, Z_{q + 1}, \dots, Z_n\right)|_{W}\right) &\text{if } \chi(q) = r
\end{cases}.
\end{align*}
\end{defn}

Note the additional assumption that $\min_{\prec_\chi}(\{1, \dots, n\})$ and $\max_{\prec_\chi}(\{1, \dots, n\})$ are in the same block of $\pi$ guarantees that $W$ contains an exterior block of $\pi$ and $V$ is a union of interior blocks of $\pi$.

\begin{exam}
As with operator-valued bi-multiplicative functions, one may reduce $\Phi_\pi(Z_1, \dots, Z_n)$ to an expression involving $\Psi_{1_\chi}$ and $\Phi_{1_\chi}$ for various $\chi: \{1, \dots, m\} \to \{\ell, r\}$. For example, if $\pi$ is the bi-non-crossing partition from Remark \ref{Decomposition} and $Z_k \in \A_{\chi(k)}$, then
\[\Phi_\pi(Z_1, \dots, Z_{12}) = \Phi_{\pi|_{V_1}}(Z_1, Z_2, Z_4, Z_6)\Phi_{\pi|_{V_2}}(Z_7, Z_9, Z_{11}, Z_{12})\Phi_{\pi|_{V_3}}(Z_3, Z_5, Z_8, Z_{10})\]
by condition $(3)$ of Definition \ref{BiMulti}, which can be further reduced to
\[\Phi_{1_{2, 0}}\left(Z_1L_{\Psi_{1_{2, 0}}(Z_2, Z_4)}, Z_6\right)\Phi_{1_{1, 1}}\left(Z_7L_{\Psi_{1_{1, 1}}(Z_9, Z_{12})}, Z_{11}\right)\Phi_{1_{0, 3}}\left(Z_3, R_{\Psi_{1_{0, 1}}(Z_5)}Z_8, Z_{10}\right)\]
by the modified condition $(4)$ of Definition \ref{CondBiMulti}.
\end{exam}

\subsection{Operator-valued conditionally bi-free moment pairs}

In this subsection, we define the operator-valued conditionally bi-free moment pair $(\E, \F)$ and show that it is operator-valued conditionally bi-multiplicative.

\begin{defn}\label{CBFMomentPair}
Let $(\A, \mathbb{E}, \mathbb{F}, \varepsilon)$ be a $\mathcal{B}$-$\mathcal{B}$-non-commutative probability space with a pair of $(\B, \D)$-valued expectations. The \textit{operator-valued conditionally bi-free moment pair} on $\A$ is the pair of functions
\[\E: \bigcup_{n \geq 1}\bigcup_{\chi: \{1, \dots, n\} \to \{\ell, r\}}\B\N\C(\chi) \times \A_{\chi(1)} \times \cdots \times \A_{\chi(n)} \to \B\]
and
\[\F: \bigcup_{n \geq 1}\bigcup_{\chi: \{1, \dots, n\} \to \{\ell, r\}}\B\N\C(\chi) \times \A_{\chi(1)} \times \cdots \times \A_{\chi(n)} \to \D\]
where $\E$ is the operator-valued bi-free moment function on $\A$ and $\F_\pi(Z_1, \dots, Z_n)$ for $\chi: \{1, \dots, n\} \to \{\ell, r\}$, $\pi \in \B\N\C(\chi)$, and $Z_k \in \A_{\chi(k)}$ is defined as follows.
\begin{enumerate}[$\qquad(1)$]
\item If $\pi$ contains exactly one block (that is, $\pi = 1_\chi$), define $\F_{1_\chi}(Z_1, \dots, Z_n) = \bF(Z_1\cdots Z_n)$.

\item If $V_1, \ldots, V_n$ are the blocks of $\pi$, each $V_k$ is a $\chi$-intervals (thus all exterior), and $\max_{\prec_\chi}(V_k) \prec_\chi \min_{\prec_\chi}(V_{k+1})$ for all $k$, define
\[\F_\pi(Z_1, \dots, Z_n) = \F_{\pi|_{V_1}}((Z_1, \dots, Z_n)|_{V_1})\cdots \F_{\pi|_{V_m}}((Z_1, \dots, Z_n)|_{V_m})\]
and apply step $(3)$ to each piece.

\item Apply a similar recursive process as in Definition \ref{E-pi} to the interior blocks of $\pi$ as follows:  Let $V$ be the interior block of $\pi$ that terminates closest to the bottom.  Then
\begin{itemize}
\item If $V = \{k + 1, \dots, n\}$ for some $k \in \{1, \dots, n - 1\}$, then $\min(V)$ is not adjacent to any spine of $\pi$ and define
\[\F_\pi(Z_1, \dots, Z_n) = \begin{cases}
\F_{\pi|_{V^\complement}}(Z_1, \dots, Z_kL_{\E_{\pi|_V}(Z_{k + 1}, \dots, Z_n)}) &\text{if } \chi(\min(V)) = \ell\\
\F_{\pi|_{V^\complement}}(Z_1, \dots, Z_kR_{\E_{\pi|_V}(Z_{k + 1}, \dots, Z_n)}) &\text{if } \chi(\min(V)) = r\\
\end{cases}.\]

\item Otherwise, $\min(V)$ is adjacent to a spine. Let $W$ denote the block of $\pi$ corresponding to the spine adjacent to $\min(V)$ and let $k$ be the smallest element of $W$ that is larger than $\min(V)$.  Define
\[\F_\pi(Z_1, \dots, Z_n) = \begin{cases}
\F_{\pi|_{V^\complement}}((Z_1, \dots, Z_{k - 1}, L_{\E_{\pi|_V}((Z_1, \dots, Z_n)|_V)}Z_k, Z_{k + 1}, \dots, Z_n)|_{V^\complement}) &\text{if } \chi(\min(V)) = \ell\\
\F_{\pi|_{V^\complement}}((Z_1, \dots, Z_{k - 1}, R_{\E_{\pi|_V}((Z_1, \dots, Z_n)|_V)}Z_k, Z_{k + 1}, \dots, Z_n)|_{V^\complement}) &\text{if } \chi(\min(V)) = r\\
\end{cases}.\]
\end{itemize}
\end{enumerate}
\end{defn}

\begin{exam}
Again, let $\pi$ be the bi-non-crossing partition from Remark \ref{Decomposition} and $Z_k \in \A_{\chi(k)}$.  Then
\[\F_\pi(Z_1, \dots, Z_{12}) = \bF\left(Z_1L_{\bE(Z_2Z_4)}Z_6\right)\bF\left(Z_7L_{\bE(Z_9Z_{12})}Z_{11}\right)\bF\left(Z_3R_{\bE(Z_5)}Z_8Z_{10}\right).\]
In general, the rule is `one uses $\E$ to reduce the interior blocks and then factors $\F_\pi$ according to the remaining exterior blocks.'
\end{exam}

\begin{thm}
Let $(\A, \mathbb{E}, \mathbb{F}, \varepsilon)$ be a $\mathcal{B}$-$\mathcal{B}$-non-commutative probability space with a pair of $(\B, \D)$-valued expectations. The operator-valued conditionally bi-free moment pair $(\E, \F)$ on $\A$ is operator-valued conditionally bi-multiplicative.
\end{thm}

\begin{proof}
The fact that the operator-valued bi-free moment function $\E$ on $\A$ is operator-valued bi-multiplicative is the main result of \cite{CNS2015-2}*{Section 5}. It is also clear that the function $\F$ satisfies conditions $(1)$, $(2)$, and $(3)$ of Definition \ref{BiMulti}.

To demonstrate the modified condition $(4)$ in Definition \ref{CondBiMulti}, the proof relies on the techniques observed in \cite{CNS2015-2}*{Subsection 5.3} which show that the function $\E$ satisfies condition $(4)$ of Definition \ref{BiMulti}.  In particular, we refer the reader to the proofs of \cite{CNS2015-2}*{Lemmata 5.3.1 to 5.3.4} for additional details in that which follows. Under the same assumptions and notation, first note that the special case of the assertion holds under the additional assumption of \cite{CNS2015-2}*{Lemma 5.3.1}; that is, there exists a block $W_0 \subset W$ of $\pi$ such that
\[p, q, \min_{\prec_\chi}(\{1, \dots, n\}), \max_{\prec_\chi}(\{1, \dots, n\}) \in W_0.\]
Indeed, suppose $\chi(p) = \ell$ (the other case is similar), and note that $W_0$ is the only exterior block of $\pi$. By the same arguments as in the proof of \cite{CNS2015-2}*{Lemma 5.3.1}, we have
\[\F_\pi(Z_1, \dots, Z_n) = \F_{\pi|_{W}}\left(\left(Z_1, \dots, Z_{p - 1}, Z_pL_{\E_{\pi|_{V}}\left((Z_1, \dots, Z_n)|_{V}\right)}, Z_{p + 1}, \dots, Z_n\right)|_{W}\right)\]
for all three possible cases, i.e., $\chi(q) = \ell$; $\chi(q) = r$ and $p < q$; $\chi(q) = r$ and $p > q$.

To verify the modified condition $(4)$ in full generality, we examine the proof of \cite{CNS2015-2}*{Lemma 5.3.4}. Suppose $\chi(p) = \ell$ (the other case is similar), and note that under the additional assumption of the modified condition $(4)$ that there exists a block $W_0 \subset W$ of $\pi$ such that $\min_{\prec_\chi}(\{1, \dots, n\}), \max_{\prec_\chi}(\{1, \dots, n\}) \in W_0$, the block $W_0$ is always the only exterior block of $\pi$. Let
\[
\alpha = \max_{\prec_\chi}\left(\left\{k \in W_0 \, \mid \, k \preceq_\chi p\right\}\right), \quad \beta = \min_{\prec_\chi}\left(\left\{k \in W_0 \, \mid \, q \preceq_\chi k\right\}\right),
\]
and let $U = \{k \, \mid \, \alpha \prec_\chi k \prec_\chi \beta\}$.  Thus $U$ is a union of blocks of $\pi$. Let $\overline{W_0} = U^{\complement}$. Then, by the special case above (with $U$ being the $\chi$-interval), we have
\[\F_\pi(Z_1, \dots, Z_n) = \F_{\pi|_{\overline{W_0}}}\left(\left(Z_1, \dots, Z_{\alpha - 1}, Z_\alpha L_{\E_{\pi|_{U}}\left((Z_1, \dots, Z_n)|_{U}\right)}, Z_{\alpha + 1}, \dots, Z_n\right)|_{\overline{W_0}}\right).\]
Since $\E$ is operator-valued bi-multiplicative, we have
\[Z_\alpha L_{\E_{\pi|_{U}}\left((Z_1, \dots, Z_n)|_{U}\right)} = Z_pL_{\E_{\pi|_{V}}\left((Z_1, \dots, Z_n)|_{V}\right)}L_{\E_{\pi|_{U \setminus V}}\left((Z_1, \dots, Z_n)|_{U \setminus V}\right)}\]
if $\alpha = p$, and
\[\E_{\pi|_{U}}\left((Z_1, \dots, Z_n)|_{U}\right) = \E_{\pi|_{U \setminus V}}\left(\left(Z_1, \dots, Z_{p - 1}, Z_pL_{\E_{\pi|_{V}}\left((Z_1, \dots, Z_n)|_{V}\right)}, Z_{p + 1}, \dots, Z_n\right)|_{U \setminus V}\right)\]
otherwise. Since $W = \overline{W_0} \cup (U \setminus V)$, the assertion follows from applying the special case above in the opposite direction.
\end{proof}

\subsection{Operator-valued conditionally bi-free cumulant pairs}

In this subsection, we recursively define the operator-valued conditionally bi-free cumulant pair $(\kappa, \mathcal{K})$ using the pair $(\E, \F)$ from the previous subsection and show that it is also operator-valued conditionally bi-multiplicative.

\begin{defn}\label{OpVCBFCumulants}
Let $(\A, \mathbb{E}, \mathbb{F}, \varepsilon)$ be a $\mathcal{B}$-$\mathcal{B}$-non-commutative probability space with a pair of $(\B, \D)$-valued expectations and let $(\E, \F)$ be the operator-valued conditionally bi-free moment pair on $\A$. The \textit{operator-valued conditionally bi-free cumulant pair} on $\A$ is the pair of functions
\[\kappa: \bigcup_{n \geq 1}\bigcup_{\chi: \{1, \dots, n\} \to \{\ell, r\}}\B\N\C(\chi) \times \A_{\chi(1)} \times \cdots \times \A_{\chi(n)} \to \B\]
and
\[\K: \bigcup_{n \geq 1}\bigcup_{\chi: \{1, \dots, n\} \to \{\ell, r\}}\B\N\C(\chi) \times \A_{\chi(1)} \times \cdots \times \A_{\chi(n)} \to \D\]
where $\kappa$ is the operator-valued bi-free cumulant function on $\A$ and $\K$ is recursively defined as follows.
\begin{enumerate}[$\qquad(1)$]
\item If $n = 1$, then $\K_{1_{1,0}}(Z_\ell) = \F_{1_{1,0}}(Z_\ell)$ for $Z_\ell \in \A_\ell$ and $\K_{1_{0,1}}(Z_r) = \F_{1_{0,1}}(Z_r)$ for $Z_r \in \A_r$.

\item Fix $n \geq 2$, $\chi: \{1, \dots, n\} \to \{\ell, r\}$, $\pi \in \B\N\C(\chi)$, and $Z_k \in \A_{\chi(k)}$. If $\pi \neq 1_\chi$, then let $V_1, \dots, V_m$ be the partition of $\pi$ as described in Remark \ref{Decomposition}. We define
\[\K_\pi(Z_1, \dots, Z_n) = \K_{\pi|_{V_1}}((Z_1, \dots, Z_n)|_{V_1})\cdots \K_{\pi|_{V_m}}((Z_1, \dots, Z_n)|_{V_m}),\]
where each $\K_{\pi|_{V_k}}((Z_1, \dots, Z_n)|_{V_k})$ is defined as follows. Let $V'_k \subset V_k$ be the block containing $\min_{\prec_\chi}(V_k)$ and $\max_{\prec_\chi}(V_k)$, let $V \subset V_k \setminus V_k'$ be the block which terminates closest to the bottom (compared to other blocks of $V_k$).  If $p = \max_{\prec_\chi}\left(\left\{j \in V_k \, \mid \, j \prec_\chi \min_{\prec_\chi}(V)\right\}\right)$ define
\begin{align*}
&\K_{\pi|_{V_k}}((Z_1, \dots, Z_n)|_{V_k})\\
&= \begin{cases}
\K_{\pi|_{V_k \setminus V}}\left(\left(Z_1, \dots, Z_{p - 1}, Z_pL_{\kappa_{\pi|_{V}}\left((Z_1, \dots, Z_n)|_{V}\right)}, Z_{p + 1}, \dots, Z_n\right)|_{V_k \setminus V}\right) &\text{if } \chi(p) = \ell\\
\K_{\pi|_{V_k \setminus V}}\left(\left(Z_1, \dots, Z_{p - 1}, R_{\kappa_{\pi|_{V}}\left((Z_1, \dots, Z_n)|_{V}\right)}Z_p, Z_{p + 1}, \dots, Z_n\right)|_{V_k \setminus V}\right) &\text{if } \chi(p) = r
\end{cases}.
\end{align*}
Repeat this process until the only remaining block of $V_k$ is $V'_k$. 

\item Otherwise $\pi = 1_\chi$ and define
\[\K_{1_\chi}(Z_1, \dots, Z_n) = \F_{1_\chi}(Z_1, \dots, Z_n) - \sum_{\substack{\pi \in \B\N\C(\chi)\\\pi \neq 1_\chi}}\K_\pi(Z_1, \dots, Z_n).\]
\end{enumerate}
\end{defn}

\begin{thm}
Let $(\A, \mathbb{E}, \mathbb{F}, \varepsilon)$ be a $\mathcal{B}$-$\mathcal{B}$-non-commutative probability space with a pair of $(\B, \D)$-valued expectations. The operator-valued conditionally bi-free cumulant pair $(\kappa, \K)$ on $\A$ is operator-valued conditionally bi-multiplicative.
\end{thm}

\begin{proof}
The fact that the operator-valued bi-free cumulant function $\kappa$ on $\A$ is operator-valued bi-multiplicative was proved in \cite{CNS2015-2}*{Section 6}. Moreover, it is easy to see that the function $\K$ satisfies condition $(3)$ of Definition \ref{BiMulti}.

For condition $(1)$ of Definition \ref{BiMulti} we will proceed by induction on $n$ to show that condition $(1)$ holds in greater generality.  To be specific, we will demonstrate that condition $(1)$ holds whenever $1_\chi$ is replaced with $\pi$.  To proceed, note the base case where $n = 1$ is trivial.  For the inductive step, suppose the assertion holds for all $1 \leq n_0 \leq n - 1$, $\chi_0: \{1, \dots, n_0\} \to \{\ell, r\}$, and $\pi_0 \in \B\N\C(\chi_0)$.  Suppose $\chi : \{1,\ldots, n\} \to \{\ell, r\}$ and that $\chi(n) = \ell$ (as the other case is similar).   If $q = -\infty$, then $\chi: \{1, \dots, n\} \to \{\ell\}$ is the constant map, and thus for each $\pi \in \B\N\C(\chi)$, $n$ necessarily belongs to an exterior block of $\pi$. Since $\K_\pi(Z_1, \dots, Z_{n - 1}, Z_nL_b)$ factors according to the exterior blocks of $\pi$, we have $\K_\pi(Z_1, \dots, Z_{n - 1}, Z_nL_b) = \K_\pi(Z_1, \dots, Z_n)b$ if $\pi \neq 1_\chi$ by the induction hypothesis.  Thus
\begin{align*}
\K_{1_\chi}(Z_1, \dots, Z_{n - 1}, Z_nL_b) &= \F_{1_\chi}(Z_1, \dots, Z_{n - 1}, Z_nL_b) - \sum_{\substack{\pi \in \B\N\C(\chi)\\\pi \neq 1_\chi}}\K_\pi(Z_1, \dots, Z_{n - 1}, Z_nL_b)\\
&= \F_{1_\chi}(Z_1, \dots, Z_n)b - \sum_{\substack{\pi \in \B\N\C(\chi)\\\pi \neq 1_\chi}}\K_\pi(Z_1, \dots, Z_n)b\\
&= \K_{1_\chi}(Z_1, \dots, Z_n)b.
\end{align*}

If $q \neq -\infty$ and $\pi \in \B\N\C(\chi)$ such that $\pi \neq 1_\chi$, then as $n$ and $q$ are adjacent with respect to $\prec_\chi$, we have the following possible cases:
\begin{enumerate}[$\qquad(i)$]
\item $n, q \in V$ such that $V$ is an interior block of $\pi$;

\item $n, q \in V$ such that $V$ is an exterior block of $\pi$;

\item $n \in V_1$ and $q \in V_2$ such that $n= \max_{\prec_\chi}(V_1) \prec_\chi \min_{\prec_\chi}(V_2) = q$, and both of $V_1$ and $V_2$ are interior blocks of $\pi$;

\item $n \in V_1$ and $q \in V_2$ such that $n= \max_{\prec_\chi}(V_1) \prec_\chi \min_{\prec_\chi}(V_2)= q$, and both of $V_1$ and $V_2$ are exterior blocks of $\pi$;

\item $n \in V_1$ and $q \in V_2$ such that $\min_{\prec_\chi}(V_2) \prec_\chi \min_{\prec_\chi}(V_1)$ (thus $V_1$ is interior with respect to $V_2$), and $V_2$ is an interior block of $\pi$;

\item $n \in V_1$ and $q \in V_2$ such that $\min_{\prec_\chi}(V_2) \prec_\chi \min_{\prec_\chi}(V_1)$ (thus $V_1$ is interior with respect to $V_2$), and $V_2$ is an exterior block of $\pi$;

\item $n \in V_1$ and $q \in V_2$ such that $\max_{\prec_\chi}(V_2) \prec_\chi \max_{\prec_\chi}(V_1)$ (thus $V_2$ is interior with respect to $V_1$), and $V_1$ is an interior block of $\pi$;

\item $n \in V_1$ and $q \in V_2$ such that $\max_{\prec_\chi}(V_2) \prec_\chi \max_{\prec_\chi}(V_1)$ (thus $V_2$ is interior with respect to $V_1$), and $V_1$ is an exterior block of $\pi$.
\end{enumerate}
Since $\pi \neq 1_\chi$, cases $(i)$, $(ii)$, $(iii)$, $(v)$, $(vi)$, $(vii)$, and $(viii)$ follow from the induction hypothesis and from the fact that $\kappa$ is operator-valued bi-multiplicative. For case $(iv)$, $V_1 \subset \chi^{-1}(\{\ell\})$ and $V_2 \subset \chi^{-1}(\{r\})$, so the result follows from the $q = -\infty$ situation (and the proof where $\chi(n) = r$ which must be run simultaneously with induction). Therefore, we have
\[\K_\pi(Z_1, \dots, Z_{n - 1}, Z_nL_b) = \K_\pi(Z_1, \dots, Z_{q - 1}, Z_qR_b, Z_{q + 1}, \dots, Z_n)\]
for all $\pi \neq 1_\chi$, and hence
\[
\K_{1_\chi}(Z_1, \dots, Z_{n - 1}, Z_nL_b) = \K_{1_\chi}(Z_1, \dots, Z_{q - 1}, Z_qR_b, Z_{q + 1}, \dots, Z_n)
\]
by the same calculation as the $q = -\infty$ situation.

The verification for condition $(2)$ of Definition \ref{BiMulti} follows from essentially the same induction arguments and casework as above with $p$ replacing $n$. The only difference is that if $q = -\infty$, then $p$ is the smallest element with respect to $\prec_\chi$, and hence necessarily belongs to an exterior block of $\pi$. Note this shows that the function $\K$ actually satisfies the additional properties that conditions $(1)$ and $(2)$ of Definition \ref{BiMulti} hold for all $\pi \in \B\N\C(\chi)$.

Finally, for the modified condition $(4)$ as given in Definition \ref{CondBiMulti}, the result follows from the extended conditions $(1)$ and $(2)$ of Definition \ref{BiMulti} as stated above along with the recursive definition in Definition \ref{CBFMomentPair} and the fact that $\kappa$ is operator-valued bi-multiplicative.
\end{proof}

\section{Universal moment expressions for c-bi-free independence with amalgamation}\label{sec:moment-express}

In this section, we will demonstrate that a family of pairs of $\B$-algebras is c-bi-free over $(\B, \D)$ if and only if certain operator-valued moment expressions hold. To do so, we note that the shaded diagrams from Definition \ref{ShadedDiagrams} and \cite{CNS2015-2}*{Lemma 7.1.3}   will be useful.

\begin{defn}
Let $\{(\X_k, \X_k^\circ, \mathfrak{p}_k, \mathfrak{q}_k)\}_{k \in K}$ be a family of $\B$-$\B$-bimodules with pairs of specified $(\B, \D)$-valued states, let $\lambda_k$ and $\rho_k$ be the left and right representations of $\L(\X_k)$ on $\L(\X)$, and let $\X = (*_\B)_{k \in K}\X_k$. Fix $\chi: \{1, \dots, n\} \to \{\ell, r\}$, $\omega: \{1, \dots, n\} \to K$, $Z_k \in \L_{\chi(k)}(\X_{\omega(k)})$, and let $\mu_k(Z_k) = \lambda_{\omega(k)}(Z_k)$ if $\chi(k) = \ell$ and $\mu_k(Z_k) = \rho_{\omega(k)}(Z_k)$ if $\chi(k) = r$.

For $D \in \L\R^\lat(\chi, \omega)$, recursively define $\bE_D(\mu_1(Z_1), \dots, \mu_n(Z_n))$ as follows: If $D \in \L\R_0^\lat(\chi, \omega)$, then
\[\bE_D(\mu_1(Z_1), \dots, \mu_n(Z_n)) = (\E_{\L(\X)})_\pi(\mu_1(Z_1), \dots, \mu_n(Z_n)) \in \B,\]
where $\pi \in \B\N\C(\chi)$ is the bi-non-crossing partition corresponding to $D$. If every block of $D$ has a spine reaching the top, then enumerate the blocks from left to right according to their spines as $V_1, \dots, V_m$ with $V_j = \{k_{j, 1} < \cdots < k_{j, q_j}\}$, and set
\[\bE_D(\mu_1(Z_1), \dots, \mu_n(Z_n)) = [(1 - \mathfrak{p}_{\omega(k_{1, 1})})Z_{k_{1, 1}}\cdots Z_{k_{1, q_1}}(1 \oplus 0)] \otimes \cdots \otimes [(1 - \mathfrak{p}_{\omega(k_{m, 1})})Z_{k_{m, 1}}\cdots Z_{k_{m, q_m}}(1 \oplus 0)], \]
which is an element of $\X^\circ$.
Otherwise, apply the recursive process using $\bE_{\L(\X)}$ as in Definition \ref{E-pi} until every block of $D$ has a spine reaching the top.
\end{defn}

Under the above assumptions and notation, it was demonstrated in \cite{CNS2015-2}*{Lemma 7.1.3} that
\begin{equation}\label{MomentExpression}
\mu_1(Z_1)\cdots\mu_n(Z_n)(1 \oplus 0) = \sum_{k = 0}^n\sum_{D \in \L\R_k^\lat(\chi, \omega)}\left[\sum_{\substack{D' \in \L\R_k(\chi, \omega)\\D' \geq_\lat D}}(-1)^{|D| - |D'|}\right]\bE_D(\mu_1(Z_1), \dots, \mu_n(Z_n))
\end{equation}
and, consequently,
\begin{equation}\label{E-Moment}
\bE_{\L(\X)}(\mu_1(Z_1)\cdots\mu_n(Z_n)) = \sum_{\pi \in \B\N\C(\chi)}\left[\sum_{\substack{\sigma \in \B\N\C(\chi)\\\pi \leq \sigma \leq \omega}}\mu_{\B\N\C}(\pi, \sigma)\right](\E_{\L(\X)})_\pi(\mu_1(Z_1), \dots, \mu_n(Z_n)).
\end{equation}

For $D \in \L\R^{\lat\capp}(\chi, \omega)$, we define $\bE_D(\mu_1(Z_1), \dots, \mu_n(Z_n))$ by exactly the same recursive process that used to define $\bE_{D'}(\mu_1(Z_1), \dots, \mu_n(Z_n))$ for $D' \in \L\R^\lat(\chi, \omega)$. Note that, unlike $\bE_{D'}(\mu_1(Z_1), \dots, \mu_n(Z_n))$, it is not necessarily true that $\bE_D(\mu_1(Z_1), \dots, \mu_n(Z_n)) \in \X$ for all $D \in \L\R^{\lat\capp}(\chi, \omega)$ as such diagrams may have spines reaching the top which do not alternate in colour. 

If $\bE_D(\mu_1(Z_1), \dots, \mu_n(Z_n)) = X_1 \otimes \cdots \otimes X_m$, let
\[\mathfrak{q}(\bE_D(\mu_1(Z_1), \dots, \mu_n(Z_n))) = \mathfrak{q}(X_1)\cdots\mathfrak{q}(X_m) \in \D.\]
Observe that although it is possible $X_1 \otimes \cdots \otimes X_m \notin \X^\circ$, it is still true that every $X_j$ belongs to some $\X_{k_j}^\circ$, and thus the above expression makes sense.

Finally for $D \in \L\R^{\lat\capp}(\chi, \omega)$, recursively define $\bF_D(\mu_1(Z_1), \dots, \mu_n(Z_n))$ as follows: If $D \in \L\R_0^{\lat\capp}(\chi, \omega)$, then
\[\bF_D(\mu_1(Z_1), \dots, \mu_n(Z_n)) = \bE_D(\mu_1(Z_1), \dots, \mu_n(Z_n)) \in \B \subset \D.\]
If every block of $D$ has a spine reaching the top, then enumerate the blocks from left to right according to their spines as $V_1, \dots, V_m$ with $V_j = \{k_{j, 1} < \cdots < k_{j, q_j}\}$, and set
\[\bF_D(\mu_1(Z_1), \dots, \mu_n(Z_n)) = \bF_{\L(\X)}(Z_{k_{1, 1}}\cdots Z_{k_{1, q_1}})\cdots\bF_{\L(\X)}(Z_{k_{m, 1}}\cdots Z_{k_{m, q_m}}) \in \D.\]
Otherwise, apply the recursive process using $\bE_{\L(\X)}$ as in Definition \ref{E-pi} until every block of $D$ has a spine reaching the top.

Note the values of $\bF_D(\mu_1(Z_1), \dots, \mu_n(Z_n))$ depend only on the values of $\bF_{\L(\X_{k_j})}(Z_{k_{j, 1}}\cdots Z_{k_{j, q_j}})$ and the values of $\bE_{\L(\X_{k_j})}(Z_{k_{j, 1}}\cdots Z_{k_{j, q_j}})$ for some $k_j \in K$.  Hence $\bF_D$ makes sense in any $\B$-$\B$-non-commutative probability space with a pair of $(\B, \D)$-vector expectations and will not depend on the representation of the pairs of $\B$-algebras.  

\begin{lem}\label{ChangeCoeff}
Under the above assumptions and notation, for all $D \in \L\R^\lat(\chi, \omega)$
\[\sum_{\substack{D' \in \L\R^{\lat\capp}(\chi, \omega)\\D' \leq_\capp D}}\mathfrak{q}\left(\bE_{D'}(\mu_1(Z_1), \dots, \mu_n(Z_n))\right) = \bF_D(\mu_1(Z_1), \dots, \mu_n(Z_n)).\]
\end{lem}

\begin{proof}
If $D \in \L\R^\lat_0(\chi, \omega)$, then the only diagram $D' \in \L\R^{\lat\capp}(\chi, \omega)$ such that $D' \leq_\capp D$ is $D$ itself. Thus the equation is trivially true by definition in this case. 

For $D \in \L\R_m^\lat(\chi, \omega)$ with $0 < m \leq n$, it suffices to prove the following claim: Let $V_1, \dots, V_m$ be the blocks of $D$ with spines reaching the top, ordered from left to right according to their spines, let $V_1 = \{k_{1, 1} < \cdots < k_{1, q_1}\}$, and let $V_{1, 1}, \dots, V_{1, m_1}$ be the blocks of $D$ which reduce to appropriate $L_b$ or $R_b$ multiplied on the left and/or right of some $Z_{k_{1, j}}$ in the recursive process.  Suppose $D', D'' \in \L\R^{\lat\capp}(\chi, \omega)$ are such that $D' \leq_{\capp} D$,  $D'' \leq_{\capp} D$, the spine of the block $V_1$ reaches the top in $D'$ but not in $D''$, and the spines of all other blocks in $D'$ and $D''$ agree.  We claim that
\begin{align*}
&\mathfrak{q}\left(\bE_{D'}(\mu_1(Z_1), \dots, \mu_n(Z_n))\right) + \mathfrak{q}\left(\bE_{D''}(\mu_1(Z_1), \dots, \mu_n(Z_n))\right)\\
&= \bF_{\L(\X)}(Z'_{k_{1, 1}}\cdots Z'_{k_{1, q_1}})\mathfrak{q}\left(\bE_{D' \setminus (V_1 \cup V_{1, 1} \cup \cdots \cup V_{1, m_1})}\left((\mu_1(Z_1), \dots, \mu_n(Z_n))|_{D' \setminus (V_1 \cup V_{1, 1} \cup \cdots \cup V_{1, m_1})}\right)\right),
\end{align*}
where $Z'_{k_{1, j}}$ is $Z_{k_{1, j}}$, potentially multiplied on the left and/or right by appropriate $L_b$ and $R_b$ such that the multiplications correspond to the blocks $V_{1, 1}, \dots, V_{1, m_1}$.

Indeed, if the claim is true, then for a given $D$ as above, the spine of $V_1$ reaches the top in exactly half of the cappings of $D$ and each such capping $D'$ can be paired with another capping $D''$ such that the only difference between $D'$ and $D''$ is that the spine of $V_1$ does not reach the top in $D''$. Adding up $\mathfrak{q}\left(\bE_{D'}(\mu_1(Z_1), \dots, \mu_n(Z_n))\right)$ and $\mathfrak{q}\left(\bE_{D''}(\mu_1(Z_1), \dots, \mu_n(Z_n))\right)$ for all pairs  yield the result by induction.

To prove the claim, note if $m = 1$ (that is, the only spine that reaches the top is the spine of $V_1$), then $V_1 \cup V_{1, 1} \cup \cdots \cup V_{1, m_1} = D'$ and we have
\begin{align*}
\mathfrak{q}\left(\bE_{D'}(\mu_1(Z_1), \dots, \mu_n(Z_n))\right) &= \bF_{\L(\X)}(Z'_{k_{1, 1}}\cdots Z'_{k_{1, q_1}}) - \bE_{\L(\X)}(Z'_{k_{1, 1}}\cdots Z'_{k_{1, q_1}})\\
&= \bF_{\L(\X)}(Z'_{k_{1, 1}}\cdots Z'_{k_{1, q_1}}) - \mathfrak{q}\left(\bE_{D''}(\mu_1(Z_1), \dots, \mu_n(Z_n))\right)
\end{align*}
as $D''$ has no spine reaching the top and $\mathfrak{q}(b) = b$.   Thus the result follows when $m = 1$.  

Otherwise, $m > 1$. Let $V = V_1 \cup V_{1, 1} \cup \cdots \cup V_{1, m_1}$.  Since left $\B$-operators commute with elements of $\L_r(\X)$, right $\B$-operators commute with elements of $\L_\ell(\X)$, and by the properties of $\bE_{\L(\X)}$ and $\bF_{\L(\X)}$ (i.e., there are bi-multiplicative-like properties implied by the recursive definition), it can be checked via casework that
\begin{align*}
\mathfrak{q}&\left(\bE_{D'}(\mu_1(Z_1), \dots, \mu_n(Z_n))\right) \\
&= \left(\bF_{\L(\X)}(Z'_{k_{1, 1}}\cdots Z'_{k_{1, q_1}}) - \bE_{\L(\X)}(Z'_{k_{1, 1}}\cdots Z'_{k_{1, q_1}})\right)\mathfrak{q}\left(\bE_{D' \setminus V}\left((\mu_1(Z_1), \dots, \mu_n(Z_n))|_{D' \setminus V}\right)\right)
\end{align*}

and
\[\bE_{\L(\X)}(Z'_{k_{1, 1}}\cdots Z'_{k_{1, q_1}})\mathfrak{q}\left(\bE_{D' \setminus V}\left((\mu_1(Z_1), \dots, \mu_n(Z_n))|_{D' \setminus V}\right)\right) = \mathfrak{q}\left(\bE_{D''}(\mu_1(Z_1), \dots, \mu_n(Z_n))\right)\]
for all $D'$ and $D''$.  Thus the claim and proof follows.
\end{proof}

To keep track of some coefficients that occur, we make the following definition.

\begin{defn}\label{CDandC'D}
For $D \in \L\R^{\lat\capp}_k(\chi, \omega)$, define $C'_D$ as follows: First define
\[
C_D = \begin{cases}
\displaystyle\sum_{\substack{D' \in \mathcal{L}\mathcal{R}_k(\chi, \omega)\\D' \geq_\lat D}}(-1)^{|D| - |D'|} &\text{if } D \in \mathcal{L}\mathcal{R}^{\lat}_k(\chi, \omega)\\
0 &\text{otherwise}
\end{cases}.
\]
Recursively, starting with $k = n$, define
\[
C'_D = C_D - \sum^n_{m = k + 1}\sum_{\substack{D' \in \mathcal{L}\mathcal{R}^{\lat\capp}_m(\chi, \omega)\\D' \geq_\capp D}}C'_{D'}.
\]
\end{defn}

With Lemma \ref{ChangeCoeff} complete, we obtain the following operator-valued analogue of \cite{GS2016}*{Lemma 4.6}.

\begin{lem}\label{F-Moment}
Under the above assumptions and notation,
\[\bF_{\L(\X)}(\mu_1(Z_1)\cdots\mu_n(Z_n)) = \sum_{k = 0}^n\sum_{D \in \mathcal{L}\mathcal{R}^{\lat\capp}_k(\chi, \omega)}C'_D\bF_D(\mu_1(Z_1), \dots, \mu_n(Z_n)),\]
and
\[C'_D = \sum_{\substack{ D' \in \mathcal{L}\mathcal{R}(\chi, \omega)\\D' \geq_{\lat\capp} D}}(-1)^{|D| - |D'|} = \sum^n_{m = k}\sum_{\substack{D' \in \mathcal{L}\mathcal{R}_m(\chi, \omega)\\D' \geq_{\lat\capp} D}}(-1)^{|D| - |D'|}\]
for $D \in \mathcal{L}\mathcal{R}^{\lat\capp}_k(\chi, \omega)$.
\end{lem}

\begin{proof}
For $Z_1, \dots, Z_n$ as above, the expression $\bF_{\L(\X)}(\mu_1(Z_1)\cdots\mu_n(Z_n))$ is obtained by applying $\mathfrak{q}$ to the left-hand side of equation \eqref{MomentExpression}. Using Definition \ref{CDandC'D}, we have
\begin{align*}
&\bF_{\L(\X)}(\mu_1(Z_1)\cdots\mu_n(Z_n))\\
&= \mathfrak{q}\left(\mu_1(Z_1)\cdots\mu_n(Z_n)(1 \oplus 0)\right)\\
&= \sum_{k = 0}^n\sum_{D \in \L\R_k^\lat(\chi, \omega)}C_D\mathfrak{q}\left(\bE_D(\mu_1(Z_1), \dots, \mu_n(Z_n))\right)\\
&= \sum_{k = 0}^n\sum_{D \in \L\R_k^\lat(\chi, \omega)}C_D\left(\bF_D(\mu_1(Z_1), \dots, \mu_n(Z_n)) - \sum_{\substack{D' \in \L\R^{\lat\capp}(\chi, \omega)\\D' \leq_\capp D\\D' \neq D}}\mathfrak{q}\left(\bE_{D'}(\mu_1(Z_1), \dots, \mu_n(Z_n))\right)\right)\\
&= \sum_{k = 0}^n\sum_{D \in \mathcal{L}\mathcal{R}^{\lat\capp}_k(\chi, \omega)}C'_D\bF_D(\mu_1(Z_1), \dots, \mu_n(Z_n)),
\end{align*}
where the third equality follows from Lemma \ref{ChangeCoeff} and the fourth equality follows from Definition \ref{CDandC'D} as the coefficient $C'_D$ for $D \in \mathcal{L}\mathcal{R}^{\lat\capp}_k(\chi, \omega)$ was specifically defined this way. The second result regarding $C'_D$ is exactly the content of \cite{GS2016}*{Lemma 4.7}.
\end{proof}

Combining these results, we have the following moment type characterization of c-bi-free independence with amalgamation.

\begin{thm}\label{MomentFormulae}
A family $\{(\A_{k, \ell}, \A_{k, r})\}_{k \in K}$ of pairs of $\B$-algebras in a $\B$-$\B$-non-commutative probability space with a pair of $(\B, \D)$-valued expectations $(\A, \mathbb{E}, \mathbb{F}, \varepsilon)$ is c-bi-free over $(\B, \D)$ if and only if
\begin{equation}\label{EA-Moment}
\bE(Z_1\cdots Z_n) = \sum_{\pi \in \B\N\C(\chi)}\left[\sum_{\substack{\sigma \in \B\N\C(\chi)\\\pi \leq \sigma \leq \omega}}\mu_{\B\N\C}(\pi, \sigma)\right]\E_\pi(Z_1, \dots, Z_n)
\end{equation}
and
\begin{equation}\label{FA-Moment}
\bF(Z_1\cdots Z_n) = \sum_{D \in \L\R^{\lat\capp}(\chi, \omega)}\left[\sum_{\substack{D' \in \L\R(\chi, \omega)\\D' \geq_{\lat\capp} D}}(-1)^{|D| - |D'|}\right]\bF_D(Z_1, \dots, Z_n)
\end{equation}
for all $n \geq 1$, $\chi: \{1, \dots, n\} \to \{\ell, r\}$, $\omega: \{1, \dots, n\} \to K$, and $Z_1, \dots, Z_n \in \A$ with $Z_k \in \A_{\omega(k), \chi(k)}$.
\end{thm}

\begin{proof}
Under the above notation, if the family $\{(\A_{k, \ell}, \A_{k, r})\}_{k \in K}$ is c-bi-free over $(\B, \D)$, then there exists a family $\{(\X_k, \X^\circ_k, \mathfrak{p}_k, \mathfrak{q}_k)\}_{k \in K}$ such that
\[\bE(Z_1\cdots Z_n) = \bE_{\L(\X)}(\mu_1(Z_1)\cdots\mu_n(Z_n))\qand\bF(Z_1\cdots Z_n) = \bF_{\L(\X)}(\mu_1(Z_1)\cdots\mu_n(Z_n)),\]
where each $Z_k$ on the right-hand side of the above equations is identified as $\ell_k(Z_k)$ if $\chi(k) = \ell$ and $r_k(Z_k)$ if $\chi(k) = r$ acting on $\X_{\omega(k)}$. The fact that equation \eqref{EA-Moment} holds is part of \cite{CNS2015-2}*{Theorem 7.1.4}, and the fact that equation \eqref{FA-Moment} holds follows from Lemma \ref{F-Moment}.

Conversely, suppose equations \eqref{EA-Moment} and \eqref{FA-Moment} hold. By Theorem \ref{Embedding}, there exist $(\mathcal{X}, \mathcal{X}^\circ, \mathfrak{p}, \mathfrak{q})$ and a unital homomorphism $\theta: \A \to \L(\X)$ such that
\begin{gather*}
\theta(L_{b_1}R_{b_2}) = L_{b_1}R_{b_2},\quad\theta(\A_\ell) \subset \L_\ell(\X),\quad\theta(\A_r) \subset \L_r(\X),\\
\mathbb{E}_{\L(\X)}(\theta(Z)) = \mathbb{E}(Z),\qand\mathbb{F}_{\L(\X)}(\theta(Z)) = \mathbb{F}(Z)
\end{gather*}
for all $b_1, b_2 \in \B$ and $Z \in \A$. For each $k \in K$, let $(\X_k, \X^\circ_k, \mathfrak{p}_k, \mathfrak{q}_k)$ be a copy of $(\X, \X^\circ, \mathfrak{p}, \mathfrak{q})$, and let $\ell_k$ and $r_k$ be copies of $\theta: \A \to \L(\X_k)$. By \cite{CNS2015-2}*{Lemma 7.1.3} and Lemma \ref{F-Moment}, we have
\[\bE(Z_1\cdots Z_n) = \bE_{\L((*_\B)_{k \in K}\X_k)}(\mu_1(Z_1)\cdots\mu_n(Z_n))\qand\bF(Z_1\cdots Z_n) = \bF_{\L((*_\B)_{k \in K}\X_k)}(\mu_1(Z_1)\cdots\mu_n(Z_n)),\]
where each $Z_k$ on the right-hand side of the above equations is identified as $\theta(Z_k)$ acting on $\X_{\omega(k)}$. Hence, the family $\{(\A_{k, \ell}, \A_{k, r})\}_{k \in K}$ is c-bi-free over $(\B, \D)$ by definition.
\end{proof}

As $\bF_D(Z_1, \ldots, Z_n)$ and $\E_\pi(Z_1, \ldots, Z_n)$ depend only on the distributions of individual pairs $(\A_{k,\ell}, \A_{k, r})$ inside our $\B$-$\B$-non-commutative probability space with a pair of $(\B, \D)$-valued expectations, we obtain that Definition \ref{defn:op-c-bi-free-definition} is well-defined in that the joint distributions do not depend on the representations.

\section{Additivity of operator-valued conditionally bi-free cumulant pairs}\label{sec:additivity}

The goal of this section is to prove the operator-valued analogue of \cite{GS2016}*{Theorem 4.1}; namely that  conditionally bi-free independence with amalgamation is equivalent to the vanishing of mixed operator-valued bi-free and conditionally bi-free cumulants. To establish the result, we will need a method, analogous to \cite{S2015}*{Lemma 3.8} for constructing a pair of $\B$-algebras with any given operator-valued bi-free and conditionally bi-free cumulants. To this end, we discuss moment and cumulant series first.

Let $(\A, \bE, \bF, \varepsilon)$ be a $\B$-$\B$-non-commutative probability space with a pair of $(\B, \D)$-valued expectations, and let $(\C_\ell, \C_r)$ be a pair of $\B$-algebras such that
\[\C_\ell = \alg(\{Z_i\}_{i \in I}, \varepsilon(\B \otimes 1)) \qand \C_r = \alg(\{Z_j\}_{j \in J}, \varepsilon(1 \otimes \B^{\mathrm{op}}))\]
for some $\{Z_i\}_{i \in I} \subset \A_\ell$ and $\{Z_j\}_{j \in J} \subset \A_r$. By discussions in \cite{S2015}*{Section 2} and by using the operator-valued conditionally bi-multiplicative properties, only certain operator-valued bi-free and conditionally bi-free moments/cumulants are required to study the joint distributions of elements in $\alg(\C_\ell, \C_r)$ with respect to $(\bE, \bF)$. We make the following notation (in addition to \cite{S2015}*{Notation 2.18} with some slight notational changes) and definition to describe the necessary moments and cumulants.

\begin{nota}\label{MCSeries}
Let $\Z = \{Z_i\}_{i \in I} \sqcup \{Z_j\}_{j \in J}$ be as above, $n \geq 1$, $\omega: \{1, \dots, n\} \to I \sqcup J$, and $b_1, \dots, b_{n - 1} \in \B$.
\begin{itemize}
\item If $\omega(k) \in I$ for all $k$, define
\begin{align*}
\nu_{\omega}^{\mathcal{Z}}(b_1, \dots, b_{n - 1}) &= \bE\left(Z_{\omega(1)}L_{b_1}Z_{\omega(2)}\cdots L_{b_{n - 1}}Z_{\omega(n)}\right) \in \B,\\
\mu_{\omega}^{\mathcal{Z}}(b_1, \dots, b_{n - 1}) &= \bF\left(Z_{\omega(1)}L_{b_1}Z_{\omega(2)}\cdots L_{b_{n - 1}}Z_{\omega(n)}\right) \in \D,\\
\rho_{\omega}^{\mathcal{Z}}(b_1, \dots, b_{n - 1}) &= \kappa_{1_{\chi_\omega}}\left(Z_{\omega(1)}, L_{b_1}Z_{\omega(2)}, \dots, L_{b_{n - 1}}Z_{\omega(n)}\right) \in \B, \text{ and}\\
\eta_{\omega}^{\mathcal{Z}}(b_1, \dots, b_{n - 1}) &= \K_{1_{\chi_\omega}}\left(Z_{\omega(1)}, L_{b_1}Z_{\omega(2)}, \dots, L_{b_{n - 1}}Z_{\omega(n)}\right) \in \D.
\end{align*}

\item If $\omega(k) \in J$ for all $k$, define
\begin{align*}
\nu_{\omega}^{\mathcal{Z}}(b_1, \dots, b_{n - 1}) &= \bE\left(Z_{\omega(1)}R_{b_1}Z_{\omega(2)}\cdots R_{b_{n - 1}}Z_{\omega(n)}\right) \in \B,\\
\mu_{\omega}^{\mathcal{Z}}(b_1, \dots, b_{n - 1}) &= \bF\left(Z_{\omega(1)}R_{b_1}Z_{\omega(2)}\cdots R_{b_{n - 1}}Z_{\omega(n)}\right) \in \D,\\
\rho_{\omega}^{\mathcal{Z}}(b_1, \dots, b_{n - 1}) &= \kappa_{1_{\chi_\omega}}\left(Z_{\omega(1)}, R_{b_1}Z_{\omega(2)}, \dots, R_{b_{n - 1}}Z_{\omega(n)}\right) \in \B, \text{ and}\\
\eta_{\omega}^{\mathcal{Z}}(b_1, \dots, b_{n - 1}) &= \K_{1_{\chi_\omega}}\left(Z_{\omega(1)}, R_{b_1}Z_{\omega(2)}, \dots, R_{b_{n - 1}}Z_{\omega(n)}\right) \in \D.
\end{align*}

\item Otherwise, let $k_\ell = \min\{k \, \mid \, \omega(k) \in I\}$ and $k_r = \min\{k \, \mid \, \omega(k) \in J\}$. Then $\{k_\ell, k_r\} = \{1, k_0\}$ for some $k_0$. Define $\nu_{\omega}^{\mathcal{Z}}(b_1, \dots, b_{n - 1})$ and $\mu_{\omega}^{\mathcal{Z}}(b_1, \dots, b_{n - 1})$ to be
\[\bE\left(Z_{\omega(1)}C_{b_1}^{\omega(2)}Z_{\omega(2)}\cdots C_{b_{k_0 - 2}}^{\omega(k_0 - 1)}Z_{\omega(k_0 - 1)}Z_{\omega(k_0)}C_{b_{k_0 - 1}}^{\omega(k_0 + 1)}Z_{\omega(k_0 + 1)}\cdots C_{b_{n - 3}}^{\omega(n - 1)}Z_{\omega(n - 1)}C_{b_{n - 2}}^{\omega(n)}Z_{\omega(n)}C_{b_{n - 1}}^{\omega(n)}\right) \in \B\]
and
\[\bF\left(Z_{\omega(1)}C_{b_1}^{\omega(2)}Z_{\omega(2)}\cdots C_{b_{k_0 - 2}}^{\omega(k_0 - 1)}Z_{\omega(k_0 - 1)}Z_{\omega(k_0)}C_{b_{k_0 - 1}}^{\omega(k_0 + 1)}Z_{\omega(k_0 + 1)}\cdots C_{b_{n - 3}}^{\omega(n - 1)}Z_{\omega(n - 1)}C_{b_{n - 2}}^{\omega(n)}Z_{\omega(n)}C_{b_{n - 1}}^{\omega(n)}\right) \in \D\]
respectively, and define $\rho_{\omega}^{\mathcal{Z}}(b_1, \dots, b_{n - 1})$ and $\eta_{\omega}^{\mathcal{Z}}(b_1, \dots, b_{n - 1})$ to be
\[
\kappa_{1_{\chi_\omega}}\left(Z_{\omega(1)}, C_{b_1}^{\omega(2)}Z_{\omega(2)}, \dots, C_{b_{k_0 - 2}}^{\omega(k_0 - 1)}Z_{\omega(k_0 - 1)}, Z_{\omega(k_0)}, C_{b_{k_0 - 1}}^{\omega(k_0 + 1)}Z_{\omega(k_0 + 1)}, \dots, C_{b_{n - 3}}^{\omega(n - 1)}Z_{\omega(n - 1)}, C_{b_{n - 2}}^{\omega(n)}Z_{\omega(n)}C_{b_{n - 1}}^{\omega(n)}\right) \in \B
\]
and
\[\K_{1_{\chi_\omega}}\left(Z_{\omega(1)}, C_{b_1}^{\omega(2)}Z_{\omega(2)}, \dots, C_{b_{k_0 - 2}}^{\omega(k_0 - 1)}Z_{\omega(k_0 - 1)}, Z_{\omega(k_0)}, C_{b_{k_0 - 1}}^{\omega(k_0 + 1)}Z_{\omega(k_0 + 1)}, \dots, C_{b_{n - 3}}^{\omega(n - 1)}Z_{\omega(n - 1)}, C_{b_{n - 2}}^{\omega(n)}Z_{\omega(n)}C_{b_{n - 1}}^{\omega(n)}\right) \in \D\]
respectively, where
\[C_b^{\omega(k)} = \begin{cases}
L_b & \text{if } \omega(k) \in I\\
R_b & \text{if } \omega(k) \in J
\end{cases}.\]
\end{itemize}
\end{nota}

\begin{defn}\label{MCSeriesDefn}
Let $\Z = \{Z_i\}_{i \in I} \sqcup \{Z_j\}_{j \in J}$ be as above. The \textit{moment and cumulant series} of $\Z$ with respect to $(\bE, \bF)$ are the collections of maps
\begin{align*}
\nu^\Z &= \{\nu_\omega^{\Z}: \B^{n - 1} \to \B \, \mid \, n \geq 1, \omega: \{1, \dots, n\} \to I \sqcup J\},\\
\mu^\Z &= \{\mu_\omega^{\Z}: \B^{n - 1} \to \D \, \mid \, n \geq 1, \omega: \{1, \dots, n\} \to I \sqcup J\},
\end{align*}
and
\begin{align*}
\rho^\Z &= \{\rho_\omega^{\Z}: \B^{n - 1} \to \B \, \mid \, n \geq 1, \omega: \{1, \dots, n\} \to I \sqcup J\},\\
\eta^\Z &= \{\eta_\omega^{\Z}: \B^{n - 1} \to \D \, \mid \, n \geq 1, \omega: \{1, \dots, n\} \to I \sqcup J\},
\end{align*}
respectively. Note that if $n = 1$, then $\nu_\omega^\Z = \rho_\omega^\Z = \bE(Z_{\omega(1)})$ and $\mu_\omega^\Z = \eta_\omega^\Z = \bF(Z_{\omega(1)})$.
\end{defn}

\begin{lem}\label{Existence}
Let $I$ and $J$ be non-empty disjoint index sets, and let $\B$ and $\D$ be unital algebras such that $1 := 1_\D \in \B \subset \D$. For every $n \geq 1$ and $\omega: \{1, \dots, n\} \to I \sqcup J$, let $\Theta_\omega: \B^{n - 1} \to \B$ and $\Upsilon_\omega: \B^{n - 1} \to \D$ be linear in each coordinate. There exist a $\B$-$\B$-non-commutative probability space with a pair of $(\B, \D)$-valued expectations $(\A, \bE, \bF, \varepsilon)$ and elements $\{Z_i\}_{i \in I} \subset \A_\ell$ and $\{Z_j\}_{j \in J} \subset \A_r$ such that if $\mathcal{Z} = \{Z_i\}_{i \in I} \sqcup \{Z_j\}_{j \in J}$, then
\[\rho_\omega^{\mathcal{Z}}(b_1, \dots, b_{n - 1}) = \Theta_\omega(b_1, \dots, b_{n - 1}) \qand \eta_\omega^{\mathcal{Z}}(b_1, \dots, b_{n - 1}) = \Upsilon_\omega(b_1, \dots, b_{n - 1})\]
for all $n \geq 1$, $\omega: \{1, \dots, n\} \to I \sqcup J$, and $b_1, \dots, b_{n - 1} \in \B$.
\end{lem}

\begin{proof}
By the same construction presented in the proof of \cite{S2015}*{Lemma 3.8}, there exist a $\B$-$\B$-non-commutative probability space $(\A, \bE, \varepsilon)$ and $\mathcal{Z} = \{Z_i\}_{i \in I} \sqcup \{Z_j\}_{j \in J}$ with $\{Z_i\}_{i \in I} \subset \A_\ell$ and $\{Z_j\}_{j \in J} \subset \A_r$ such that
\[\rho_\omega^{\mathcal{Z}}(b_1, \dots, b_{n - 1}) = \Theta_\omega(b_1, \dots, b_{n - 1})\]
for all $n \geq 1$, $\omega: \{1, \dots, n\} \to I \sqcup J$, and $b_1, \dots, b_{n - 1} \in \B$.  Thus we need only define an expectation $\bF$ to produce the correct operator-valued conditionally bi-free cumulants.

For $n \geq 1$, $\omega: \{1, \dots, n\} \to I \sqcup J$, and $b_1, \dots, b_{n + 1} \in \B$, let
\[C_b^{\omega(k)} = \begin{cases}
L_b & \text{if } \omega(k) \in I\\
R_b & \text{if } \omega(k) \in J
\end{cases},\]
and define
\[\widehat{\Upsilon}_{1_{\chi_\omega}}\left(C_{b_1}^{\omega(1)}Z_{\omega(1)}, \dots, C_{b_{n - 1}}^{\omega(n - 1)}Z_{\omega(n - 1)}, C_{b_n}^{\omega(n)}Z_{\omega(n)}C_{b_{n + 1}}^{\omega(n)}\right) \in \D\]
like how $\widehat{\Theta}_{1_{\chi_\omega}}$ is defined in the proof of \cite{S2015}*{Lemma 3.8} using $\Upsilon_\omega$ instead of $\Theta_\omega$. Subsequently, for $\omega: \{1, \dots, n\} \to I \sqcup J$ and $\pi \in \B\N\C(\chi_\omega)$, define
\[\widehat{\Upsilon}_\pi\left(C_{b_1}^{\omega(1)}Z_{\omega(1)}, \dots, C_{b_{n - 1}}^{\omega(n - 1)}Z_{\omega(n - 1)}, C_{b_n}^{\omega(n)}Z_{\omega(n)}C_{b_{n + 1}}^{\omega(n)}\right) \in \D\]
by selecting one of the many possible ways to reduce an operator-valued conditionally bi-multiplicative function where $\widehat{\Theta}_{1_\chi}$ is used for interior blocks and $\widehat{\Upsilon}_{1_\chi}$ is used for exterior blocks.

As seen in the proof of \cite{S2015}*{Lemma 3.8}, every element in $\A$ is a linear combination of the form
\[C_{b_1}^{\omega(1)}Z_{\omega(1)}\cdots C_{b_n}^{\omega(n)}Z_{\omega(n)}L_bR_{b'} + \mathcal{I},\]
where $n \geq 0$, $\omega: \{1, \dots, n\} \to I \sqcup J$ when $n \geq 1$, $b_1, \dots, b_n, b, b' \in \B$, and $\mathcal{I}$ is some two-sided ideal. Define $\bF: \A \to \D$ by
\[\bF(L_bR_{b'} + \mathcal{I}) = bb'\]
for all $b, b' \in \B$, and
\[\bF\left(C_{b_1}^{\omega(1)}Z_{\omega(1)}\cdots C_{b_n}^{\omega(n)}Z_{\omega(n)}L_bR_{b'} + \mathcal{I}\right) = \sum_{\pi \in \B\N\C(\chi_\omega)}\widehat{\Upsilon}_\pi\left(C_{b_1}^{\omega(1)}Z_{\omega(1)}, \dots, C_{b_{n - 1}}^{\omega(n - 1)}Z_{\omega(n - 1)}, C_{b_n}^{\omega(n)}Z_{\omega(n)}C_{bb'}\right)\]
for all $n \geq 1$ and $\omega: \{1, \dots, n\} \to I \sqcup J$, where $C_{bb'} = L_{bb'}$ if $\omega(n) \in I$ and $C_{bb'} = R_{bb'}$ if $\omega(n) \in J$, and extend $\bF$ by linearity.
By construction and commutation in $\A$, one can verify that $\bF$ is well-defined and
\[\bF(L_bR_{b'}Z + \I) = b\bF(Z + \I)b' \qand \bF(ZL_b + \I) = \bF(ZR_b + \I)\]
for all $b, b' \in \B$ and $Z + \I \in \A$. Finally, since Definition \ref{OpVCBFCumulants} completely determines the operator-valued conditionally bi-free cumulants and by our definition of $\hat{\Upsilon}$ via a choice of operator-valued conditionally bi-multiplicative reduction, \cite{S2015}*{Lemma 3.8} with an induction argument together imply that if $\mathcal{Z} = \{Z_i\}_{i \in I} \sqcup \{Z_j\}_{j \in J}$, then
\[\eta_\omega^{\mathcal{Z}}(b_1, \dots, b_{n - 1}) = \Upsilon_\omega(b_1, \dots, b_{n - 1})\]
for all $n \geq 1$, $\omega: \{1, \dots, n\} \to I \sqcup J$, and $b_1, \dots, b_{n - 1} \in \B$.
\end{proof}

We are now ready to prove the main result of this section.

\begin{thm}\label{VanishingEquiv}
A family $\{(\A_{k, \ell}, \A_{k, r})\}_{k \in K}$ of pairs of $\B$-algebras in a $\B$-$\B$-non-commutative probability space with a pair of $(\B, \D)$-valued expectations $(\A, \mathbb{E}, \mathbb{F}, \varepsilon)$ is c-bi-free over $(\B, \D)$ if and only if for all $n \geq 2$, $\chi: \{1, \dots, n\} \to \{\ell, r\}$, $\omega: \{1, \dots, n\} \to K$, and $Z_k \in \A_{\omega(k), \chi(k)}$, we have
\[\kappa_{1_\chi}(Z_1, \dots, Z_n) = \K_{1_\chi}(Z_1, \dots, Z_n) = 0\]
whenever $\omega$ is not constant.
\end{thm}

\begin{proof}
If all mixed cumulants vanish, then $\{(\A_{k, \ell}, \A_{k, r})\}_{k \in K}$ is bi-free over $\B$ so equation \eqref{EA-Moment} holds. To see that equation \eqref{FA-Moment} also holds, recall from \cite{GS2016}*{Subsection 4.2} that $\B\N\C(\chi, ie)$ denotes the set of all pairs $(\pi, \iota)$ where $\pi \in \B\N\C(\chi)$ is a bi-non-crossing partition and $\iota: \pi \to \{i, e\}$ is a function on the blocks of $\pi$. By Definitions \ref{CBFMomentPair} and \ref{OpVCBFCumulants}, and the assumption that all mixed cumulants vanish, we have
\[\bF(Z_1\cdots Z_n) = \F_{1_\chi}(Z_1, \dots, Z_n) = \sum_{\substack{\pi \in \B\N\C(\chi)\\\pi \leq \omega}}\K_\pi(Z_1, \dots, Z_n).\]
By applying Definition \ref{OpVCBFCumulants} recursively, we obtain that
\begin{equation}\label{VanishingExpansion}
\bF(Z_1\cdots Z_n) = \sum_{\substack{(\pi, \iota) \in \B\N\C(\chi, ie)\\\pi \leq \omega}}c(\chi, \omega; \pi, \iota)\Theta_{(\pi, \iota)}(Z_1, \dots, Z_n),
\end{equation}
where $c(\chi, \omega; \pi, \iota)\Theta_{(\pi, \iota)}(Z_1, \dots, Z_n)$ for $(\pi, \iota) \in \B\N\C(\chi, ie)$ is defined as follows: If there is an interior block $V$ of $\pi$ such that $\iota(V) = e$, then $c(\chi, \omega; \pi, \iota) = 0$. Otherwise, apply the recursive process using $\bE$ as in Definition \ref{E-pi} to the interior blocks of $\pi$, order the remaining $\chi$-intervals by $\prec_\chi$ as $V_1, \dots, V_m$, and define
\[\Theta_{(\pi, \iota)}(Z_1, \dots, Z_n) = \Theta_{\pi|_{V_1}}((Z_1, \dots, Z_n)|_{V_1})\cdots\Theta_{\pi|_{V_m}}((Z_1, \dots, Z_n)|_{V_m}),\]
where $\Theta_{\pi|_{V_j}} = \bE_{\pi|_{V_j}}$ if $\iota(V_j) = i$ and $\Theta_{\pi|_{V_j}} = \bF_{\pi|_{V_j}}$ if $\iota(V_j) = e$.

Notice that, as with the scalar-valued case (see \cite{GS2016}*{Remark 4.9}), $\Theta_{(\pi, \iota)}(Z_1, \dots, Z_n)$ and $\bF_D(Z_1, \dots, Z_n)$ agree for certain $(\pi, \iota) \in \B\N\C(\chi, ie)$ and $D \in \L\R^{\lat\capp}(\chi, \omega)$. Indeed, given $D \in \L\R^{\lat\capp}(\chi, \omega)$, defining $\pi$ via the blocks of $D$ and $\iota$ via $\iota(V) = e$ if the spine of $V$ reaches the top and $\iota(V) = i$ otherwise will produce such an equality.

If each $(\pi, \iota) \in \B\N\C(\chi, ie)$ with $\pi \leq \omega$ and $c(\chi, \omega; \pi, \iota) \neq 0$ corresponds to some $D \in \L\R^{\lat\capp}(\chi, \omega)$ in the sense as described above and 
\[
c(\chi, \omega; \pi, \iota) = \sum_{\substack{D' \in \L\R(\chi, \omega)\\D' \geq_{\lat\capp} D}}(-1)^{|D| - |D'|}
\]
for such $(\pi, \iota)$, then equations \eqref{FA-Moment}  and \eqref{VanishingExpansion} coincide implying that $\{(\A_{k, \ell}, \A_{k, r})\}_{k \in K}$ is c-bi-free over $(\B, \D)$ by Theorem \ref{MomentFormulae}. Since the property that $(\pi, \iota)$ corresponds to a $D \in \L\R^{\lat\capp}(\chi, \omega)$ and the value of $c(\chi, \omega; \pi, \iota)$ do not depend on the algebras $\B$ and $\D$, the result follows from the $\B = \D = \bC$ case by \cite{GS2016}*{Lemma 4.13}.

Conversely, if $\{(\A_{k, \ell}, \A_{k, r})\}_{k \in K}$ is c-bi-free over $(\B, \D)$, then equations \eqref{EA-Moment} and \eqref{FA-Moment} hold by Theorem \ref{MomentFormulae}. As shown in \cite{CNS2015-2}*{Theorem 8.1.1}, equation \eqref{EA-Moment} is equivalent to the vanishing of mixed operator-valued bi-free cumulants. Thus we need only show that mixed operator-valued conditionally bi-free cumulants vanish. For fixed $n \geq 2$, $\chi: \{1, \dots, n\} \to \{\ell, r\}$, $\omega: \{1, \dots, n\} \to K$, and $Z_k \in \A_{\omega(k), \chi(k)}$, construct a $\B$-$\B$-non-commutative probability space with a pair of $(\B, \D)$-valued expectations $(\A', \bE_{\A'}, \bF_{\A'}, \varepsilon')$, pairs of $\B$-algebras $\{(\A'_{k, \ell}, \A'_{k, r})\}_{k \in K}$, and elements $Z'_k \in \A'_{\omega(k), \chi(k)}$ such that
\begin{itemize}
\item for each $k \in \{1, \dots, n\}$, $\{Z'_j \, \mid \, \omega(j) = \omega(k), \chi(j) = \chi(k)\}$ generated $\A'_{\omega(k), \chi(k)}$,

\item any joint operator-valued conditionally bi-free cumulant involving $Z'_1, \dots, Z'_n$ containing a pair $Z'_{k_1}, Z'_{k_2}$ with $\omega(k_1) \neq \omega(k_2)$ is zero, and

\item for each $k \in \{1, \dots, n\}$, the joint distribution of $\{Z'_j \, \mid \, \omega(j) = \omega(k)\}$ with respect to $(\bE_{\A'}, \bF_{\A'})$ equals the joint distribution of $\{Z_j \, \mid \, \omega(j) = \omega(k)\}$ with respect to $(\bE, \bF)$.
\end{itemize}
The above is possible via Lemma \ref{Existence} by defining the operator-valued bi-free and conditionally bi-free cumulants appropriately.

By construction, $Z'_1, \dots, Z'_n$ have vanishing mixed cumulants and hence satisfy equations \eqref{EA-Moment} and \eqref{FA-Moment} by the first part of the proof. However, since for each $k \in \{1, \dots, n\}$, the joint distribution of $\{Z'_j \, \mid \, \omega(j) = \omega(k)\}$ with respect to $(\bE_{\A'}, \bF_{\A'})$ equals the joint distribution of $\{Z_j \, \mid \, \omega(j) = \omega(k)\}$ with respect to $(\bE, \bF)$, equations \eqref{EA-Moment} and \eqref{FA-Moment} imply that the joint distribution of $Z_1, \dots, Z_n$ with respect to $(\bE, \bF)$ equals the joint distribution of $Z'_1, \dots, Z'_n$ with respect to $(\bE_{\A'}, \bF_{\A'})$. Since the operator-valued bi-free and conditionally bi-free moments completely determine the operator-valued bi-free and conditionally bi-free cumulants, and since $Z'_1, \dots, Z'_n$ have vanishing mixed cumulants, the result follows.
\end{proof}

\section{Additional properties of c-bi-free independence with amalgamation}\label{sec:additional}

In this section, we collect a list of additional properties of c-bi-free independence with amalgamation and operator-valued conditionally bi-free cumulants. All of the results below are analogues of known results in the current framework with essentially the same proofs. We begin by recalling the following notation from \cite{CNS2015-2}*{Notation 6.3.1}.

\begin{nota}
Let $\chi: \{1, \dots, n\} \to \{\ell, r\}$, $\pi \in \B\N\C(\chi)$, and $q \in \{1, \dots, n\}$. We denote by $\chi|_{\setminus q}$ the restriction of $\chi$ to the set $\{1, \dots, n\} \setminus q$. If $q \neq n$ and $\chi(q) = \chi(q + 1)$, define $\pi|_{q = q + 1} \in \B\N\C(\chi|_{\setminus q})$ to be the bi-non-crossing partition which results from identifying $q$ and $q + 1$ in $\pi$ (i.e., if $q$ and $q + 1$ are in the same block, then $\pi|_{q = q + 1}$ is obtained from $\pi$ by just removing $q$ from the block in which $q$ occurs, while if $q$ and $q + 1$ are in different blocks, then $\pi|_{q = q + 1}$ is obtained from $\pi$ by merging the two blocks and then removing $q$).
\end{nota}

\subsection{Vanishing of operator-valued cumulants}

The following demonstrates that, like with many other kinds of cumulants, the operator-valued conditionally bi-free cumulants of order at least two vanish if at least one input is a $\B$-operator.

\begin{prop}
Let $(\A, \mathbb{E}, \mathbb{F}, \varepsilon)$ be a $\mathcal{B}$-$\mathcal{B}$-non-commutative probability space with a pair of $(\B, \D)$-valued expectations, $\chi: \{1, \dots, n\} \to \{\ell, r\}$ with $n \geq 2$, and $Z_k \in \A_{\chi(k)}$. If there exist $q \in \{1, \dots, n\}$ and $b \in \B$ such that $Z_q = L_b$ if $\chi(q) = \ell$ or $Z_q = R_b$ if $\chi(q) = r$, then
\[\kappa_{1_\chi}(Z_1, \dots, Z_n) = \mathcal{K}_{1_\chi}(Z_1, \dots, Z_n) = 0.\]
\end{prop}

\begin{proof}
The assertion that $\kappa_{1_\chi}(Z_1, \dots, Z_n) = 0$ was proved in \cite{CNS2015-2}*{Proposition 6.4.1}, and the other assertion will be proved by induction with the base case easily verified by direct computations.

For the inductive step, suppose the assertion is true for all $\chi: \{1, \dots, m\} \to \{\ell, r\}$ with $2 \leq m \leq n - 1$. Fix $\chi: \{1, \dots, n\} \to \{\ell, r\}$ and $Z_k \in \A_{\chi(k)}$. Suppose there exist $q \in \{1, \dots, n\}$ and $b \in \B$ such that $\chi(q) = \ell$ and $Z_q = L_b$ (the case $\chi(q) = r$ and $Z_q = R_b$ is similar). Let
\[p = \max\{k \in \{1, \dots, n\} \, \mid \, \chi(k) = \ell, k < q\}.\]
There are two cases. If $p \neq -\infty$, then by the first assertion and the induction hypothesis,
\begin{align*}
\K_{1_\chi}(Z_1, \dots, Z_n) &= \F_{1_\chi}(Z_1, \dots, Z_n) - \sum_{\substack{\pi \in \B\N\C(\chi)\\\pi \neq 1_\chi}}\K_\pi(Z_1, \dots, Z_n)\\
&= \F_{1_\chi}(Z_1, \dots, Z_n) - \sum_{\substack{\pi \in \B\N\C(\chi)\\\{q\} \in \pi}}\K_\pi(Z_1, \dots, Z_{q - 1}, L_b, Z_{q + 1}, \dots, Z_n)\\
&= \F_{1_\chi}(Z_1, \dots, Z_n) - \sum_{\sigma \in \B\N\C(\chi|_{\setminus q})}\K_\sigma(Z_1, \dots, Z_pL_b, Z_{p + 1}, \dots, Z_{q - 1}, Z_{q + 1}, \dots, Z_n)
\end{align*}
by properties of $(\kappa, \K)$. On the other hand, we have
\begin{align*}
\F_{1_\chi}(Z_1, \dots, Z_n) &= \bF(Z_1\cdots Z_{q - 1}L_bZ_{q + 1}\cdots Z_n)\\
&= \bF(Z_1\cdots Z_pL_bZ_{p + 1}\cdots Z_{q - 1}Z_{q + 1}\cdots Z_n)\\
&= \F_{1_{\chi|_{\setminus q}}}(Z_1, \dots, Z_pL_b, Z_{p + 1}, \dots, Z_{q - 1}, Z_{q + 1}, \dots, Z_n)\\
&= \sum_{\sigma \in \B\N\C(\chi|_{\setminus q})}\K_\sigma(Z_1, \dots, Z_pL_b, Z_{p + 1}, \dots, Z_{q - 1}, Z_{q + 1}, \dots, Z_n),
\end{align*}
thus the assertion is true in this case. If $p = -\infty$, then by the first assertion and the induction hypothesis,
\begin{align*}
\K_{1_\chi}(Z_1, \dots, Z_n) &= \F_{1_\chi}(Z_1, \dots, Z_n) - \sum_{\substack{\pi \in \B\N\C(\chi)\\\pi \neq 1_\chi}}\K_\pi(Z_1, \dots, Z_n)\\
&= \F_{1_\chi}(Z_1, \dots, Z_n) - \sum_{\substack{\pi \in \B\N\C(\chi)\\\{q\} \in \pi}}\K_\pi(Z_1, \dots, Z_{q - 1}, L_b, Z_{q + 1}, \dots, Z_n)\\
&= \F_{1_\chi}(Z_1, \dots, Z_n) - \sum_{\sigma \in \B\N\C(\chi|_{\setminus q})}b\K_\sigma(Z_1, \dots, Z_{q - 1}, Z_{q + 1}, \dots, Z_n)
\end{align*}
by properties of $(\kappa, \K)$ as $q = \min_{\prec_\chi}(\{1, \dots, n\})$ in this case. On the other hand, we have
\begin{align*}
\F_{1_\chi}(Z_1, \dots, Z_n) &= \bF(Z_1\cdots Z_{q - 1}L_bZ_{q + 1}\cdots Z_n)\\
&= \bF(L_bZ_1\cdots Z_{q - 1}Z_{q + 1}\cdots Z_n)\\
&= b\F_{1_{\chi|_{\setminus q}}}(Z_1, \dots, Z_{q - 1}, Z_{q + 1}, \dots, Z_n)\\
&= \sum_{\sigma \in \B\N\C(\chi|_{\setminus q})}b\K_\sigma(Z_1, \dots, Z_{q - 1}, Z_{q + 1}, \dots, Z_n),
\end{align*}
thus the assertion is true in this case as well.
\end{proof}

\subsection{Operator-valued cumulants of products}

Next, we analyze operator-valued conditionally bi-free cumulants involving products of operators.

\begin{lem}
Let $(\A, \bE, \bF, \varepsilon)$ be a $\B$-$\B$-non-commutative probability space with a pair of $(\B, \D)$-valued expectations. If $\chi: \{1, \dots, n\} \to \{\ell, r\}$, $Z_k \in \A_{\chi(k)}$, and $q \in \{1, \dots, n - 1\}$ with $\chi(q) = \chi(q + 1),$ then
\[\K_\pi(Z_1, \dots, Z_{q - 1}, Z_qZ_{q + 1}, Z_{q + 2}, \dots, Z_n) = \sum_{\substack{\sigma \in \B\N\C(\chi)\\\sigma|_{q = q + 1} = \pi}}\K_\sigma(Z_1, \dots, Z_n)\]
for all $\pi \in \B\N\C(\chi|_{\setminus q})$.
\end{lem}

\begin{proof}
We proceed by induction on $n$. If $n = 1$, there is nothing to check. If $n = 2$, then
\[\K_{0_\chi|_{1 = 2}}(Z_1Z_2) = \K_{1_\chi|_{1 = 2}}(Z_1Z_2) = \F_{1_\chi|_{1 = 2}}(Z_1Z_2) = \F_{1_\chi}(Z_1, Z_2) = \K_{0_\chi}(Z_1, Z_2) + \K_{1_\chi}(Z_1, Z_2)\]
as required. Suppose the assertion holds for $n - 1$, and note from \cite{CNS2015-2}*{Theorem 6.3.5} that the analogous result also holds for the operator-valued bi-free cumulant function $\kappa$. Using the induction hypothesis and the operator-valued conditionally bi-multiplicativity of $(\kappa, \K)$, we see for all $\pi \in \B\N\C(\chi|_{\setminus q}) \setminus \{1_{\chi|_{\setminus q}}\}$ that
\[\K_\pi(Z_1, \dots, Z_{q - 1}, Z_qZ_{q + 1}, Z_{q + 2}, \dots, Z_n) = \sum_{\substack{\sigma \in \B\N\C(\chi)\\\sigma|_{q = q + 1} = \pi}}\K_\sigma(Z_1, \dots, Z_n).\]
Hence,
\begin{align*}
&\K_{1_{\chi|_{\setminus q}}}(Z_1, \dots, Z_{q - 1}, Z_qZ_{q + 1}, Z_{q + 2}, \dots, Z_n)\\
&= \F_{1_{\chi|_{\setminus q}}}(Z_1, \dots, Z_{q - 1}, Z_qZ_{q + 1}, Z_{q + 2}, \dots, Z_n) - \sum_{\substack{\pi \in \B\N\C(\chi|_{\setminus q})\\\pi \neq 1_{\chi|_{\setminus q}}}}\K_\pi(Z_1, \dots, Z_{q - 1}, Z_qZ_{q + 1}, Z_{q + 2}, \dots, Z_n)\\
&= \F_{1_\chi}(Z_1, \dots, Z_n) - \sum_{\substack{\pi \in \B\N\C(\chi|_{\setminus q})\\\pi \neq 1_{\chi|_{\setminus q}}}}\sum_{\substack{\sigma \in \B\N\C(\chi)\\\sigma|_{q = q + 1} = \pi}}\K_\sigma(Z_1, \dots, Z_n)\\
&= \sum_{\sigma \in \B\N\C(\chi)}\K_\sigma(Z_1, \dots, Z_n) - \sum_{\substack{\sigma \in \B\N\C(\chi)\\\sigma|_{q = q + 1} \neq 1_{\chi|_{\setminus q}}}}\K_\sigma(Z_1, \dots, Z_n)\\
&= \sum_{\substack{\sigma \in \B\N\C(\chi)\\\sigma|_{q = q + 1} = 1_{\chi|_{\setminus q}}}}\K_\sigma(Z_1, \dots, Z_n),
\end{align*}
completing the inductive step.
\end{proof}

Given two partitions $\pi, \sigma \in \B\N\C(\chi)$, let $\pi \vee \sigma$ denote the smallest partition in $\B\N\C(\chi)$ greater than $\pi$ and $\sigma$. Furthermore, suppose $m, n \geq 1$ with $m < n$ are fixed, and consider a sequence of integers
\[0 = k(0) < k(1) < \cdots < k(m) = n.\]
For $\chi: \{1, \dots, m\} \to \{\ell, r\}$, define $\widehat{\chi}: \{1, \dots, n\} \to \{\ell, r\}$ by
\[\widehat{\chi}(q) = \chi(p_q),\]
where $p_q$ is the unique number in $\{1, \dots, m\}$ such that $k(p_q - 1) < q \leq k(p_q)$. Let $\widehat{0_\chi}$ be the partition of $\{1, \dots, n\}$ with blocks $\{\{k(p - 1) + 1, \dots, k(p)\}\}_{p = 1}^m$. Recursively applying the previous lemma along with \cite{CNS2015-2}*{Theorem 9.1.5}  yields the following operator-valued analogue of \cite{GS2016}*{Theorem 4.22}.

\begin{thm}
Let $(\A, \bE, \bF, \varepsilon)$ be a $\B$-$\B$-non-commutative probability space with a pair of $(\B, \D)$-valued expectations. Under the above notation, we have
\[\K_{1_\chi}\left(Z_1\cdots Z_{k(1)}, Z_{k(1) + 1}\cdots Z_{k(2)}, \dots, Z_{k(m - 1) + 1}\cdots Z_{k(m)}\right) = \sum_{\substack{\sigma \in \B\N\C(\widehat{\chi})\\\sigma \vee \widehat{0_\chi} = 1_{\widehat{\chi}}}}\K_\sigma(Z_1, \dots, Z_n)\]
for all $\chi: \{1, \dots, m\} \to \{\ell, r\}$ and $Z_k \in \A_{\widehat{\chi}(k)}$.
\end{thm}

\subsection{Operator-valued conditionally bi-moment and bi-cumulant pairs}

In \cite{S1998}*{Subsection 3.2}, the classes of operator-valued moment and cumulant functions were introduced as a tool to calculate moment expressions of elements in amalgamated free products. The c-free extension (in the special case $\B = \bC$) was achieved in \cite{M2002}*{Section 3} and the bi-free analogue was obtained in \cite{CNS2015-2}*{Subsection 6.3}. In this subsection, we extend the notions of operator-valued bi-moment and bi-cumulant functions from \cite{CNS2015-2}*{Definition 6.3.2} to pairs of functions.

\begin{defn}\label{CondBiMC}
Let $(\A, \bE, \bF, \varepsilon)$ be a $\B$-$\B$-non-commutative probability space with a pair of $(\B, \D)$-valued expectations, and let
\[\phi: \bigcup_{n \geq 1}\bigcup_{\chi: \{1, \dots, n\} \to \{\ell, r\}}\B\N\C(\chi) \times \A_{\chi(1)} \times \cdots \times \A_{\chi(n)} \to \B\]
and
\[\Phi: \bigcup_{n \geq 1}\bigcup_{\chi: \{1, \dots, n\} \to \{\ell, r\}}\B\N\C(\chi) \times \A_{\chi(1)} \times \cdots \times \A_{\chi(n)} \to \D\]
be an operator-valued conditionally bi-multiplicative pair.
\begin{enumerate}[$\qquad(1)$]
\item We say that $(\phi, \Phi)$ is an \textit{operator-valued conditionally bi-moment pair} if whenever $\chi: \{1, \dots, n\} \to \{\ell, r\}$ is such that there exists a $q \in \{1, \dots, n - 1\}$ with $\chi(q) = \chi(q + 1)$, then
\[\phi_{1_{\chi|_{\setminus q}}}(Z_1, \dots, Z_{q - 1}, Z_qZ_{q + 1}, Z_{q + 2}, \dots, Z_n) = \phi_{1_\chi}(Z_1, \dots, Z_n)\]
and
\[\Phi_{1_{\chi|_{\setminus q}}}(Z_1, \dots, Z_{q - 1}, Z_qZ_{q + 1}, Z_{q + 2}, \dots, Z_n) = \Phi_{1_\chi}(Z_1, \dots, Z_n)\]
for all $Z_k \in \mathcal{A}_{\chi(k)}$.

\item We say that $(\phi, \Phi)$ is an \textit{operator-valued conditionally bi-cumulant pair} if whenever $\chi: \{1, \dots, n\} \to \{\ell, r\}$ is such that there exists a $q \in \{1, \dots, n - 1\}$ with $\chi(q) = \chi(q + 1)$, then
\[\phi_{1_{\chi|_{\setminus q}}}(Z_1, \dots, Z_{q - 1}, Z_qZ_{q + 1}, Z_{q + 2}, \dots, Z_n) = \phi_{1_\chi}(Z_1, \dots, Z_n) + \sum_{\substack{\pi \in \B\N\C(\chi)\\|\pi| = 2, q \not\sim_\pi q + 1}}\phi_\pi(Z_1, \dots, Z_n)\]
and
\[\Phi_{1_{\chi|_{\setminus q}}}(Z_1, \dots, Z_{q - 1}, Z_qZ_{q + 1}, Z_{q + 2}, \dots, Z_n) = \Phi_{1_\chi}(Z_1, \dots, Z_n) + \sum_{\substack{\pi \in \B\N\C(\chi)\\|\pi| = 2, q \not\sim_\pi q + 1}}\Phi_\pi(Z_1, \dots, Z_n)\]
for all $Z_k \in \mathcal{A}_{\chi(k)}$, where $\Phi_\pi(Z_1, \dots, Z_n)$ is defined by the operator-valued conditionally bi-multiplicativity of $(\phi, \Phi)$ using $\phi$ for an interior block and $\Phi$ for an exterior block.
\end{enumerate}
\end{defn}

The following demonstrates that the two notions of pairs of functions are naturally related by summing over bi-non-crossing partitions.

\begin{thm}
Let $(\A, \mathbb{E}, \mathbb{F}, \varepsilon)$ be a $\mathcal{B}$-$\mathcal{B}$-non-commutative probability space with a pair of $(\B, \D)$-valued expectations. If
\[\phi, \psi: \bigcup_{n \geq 1}\bigcup_{\chi: \{1, \dots, n\} \to \{\ell, r\}}\B\N\C(\chi) \times \A_{\chi(1)} \times \cdots \times \A_{\chi(n)} \to \B\]
and
\[\Phi, \Psi: \bigcup_{n \geq 1}\bigcup_{\chi: \{1, \dots, n\} \to \{\ell, r\}}\B\N\C(\chi) \times \A_{\chi(1)} \times \cdots \times \A_{\chi(n)} \to \D\]
are such that $(\phi, \Phi)$ and $(\psi, \Psi)$ are operator-valued conditionally bi-multiplicative related by the formulae
\[\phi_\pi(Z_1, \dots, Z_n) = \sum_{\substack{\sigma \in \B\N\C(\chi)\\\sigma \leq \pi}}\psi_\sigma(Z_1, \dots, Z_n)\]
and
\[\Phi_\pi(Z_1, \dots, Z_n) = \sum_{\substack{\sigma \in \B\N\C(\chi)\\\sigma \leq \pi}}\Psi_\sigma(Z_1, \dots, Z_n)\]
for all $\chi: \{1, \dots, n\} \to \{\ell, r\}$, $\pi \in \B\N\C(\chi)$, and $Z_k \in \A_{\chi(k)}$, then $(\phi, \Phi)$ is an operator-valued conditionally bi-moment pair if and only if $(\psi, \Psi)$ is an operator-valued conditionally bi-cumulant pair.
\end{thm}

\begin{proof}
Let $\chi: \{1, \dots, n\} \to \{\ell, r\}$ be such that there exists a $q \in \{1, \dots, n - 1\}$ with $\chi(q) = \chi(q + 1)$. If $(\psi, \Psi)$ is an operator-valued conditionally bi-cumulant pair, then
\[\phi_{1_{\chi|_{\setminus q}}}(Z_1, \dots, Z_{q - 1}, Z_qZ_{q + 1}, Z_{q + 2}, \dots, Z_n) = \phi_{1_\chi}(Z_1, \dots, Z_n)\]
for all $Z_k \in \A_{\chi(k)}$ by \cite{CNS2015-2}*{Theorem 6.3.5}. On the other hand, using the operator-valued conditionally bi-multiplicativity of $(\psi, \Psi)$ and part $(2)$ of Definition \ref{CondBiMC}, we have
\[\Psi_\pi(Z_1, \dots, Z_{q - 1}, Z_qZ_{q + 1}, \dots, Z_n) = \sum_{\substack{\sigma \in \B\N\C(\chi)\\\sigma|_{q = q + 1} = \pi}}\Psi_\sigma(Z_1, \dots, Z_n)\]
for all $\pi \in \B\N\C(\chi|_{\setminus q})$, and it follows from the same calculations as in the first part of the proof of \cite{CNS2015-2}*{Theorem 6.3.5} that
\[\Phi_{1_{\chi|_{\setminus q}}}(Z_1, \dots, Z_{q - 1}, Z_qZ_{q + 1}, Z_{q + 2}, \dots, Z_n) = \Phi_{1_\chi}(Z_1, \dots, Z_n)\]
for all $Z_k \in \mathcal{A}_{\chi(k)}$.

Conversely, if $(\phi, \Phi)$ is an operator-valued conditionally bi-moment pair, then
\[\psi_{1_{\chi|_{\setminus q}}}(Z_1, \dots, Z_{q - 1}, Z_qZ_{q + 1}, Z_{q + 2}, \dots, Z_n) = \psi_{1_\chi}(Z_1, \dots, Z_n) + \sum_{\substack{\pi \in \B\N\C(\chi)\\|\pi| = 2, q \not\sim_\pi q + 1}}\psi_\pi(Z_1, \dots, Z_n)\]
for all $Z_k \in \A_{\chi(k)}$ by \cite{CNS2015-2}*{Theorem 6.3.5}, and it follows from the same induction arguments as in the second part of the proof of \cite{CNS2015-2}*{Theorem 6.3.5} that
\[\Psi_{1_{\chi|_{\setminus q}}}(Z_1, \dots, Z_{q - 1}, Z_qZ_{q + 1}, Z_{q + 2}, \dots, Z_n) = \Psi_{1_\chi}(Z_1, \dots, Z_n) + \sum_{\substack{\pi \in \B\N\C(\chi)\\|\pi| = 2, q \not\sim_\pi q + 1}}\Psi_\pi(Z_1, \dots, Z_n)\]
for all $Z_k \in \A_{\chi(k)}$.
\end{proof}

As an immediate corollary, we have the following expected result.

\begin{cor}
Let $(\A, \mathbb{E}, \mathbb{F}, \varepsilon)$ be a $\mathcal{B}$-$\mathcal{B}$-non-commutative probability space with a pair of $(\B, \D)$-valued expectations. The operator-valued conditionally bi-free moment pair $(\E, \F)$ is an operator-valued conditionally bi-moment pair and the operator-valued conditionally bi-free cumulant pair $(\kappa, \K)$ is an operator-valued conditionally bi-cumulant pair.
\end{cor}

\subsection{Operations on operator-valued cumulants}

The following two results demonstrate how certain operations affect operator-valued conditionally bi-free cumulants under certain conditions. The same effects in the scalar-valued setting were observed in \cite{GS2016}*{Lemmata 4.17 and 4.18}.

\begin{lem}
Let $(\A, \mathbb{E}, \mathbb{F}, \varepsilon)$ be a $\mathcal{B}$-$\mathcal{B}$-non-commutative probability space with a pair of $(\B, \D)$-valued expectations. Let $\chi: \{1, \dots, n\} \to \{\ell, r\}$ be such that $\chi(k_0) = \ell$ and $\chi(k_0 + 1) = r$ for some $k_0 \in \{1, \dots, n - 1\}$, and let $X \in \A_\ell$ and $Y \in \A_r$ be such that $\bE(ZXYZ') = \bE(ZYXZ')$ and $\bF(ZXYZ') = \bF(ZYXZ')$ for all $Z, Z' \in \A$.  Define $\chi': \{1, \dots, n\} \to \{\ell, r\}$ by
\[
\chi'(k) = \begin{cases}
r &\text{if } k = k_0\\
\ell &\text{if } k = k_0 + 1\\
\chi(k) &\text{otherwise }
\end{cases}.
\]
Then
\[\K_{1_\chi}(Z_1, \dots, Z_{k_0 - 1}, X, Y, Z_{k_0 + 2}, \dots, Z_n) = \K_{1_{\chi'}}(Z_1, \dots, Z_{k_0 - 1}, Y, X, Z_{k_0 + 2}, \dots, Z_n)\]
for all $Z_1, \dots, Z_{k_0 - 1}, Z_{k_0 + 2}, \dots, Z_n \in \A$ with $Z_k \in \A_{\chi(k)}$.
\end{lem}

\begin{proof}
By repeatedly applying Definition \ref{OpVCBFCumulants} and using Definition \ref{MomentCumulant} for interior blocks, we have
\[\K_{1_\chi}(Z_1, \dots, Z_{k_0 - 1}, X, Y, Z_{k_0 + 2}, \dots, Z_n) = \sum_{(\pi, \iota) \in \B\N\C(\chi, ie)}d(\chi; \pi, \iota)\Theta_{(\pi, \iota)}(Z_1, \dots, Z_{k_0 - 1}, X, Y, Z_{k_0 + 2}, \dots, Z_n)\]
for some integer coefficients such that $d(\chi; \pi, \iota) = 0$ if there is an interior block $V$ of $\pi$ with $i(V) = e$, and $\Theta_{(\pi, \iota)}(Z_1, \dots, Z_{k_0 - 1}, X, Y, Z_{k_0 + 2}, \dots, Z_n)$ for non-zero $d(\chi; \pi, \iota)$ is defined as in the proof of Theorem \ref{VanishingEquiv}. Similarly, we have
\begin{align*}
\K_{1_{\chi'}}(Z_1, \dots, Z_{k_0 - 1}, Y, & X, Z_{k_0 + 2}, \dots, Z_n)  \\
& = \sum_{(\pi', \iota') \in \B\N\C(\chi', ie)}d(\chi'; \pi', \iota')\Theta_{(\pi', \iota')}(Z_1, \dots, Z_{k_0 - 1}, Y, X, Z_{k_0 + 2}, \dots, Z_n).
\end{align*}
Note that there is a bijection from $\B\N\C(\chi, ie)$ to $\B\N\C(\chi', ie)$ which sends a pair $(\pi, \iota)$ to the pair $(\pi', \iota')$ obtained by swapping $k_0$ and $k_0 + 1$. Furthermore, as only the lattice structure affects the expansions of the above formulae (alternatively, by appealing to the scalar-valued case in \cite{GS2016}*{Subsection 4.2}), $d(\chi; \pi, \iota) = d(\chi'; \pi', \iota')$ under this bijection.

To complete the proof, it suffices to show that
\[\Theta_{(\pi, \iota)}(Z_1, \dots, Z_{k_0 - 1}, X, Y, Z_{k_0 + 2}, \dots, Z_n) = \Theta_{(\pi', \iota')}(Z_1, \dots, Z_{k_0 - 1}, Y, X, Z_{k_0 + 2}, \dots, Z_n)\]
for all $(\pi, \iota) \in \B\N\C(\chi, ie)$.
If $k_0$ and $k_0 + 1$ are in the same block of $\pi$, then one may reduce
\[\Theta_{(\pi, \iota)}(Z_1, \dots, Z_{k_0 - 1}, X, Y, Z_{k_0 + 2}, \dots, Z_n)\]
to an expression involving $\bE(ZXYZ')$ or $\bF(ZXYZ')$ for some $Z, Z' \in \A$, commute $X$ and $Y$ to get $\bE(ZYXZ')$ or $\bF(ZYXZ')$, and undo the reduction to obtain
\[\Theta_{(\pi', \iota')}(Z_1, \dots, Z_{k_0 - 1}, Y, X, Z_{k_0 + 2}, \dots, Z_n).\]
On the other hand, if $k_0$ and $k_0 + 1$ are in different blocks of $\pi$, then the reductions of
\[\Theta_{(\pi, \iota)}(Z_1, \dots, Z_{k_0 - 1}, X, Y, Z_{k_0 + 2}, \dots, Z_n)\qand\Theta_{(\pi', \iota')}(Z_1, \dots, Z_{k_0 - 1}, Y, X, Z_{k_0 + 2}, \dots, Z_n)\]
agree. Consequently, the proof is complete.
\end{proof}

\begin{lem}
Let $(\A, \mathbb{E}, \mathbb{F}, \varepsilon)$ be a $\mathcal{B}$-$\mathcal{B}$-non-commutative probability space with a pair of $(\B, \D)$-valued expectations. Let $\chi: \{1, \dots, n\} \to \{\ell, r\}$ be such that $\chi(n) = \ell$, and let $X \in \A_\ell$ and $Y \in \A_r$ be such that $\bE(ZX) = \bE(ZY)$ and $\bF(ZX) = \bF(ZY)$ for all $Z \in \A$.  Define $\chi': \{1, \dots, n\} \to \{\ell, r\}$ by
\[
\chi'(k) = \begin{cases}
r &\text{if } k = n\\
\chi(k) &\text{otherwise }
\end{cases}.
\]
Then
\[\K_{1_\chi}(Z_1, \dots, Z_{n - 1}, X) = \K_{1_{\chi'}}(Z_1, \dots, Z_{n - 1}, Y)\]
for all $Z_1, \dots, Z_{n - 1} \in \A$ with $Z_k \in \A_{\chi(k)}$.
\end{lem}

\begin{proof}
By the same arguments as the previous lemma, we have
\[d(\chi; \pi, \iota)\Theta_{(\pi, \iota)}(Z_1, \dots, Z_{n - 1}, X) = d(\chi'; \pi', \iota')\Theta_{(\pi', \iota')}(Z_1, \dots, Z_{n - 1}, Y)\]
for all $(\pi, \iota) \in \B\N\C(\chi, ie)$, where $(\pi', \iota')$ is obtained from $(\pi, \iota)$ by changing the last node from a left node to a right node. Consequently, the proof is complete.
\end{proof}

In \cite{CNS2015-2}*{Theorem 10.2.1}, it was demonstrated that for a family of $\B$-algebras with certain conditions, bi-free independence over $\B$ can be deduced from free independence over $\B$ of either the left $\B$-algebras or the right $\B$-algebras. The conditionally bi-free analogue in the scalar-valued setting was proved in \cite{GS2016}*{Theorem 4.20}.

\begin{thm}
Let $(\A, \mathbb{E}, \mathbb{F}, \varepsilon)$ be a $\mathcal{B}$-$\mathcal{B}$-non-commutative probability space with a pair of $(\B, \D)$-valued expectations. If $\{(\A_{k, \ell}, \A_{k, r})\}_{k \in K}$ is a family of pairs of $\B$-algebras in $(\A, \mathbb{E}, \mathbb{F}, \varepsilon)$ such that
\begin{enumerate}[$\qquad(1)$]
\item $\A_{m, \ell}$ and $\A_{n, r}$ commute for all $m, n \in K$,

\item for every $Y \in \A_{k, r}$, there exists an $X \in \A_{k, \ell}$ such that $\bE(ZY) = \bE(ZX)$ and $\bF(ZY) = \bF(ZX)$ for all $Z \in \A$,
\end{enumerate}
then $\{(\A_{k, \ell}, \A_{k, r})\}_{k \in K}$ is c-bi-free over $(\B, \D)$ if and only if $\{\A_{k, \ell}\}_{k \in K}$ is c-free over $(\B, \D)$. Consequently, if $\{\A_{k, \ell}\}_{k \in K}$ is c-free over $(\B, \D)$, then $\{\A_{k, r}\}_{k \in K}$ is c-free over $(\B, \D)$.
\end{thm}

\begin{proof}
If $\{(\A_{k, \ell}, \A_{k, r})\}_{k \in K}$ is c-bi-free over $(\B, \D)$, then it is clear that $\{\A_{k, \ell}\}_{k \in K}$ is c-free over $(\B, \D)$ and $\{\A_{k, r}\}_{k \in K}$ is c-free over $(\B, \D)$.

Suppose $\{\A_{k, \ell}\}_{k \in K}$ is c-free over $(\B, \D)$.  Given a mixed operator-valued bi-free or conditionally bi-free cumulant from $\{(\A_{k, \ell}, \A_{k, r})\}_{k \in K}$, assumptions $(1)$ and $(2)$ imply that we may apply the previous two lemmata (or \cite{S2015}*{Lemmata 2.16 and 2.17}) and reduce it to a mixed operator-valued free or conditionally free cumulant from $\{\A_{k, \ell}\}_{k \in K}$, which vanishes by c-free independence over $(\B, \D)$. Thus the result follows from Theorem \ref{VanishingEquiv}.
\end{proof}

\section{The operator-valued conditionally bi-free partial $\mathcal{R}$-transform}\label{sec:R-transform}

In this section, we construct an operator-valued conditionally bi-free partial $\R$-transform generalizing \cite{GS2016}*{Definition 5.3} and relate it to certain operator-valued moment transforms. As we will see in the proof, such transform is a function of three $\B$-variables instead of two by a similar reason as the operator-valued bi-free partial $\R$-transform developed in \cite{S2015}*{Section 5}. As in \cite{S2015}*{Section 5}, our proof will follow the combinatorial techniques used in \cite{S2016-2}*{Section 7}. In that which follows, all algebras are assumed to be Banach algebras.

\begin{defn}
A \textit{Banach $\B$-$\B$-non-commutative probability space with a pair of $(\B, \D)$-valued expectations} is a $\B$-$\B$-non-commutative probability space with a pair of $(\B, \D)$-valued expectations $(\A, \bE, \bF, \varepsilon)$ such that $\A$, $\B$, and $\D$ are Banach algebras, and $\varepsilon|_{\B \otimes 1}$, $\varepsilon_{1 \otimes \B^{\mathrm{op}}}$, $\bE$, and $\bF$ are bounded.
\end{defn}

Let $(\A, \bE, \bF, \varepsilon)$ be a Banach $\B$-$\B$-non-commutative probability space with a pair of $(\B, \D)$-valued expectations, let $Z_\ell \in \A_\ell$, $Z_r \in \A_r$, and let $b, d \in \B$. Consider the following series:
\begin{align*}
M_{Z_\ell}^\ell(b) &= 1 + \sum_{m \geq 1}\bE((L_bZ_\ell)^m),\\
\mathbb{M}_{Z_\ell}^\ell(b) &= 1 + \sum_{m \geq 1}\bF((L_bZ_\ell)^m),\\
\C_{Z_\ell}^\ell(b) &= 1 + \sum_{m \geq 1}\K_{1_{\chi_{m, 0}}}(\underbrace{L_bZ_\ell, \dots, L_bZ_\ell}_{m\,\text{entries}}),
\end{align*}
and
\begin{align*}
M_{Z_r}^r(d) &= 1 + \sum_{n \geq 1}\bE((R_dZ_r)^n),\\
\mathbb{M}_{Z_r}^r(d) &= 1 + \sum_{n \geq 1}\bF((R_dZ_r)^n),\\
\C_{Z_r}^r(d) &= 1 + \sum_{n \geq 1}\K_{1_{\chi_{0, n}}}(\underbrace{R_dZ_r, \dots, R_dZ_r}_{n\,\text{entries}}).
\end{align*}

By similar arguments as in \cite{S2015}*{Remark 5.2}, all of the series above converge absolutely for $b, d$ sufficiently small.

In the proof of Theorem \ref{CR-Transform} below, the following relations will be used. Since the statements are slightly different than the ones in the literature (see, e.g., \cite{BPV2012}*{equation (15)}), we will provide a proof.

\begin{lem}\label{CumulantTransform}
Under the above assumptions and notation, we have
\[\C_{Z_\ell}^\ell\left(M_{Z_\ell}^\ell(b)b\right) = 1 + M_{Z_\ell}^\ell(b) - M_{Z_\ell}^\ell(b)\mathbb{M}_{Z_\ell}^\ell(b)^{-1} \qand \C_{Z_r}^r\left(dM_{Z_r}^r(d)\right) = 1 + M_{Z_r}^r(d) - \mathbb{M}_{Z_r}^r(d)^{-1}M_{Z_r}^r(d)\]
for $b, d$ sufficiently small.
\end{lem}

\begin{proof}
For $m \geq 1$, we have
\[\bF((L_bZ_\ell)^m) = \sum_{\pi \in \B\N\C(\chi_{m, 0})}\K_\pi(\underbrace{L_bZ_\ell, \dots, L_bZ_\ell}_{m\,\text{entries}}).\]
For every partition $\pi \in \B\N\C(\chi_{m, 0})$, let $W_\pi$ denote the block of $\pi$ containing $1$, which is necessarily an exterior block. Rearrange the above sum (which may be done as it converges absolutely) by first choosing $s \in \{1, \dots, m\}$, $W =  \{1 = w_1 < \cdots < w_s\} \subset \{1, \dots, m\}$, and then summing over all $\pi \in \B\N\C(\chi_{m, 0})$ such that $W_\pi = W$, i.e.,
\[\bF((L_bZ_\ell)^m) = \sum_{s = 1}^m\sum_{\substack{W = \{1 = w_1 < \cdots < w_s\}\\W \subset \{1, \dots, m\}}}\sum_{\substack{\pi \in \B\N\C(\chi_{m, 0})\\W_\pi = W}}\K_\pi(\underbrace{L_bZ_\ell, \dots, L_bZ_\ell}_{m\,\text{entries}}).\]
Furthermore, using operator-valued conditionally bi-multiplicative properties, the right-most sum in the above expression is
\[b\K_{1_{\chi_{s, 0}}}(Z_\ell, L_{b_2}Z_\ell, \dots, L_{b_s}Z_\ell)\bF((L_bZ_\ell)^{m - w_s}),\]
where $b_k = \bE((L_bZ_\ell)^{w_k - w_{k - 1} - 1})b$.  Thus
\[\bF((L_bZ_\ell)^m) = \sum_{\substack{1 \leq s \leq m\\0 \leq i_1, \dots, i_s \leq m\\i_1 + \cdots + i_s = m - s}}b\K_{1_{\chi_{s, 0}}}(Z_\ell, L_{f(i_1)}Z_\ell, \dots, L_{f(i_{s - 1})}Z_\ell)\bF((L_bZ_\ell)^{i_s}),\]
where $f(k) = \bE((L_bZ_\ell)^k)b$. Note that
\[\sum_{k \geq 0}f(k) = M_{Z_\ell}^\ell(b)b \qand \sum_{k \geq 0}\bF((L_bZ_\ell)^k) = \mathbb{M}_{Z_\ell}^\ell(b).\]
Consequently, we obtain
\[\sum_{m \geq 1}\bF((L_bZ_\ell)^m) = \sum_{s \geq 1}b\K_{1_{\chi_{s, 0}}}(Z_\ell, L_{M_{Z_\ell}^\ell(b)b}Z_\ell, \dots, L_{M_{Z_\ell}^\ell(b)b}Z_\ell)\mathbb{M}_{Z_\ell}^\ell(b),\]
therefore
\[M_{Z_\ell}^\ell(b)\mathbb{M}_{Z_\ell}^\ell(b) = M_{Z_\ell}^\ell(b) + \sum_{s \geq 1}\K_{1_{\chi_{s, 0}}}(L_{M_{Z_\ell}^\ell(b)b}Z_\ell, L_{M_{Z_\ell}^\ell(b)b}Z_\ell, \dots, L_{M_{Z_\ell}^\ell(b)b}Z_\ell)\mathbb{M}_{Z_\ell}^\ell(b),\]
and hence
\[M_{Z_\ell}^\ell(b) - M_{Z_\ell}^\ell(b)\mathbb{M}_{Z_\ell}^\ell(b)^{-1} = \C_{Z_\ell}^\ell\left(M_{Z_\ell}^\ell(b)b\right) - 1,\]
which proves the first equation. The proof for the second equation is nearly identical once one uses the fact that $d \mapsto R_d$ is an anti-homomorphism.
\end{proof}

For $b, c, d \in \B$, $Z_\ell \in \A_\ell$, and $Z_r \in \A_r$, consider the following series of the pair $(Z_\ell, Z_r)$:
\begin{align*}
M_{(Z_\ell, Z_r)}(b, c, d) &= \sum_{m, n \geq 0}\bE((L_bZ_\ell)^m(R_dZ_r)^nR_c),\\
\mathbb{M}_{(Z_\ell, Z_r)}(b, c, d) &= \sum_{m, n \geq 0}\bF((L_bZ_\ell)^m(R_dZ_r)^nR_c),\\
\C_{(Z_\ell, Z_r)}(b, c, d) &= c + \sum_{m \geq 1}\K_{1_{\chi_{m, 0}}}(\underbrace{L_bZ_\ell, \dots, L_bZ_\ell}_{m - 1\,\text{entries}}, L_bZ_\ell L_c)\\
&\quad + \sum_{\substack{m \geq 0\\n \geq 1}}\K_{1_{\chi_{m, n}}}(\underbrace{L_bZ_\ell, \dots, L_bZ_\ell}_{m\,\text{entries}}, \underbrace{R_dZ_r, \dots, R_dZ_r}_{n - 1\,\text{entries}}, R_dZ_rR_c),
\end{align*}
which converge absolutely for $b, c, d$ sufficiently small by similar arguments as in \cite{S2015}*{Remarks 5.2 and 5.5}.

Notice if $(Z_{1, \ell}, Z_{1, r})$ and $(Z_{2, \ell}, Z_{2, r})$ are c-bi-free over $(\B, \D)$, then
\[\C_{(Z_{1, \ell} + Z_{2, \ell}, Z_{1, r} + Z_{2, r})}(b, c, d) - c = (\C_{(Z_{1, \ell}, Z_{1, r})}(b, c, d) - c) + (\C_{(Z_{2, \ell}, Z_{2, r})}(b, c, d) - c)\]
by Theorem \ref{VanishingEquiv}; that is, $\C_{(Z_\ell, Z_r)}(b, c, d) - c$ is an operator-valued conditionally bi-free partial $\R$-transform.

\begin{thm}\label{CR-Transform}
Let $(\A, \bE, \bF, \varepsilon)$ be a Banach $\B$-$\B$-non-commutative probability space with a pair of $(\B, \D)$-valued expectations, let $Z_\ell \in \A_\ell$, and let $Z_r \in \A_r$. Then
\begin{align*} 
& \C_{(Z_\ell, Z_r)}(M_{Z_\ell}^\ell(b)b,  M_{(Z_\ell, Z_r)}(b, c, d), dM_{Z_r}^r(d))    \\  
& =   M_{Z_\ell}^\ell(b)\mathbb{M}_{Z_\ell}^\ell(b)^{-1}\mathbb{M}_{(Z_\ell, Z_r)}(b, c, d)\mathbb{M}_{Z_r}^r(d)^{-1}M_{Z_r}^r(d) + M_{(Z_\ell, Z_r)}(b, c, d)  \nonumber  \\
& \quad + M_{Z_\ell}^\ell(b)(1 - \mathbb{M}_{Z_\ell}^\ell(b)^{-1})M_{(Z_\ell, Z_r)}(b, c, d) + M_{(Z_\ell, Z_r)}(b, c, d)(1 - \mathbb{M}_{Z_r}^r(d)^{-1})M_{Z_r}^r(d)  - M_{Z_\ell}^\ell(b)cM_{Z_r}^r(d) \nonumber
\end{align*}
for $b, c, d \in \B$ sufficiently small.
\end{thm}

\begin{rem}
Note that if $\B = \D = \bC$, $b = z$, $d = w$, and $c = 1$, then Theorem \ref{CR-Transform} produces exactly equation $(9)$ in \cite{GS2016}*{Theorem 5.6} for the scalar-valued setting. On the other hand, if $\B = \D$ and $\bE = \bF$, then  Theorem \ref{CR-Transform} produces exactly equation $(10)$ in \cite{S2015}*{Theorem 5.6} for the operator-valued bi-free setting.
\end{rem}

\begin{proof}[Proof of Theorem \ref{CR-Transform}]
For $m, n \geq 1$, let $\B\N\C_{\mathrm{vs}}(\chi_{m, n})$ denote the set of bi-non-crossing partitions where no block contains both left and right nodes.  Using operator-valued conditionally bi-multiplicativity, we obtain
\begin{align*}
\bF & ((L_bZ_\ell)^m(R_dZ_r)^nR_c)\\
&= \sum_{\pi \in \B\N\C_{\mathrm{vs}}(\chi_{m, n})}\K_\pi(\underbrace{L_bZ_\ell, \dots, L_bZ_\ell}_{m\,\text{entries}}, \underbrace{R_dZ_r, \dots, R_dZ_r}_{n - 1\,\text{entries}}, R_dZ_rR_c)\\
&\quad + \sum_{\substack{\pi \in \B\N\C(\chi_{m, n})\\\pi \notin \B\N\C_{\mathrm{vs}}(\chi_{m, n})}}\K_\pi(\underbrace{L_bZ_\ell, \dots, L_bZ_\ell}_{m\,\text{entries}}, \underbrace{R_dZ_r, \dots, R_dZ_r}_{n - 1\,\text{entries}}, R_dZ_rR_c)\\
&= \bF((L_bZ_\ell)^m)c\bF((R_dZ_r)^n) + \Theta_{m, n}(b, c, d),
\end{align*}
where $\Theta_{m, n}(b, c, d)$ denotes the sum
\[\sum_{\substack{\pi \in \B\N\C(\chi_{m, n})\\\pi \notin \B\N\C_{\mathrm{vs}}(\chi_{m, n})}}\K_\pi(\underbrace{L_bZ_\ell, \dots, L_bZ_\ell}_{m\,\text{entries}}, \underbrace{R_dZ_r, \dots, R_dZ_r}_{n - 1\,\text{entries}}, R_dZ_rR_c).\]

For every partition $\pi \in \B\N\C(\chi_{m, n}) \setminus \B\N\C_{\mathrm{vs}}(\chi_{m, n})$, let $V_\pi$ denote the block of $\pi$ with both left and right indices such that $\min(V_\pi)$ is the smallest among all blocks of $\pi$ with this property. Note that $V_\pi$ is necessarily an exterior block. Rearrange the sum in $\Theta_{m, n}(b, c, d)$ (which may be done as it converges absolutely) by first choosing $s \in \{1, \dots, m\}$, $t \in \{1, \dots, n\}$, $V \subset \{1, \dots, m + n\}$ such that
\[V_\ell := V \cap \{1, \dots, m\} = \{u_1 < \cdots < u_s\} \qand V_r := V \cap \{m + 1, \dots, m + n\} = \{v_1 < \cdots < v_t\},\]
and then summing over all $\pi \in \B\N\C(\chi_{m, n}) \setminus \B\N\C_{\mathrm{vs}}(\chi_{m, n})$ such that $V_\pi = V$.  The result is
\[\Theta_{m, n}(b, c, d) = \sum_{s = 1}^m\sum_{t = 1}^n\sum_{\substack{V\\V_\ell = \{u_1 < \cdots < u_s\}\\V_r = \{v_1 < \cdots < v_t\}}}\sum_{\substack{\pi \in \B\N\C(\chi_{m, n})\\\pi \notin \B\N\C_{\mathrm{vs}}(\chi_{m, n})\\V_\pi = V}}\K_\pi(\underbrace{L_bZ_\ell, \dots, L_bZ_\ell}_{m\,\text{entries}}, \underbrace{R_dZ_r, \dots, R_dZ_r}_{n - 1\,\text{entries}}, R_dZ_rR_c).\]

Using operator-valued conditionally bi-multiplicative properties, the right-most sum in the above expression is
\[F_b\K_{1_{\chi_{s, t}}}(Z_\ell, L_{b_2}Z_\ell, \dots, L_{b_s}Z_\ell, Z_r, R_{d_2}Z_r, \dots, R_{d_{t - 1}}Z_r, R_{d_t}Z_rR_{\bE((L_bZ_\ell)^{m - u_s}(R_dZ_r)^{n - v_t}R_c)})G_d,\]
where
\begin{gather*}
b_k = \bE((L_bZ_\ell)^{u_k - u_{k - 1} - 1})b, \quad d_k = d\bE((R_dZ_r)^{v_k - v_{k - 1} - 1}), \\ F_b = \bF((L_bZ_\ell)^{u_1 - 1})b, \qand G_d = d\bF((R_dZ_r)^{v_1 - 1}). 
\end{gather*}
Consequently, we obtain that $\Theta_{m, n}(b, c, d)$ equals
\begin{multline}\label{ThetaSum}
\sum_{\substack{1 \leq s \leq m\\0 \leq i_0, i_1, \dots, i_s \leq m\\i_0 + i_1 + \cdots + i_s = m - s}}\sum_{\substack{1 \leq t \leq n\\0 \leq j_0, j_1, \dots, j_t \leq n\\j_0 + j_1 + \cdots + j_t = n - t}}\\
F(i_1)\K_{1_{\chi_{s, t}}}(Z_\ell, L_{f(i_2)}Z_\ell, \dots, L_{f(i_s)}Z_\ell, Z_r, R_{g(j_2)}Z_r, \dots, R_{g(j_t - 1)}Z_r, R_{g(j_t)}Z_rR_{\bE((L_bZ_\ell)^{i_0}(R_dZ_r)^{j_0}R_c)})G(j_1),
\end{multline}
where 
\[f(k) = \bE((L_bZ_\ell)^k)b, \quad g(k) = d\bE((R_dZ_r)^k), \quad F(k) = \bF((L_bZ_\ell)^k)b, \qand G(k) = d\bF((R_dZ_r)^k).\]

Note that
\[\sum_{k \geq 0}f(k) = M_{Z_\ell}^\ell(b)b, \quad \sum_{k \geq 0}g(k) = dM_{Z_r}^r(d), \quad \sum_{k \geq 0}F(k) = \mathbb{M}_{Z_\ell}^\ell(b)b, \qand \sum_{k \geq 0}G(k) = d\mathbb{M}_{Z_r}^r(d).\]
On the other hand, expanding $\mathbb{M}_{(Z_\ell, Z_r)}(b, c, d)$ using the fact everything converges absolutely produces
\begin{align*}
&\mathbb{M}_{(Z_\ell, Z_r)}(b, c, d)\\
&= c + \sum_{m \geq 1}\bF((L_bZ_\ell)^mR_c) + \sum_{n \geq 1}\bF((R_dZ_r)^nR_c) + \sum_{m, n \geq 1}\bF((L_bZ_\ell)^m(R_dZ_r)^nR_c)\\
&= c + \sum_{m \geq 1}\bF((L_bZ_\ell)^mR_c) + \sum_{n \geq 1}\bF((R_dZ_r)^nR_c) + \sum_{m, n \geq 1}\bF((L_bZ_\ell)^m)c\bF((R_dZ_r)^n) + \sum_{m, n \geq 1}\Theta_{m, n}(b, c, d)\\
&= \sum_{m, n \geq 0}\bF((L_bZ_\ell)^m)c\bF((R_dZ_r)^n) + \sum_{m, n \geq 1}\Theta_{m, n}(b, c, d)\\
&= \mathbb{M}_{Z_\ell}^\ell(b)c\mathbb{M}_{Z_r}^r(d) + \sum_{m, n \geq 1}\Theta_{m, n}(b, c, d).
\end{align*}
By rearranging the remaining sum involving $\Theta_{m, n}(b, c, d)$ to sum over all fixed $s, t$ in equation \eqref{ThetaSum}, and by choosing $b, d$ sufficiently small so that $M_{Z_\ell}^\ell(b)$, $M_{Z_r}^r(d)$, $\mathbb{M}_{Z_\ell}^\ell(b)$, and $\mathbb{M}_{Z_r}^r(d)$ are invertible, we obtain
\begin{align*}
&\sum_{m, n \geq 1}\Theta_{m, n}(b, c, d)\\
&= \sum_{s, t \geq 1}\mathbb{M}_{Z_\ell}^\ell(b)b\\
&\quad \times \K_{1_{\chi_{s, t}}}(\underbrace{Z_\ell, L_{M_{Z_\ell}^\ell(b)b}Z_\ell, \dots, L_{M_{Z_\ell}^\ell(b)b}Z_\ell}_{s\,\text{entries}}, \underbrace{Z_r, R_{dM_{Z_r}^r(d)}Z_r, \dots, R_{dM_{Z_r}^r(d)}Z_r}_{t - 1\,\text{entries}}, R_{dM_{Z_r}^r(d)}Z_rR_{M_{(Z_\ell, Z_r)}(b, c, d)}) \\
&\quad \times d\mathbb{M}_{Z_r}^r(d)\\
&= \mathbb{M}_{Z_\ell}^\ell(b)M_{Z_\ell}^\ell(b)^{-1}\\
&\quad \times\sum_{s, t \geq 1}\left(\K_{1_{\chi_{s, t}}}(\underbrace{L_{M_{Z_\ell}^\ell(b)b}Z_\ell, \dots, L_{M_{Z_\ell}^\ell(b)b}Z_\ell}_{s\,\text{entries}}, \underbrace{R_{dM_{Z_r}^r(d)}Z_r, \dots, R_{dM_{Z_r}^r(d)}Z_r}_{t - 1\,\text{entries}}, R_{dM_{Z_r}^r(d)}Z_rR_{M_{(Z_\ell, Z_r)}(b, c, d)})\right)\\
&\quad \times M_{Z_r}^r(d)^{-1}\mathbb{M}_{Z_r}^r(d)\\
&= \mathbb{M}_{Z_\ell}^\ell(b)M_{Z_\ell}^\ell(b)^{-1}[\C_{(Z_\ell, Z_r)}(M_{Z_\ell}^\ell(b)b, M_{(Z_\ell, Z_r)}(b, c, d), dM_{Z_r}^r(d)) - \C_{Z_\ell}^\ell(M_{Z_\ell}^\ell(b)b) M_{(Z_\ell, Z_r)}(b, c, d)\\
&\quad -  M_{(Z_\ell, Z_r)}(b, c, d)\C_{Z_r}^r(dM_{Z_r}^r(d)) +  M_{(Z_\ell, Z_r)}(b, c, d)]M_{Z_r}^r(d)^{-1}\mathbb{M}_{Z_r}^r(d)\\
&= \mathbb{M}_{Z_\ell}^\ell(b)M_{Z_\ell}^\ell(b)^{-1}\C_{(Z_\ell, Z_r)}(M_{Z_\ell}^\ell(b)b, M_{(Z_\ell, Z_r)}(b, c, d), dM_{Z_r}^r(d))M_{Z_r}^r(d)^{-1}\mathbb{M}_{Z_r}^r(d)\\
&\quad - \mathbb{M}_{Z_\ell}^\ell(b)M_{Z_\ell}^\ell(b)^{-1}(1 + M_{Z_\ell}^\ell(b) - M_{Z_\ell}^\ell(b)\mathbb{M}_{Z_\ell}^\ell(b)^{-1})M_{(Z_\ell, Z_r)}(b, c, d)M_{Z_r}^r(d)^{-1}\mathbb{M}_{Z_r}^r(d)\\
&\quad - \mathbb{M}_{Z_\ell}^\ell(b)M_{Z_\ell}^\ell(b)^{-1}M_{(Z_\ell, Z_r)}(b, c, d)(1 + M_{Z_r}^r(d) - \mathbb{M}_{Z_r}^r(d)^{-1}M_{Z_r}^r(d))M_{Z_r}^r(d)^{-1}\mathbb{M}_{Z_r}^r(d)\\
&\quad + \mathbb{M}_{Z_\ell}^\ell(b)M_{Z_\ell}^\ell(b)^{-1}M_{(Z_\ell, Z_r)}(b, c, d)M_{Z_r}^r(d)^{-1}\mathbb{M}_{Z_r}^r(d)\\
&= \mathbb{M}_{Z_\ell}^\ell(b)M_{Z_\ell}^\ell(b)^{-1}\left(\C_{(Z_\ell, Z_r)}(M_{Z_\ell}^\ell(b)b, M_{(Z_\ell, Z_r)}(b, c, d), dM_{Z_r}^r(d)) - M_{(Z_\ell, Z_r)}(b, c, d)\right)M_{Z_r}^r(d)^{-1}\mathbb{M}_{Z_r}^r(d)\\
&\quad - (\mathbb{M}_{Z_\ell}^\ell(b) - 1)M_{(Z_\ell, Z_r)}(b, c, d)M_{Z_r}^r(d)^{-1}\mathbb{M}_{Z_r}^r(d) - \mathbb{M}_{Z_\ell}^\ell(b)M_{Z_\ell}^\ell(b)^{-1}M_{(Z_\ell, Z_r)}(b, c, d)(\mathbb{M}_{Z_r}^r(d) - 1),
\end{align*}
where the fourth equality follows from Lemma \ref{CumulantTransform}. The result now follows by combining these equations.
\end{proof}

\section{Operator-valued conditionally bi-free limit theorems}\label{sec:limit-thms}

In this section, operator-valued conditionally bi-free limit theorems are studied. Recall first from Definition \ref{MCSeriesDefn} that if $\Z = \{Z_i\}_{i \in I} \sqcup \{Z_j\}_{j \in J}$ is a two-faced family in a $\B$-$\B$-non-commutative probability space with a pair of $(\B, \D)$-valued expectations $(\A, \bE, \bF, \varepsilon)$, then the moment and cumulant series $(\nu^\Z, \mu^\Z)$ and $(\rho^\Z, \eta^\Z)$ completely describe the joint distribution of $\Z$ with respect to $(\bE, \bF)$. In that which follows, given a bi-non-crossing partition $\pi \in \B\N\C(\chi)$ it is often convenient to define
\[(\nu_\omega^\Z)_\pi(b_1, \dots, b_{n - 1}), \quad (\mu_\omega^\Z)_\pi(b_1, \dots, b_{n - 1}), \quad (\rho_\omega^\Z)_\pi(b_1, \dots, b_{n - 1}), \qand (\eta_\omega^\Z)_\pi(b_1, \dots, b_{n - 1})\]
by using operator-valued conditionally bi-multiplicativity and replacing $1_{\chi_\omega}$ with $\pi$ in Notation \ref{MCSeries}.

\subsection{The operator-valued c-bi-free central limit theorem}

Like any non-commutative probability theory, the first result is a central limit theorem in the operator-valued c-bi-free setting.

\begin{defn}
A two-faced family $\Z = ((Z_i)_{i \in I}, (Z_j)_{j \in J})$ in a $\B$-$\B$-non-commutative probability space with a pair of $(\B, \D)$-valued expectations $(\A, \bE, \bF, \varepsilon)$ is said to have a \textit{centred $(\B, \D)$-valued c-bi-free Gaussian distribution} if
\[\rho_\omega^\Z(b_1, \dots, b_{n - 1}) = \eta_\omega^\Z(b_1, \dots, b_{n - 1}) = 0\]
for all $n \geq 1$ with $n \neq 2$, $\omega: \{1, \dots, n\} \to I \sqcup J$, and $b_1, \dots, b_{n - 1} \in \mathcal{B}$.
\end{defn}

In view of the definition above and the moment-cumulant formulae, it is enough to specify $\nu_\omega^\Z(b)$ and $\mu_\omega^\Z(b)$ for $\omega: \{1, 2\} \to I \sqcup J$ and $b \in \B$.

\begin{defn}
Let $I$ and $J$ be non-empty disjoint finite sets, let $M_{|I \sqcup J|}(\B)$ and $M_{|I \sqcup J|}(\D)$ denote the $|I \sqcup J|$ by $|I \sqcup J|$ matrices with entries in $\B$ and $\D$ respectively, and let
\[\sigma: \B \to M_{|I \sqcup J|}(\B), \quad b \mapsto (\sigma_{k, \ell}(b))_{k, \ell \in I \sqcup J} \qand \tau: \B \to M_{|I \sqcup J|}(\D), \quad b \mapsto (\tau_{k, \ell}(b))_{k, \ell \in I \sqcup J}\]
be linear maps. A two-faced family $\Z = ((Z_i)_{i \in I}, (Z_j)_{j \in J})$ in a $\B$-$\B$-non-commutative probability space with a pair of $(\B, \D)$-valued expectations $(\A, \bE, \bF, \varepsilon)$ is said to have a \textit{centred $(\B, \D)$-valued c-bi-free Gaussian distribution with covariance matrices $(\sigma, \tau)$} if, in addition to having a centred $(\B, \D)$-valued c-bi-free Gaussian distribution,
\[\nu_\omega^\Z(b) = \sigma_{\omega(1), \omega(2)}(b) \in \B \qand \mu_\omega^\Z(b) = \tau_{\omega(1), \omega(2)}(b) \in \D\]
for all $\omega: \{1, 2\} \to I \sqcup J$ and $b \in \B$.
\end{defn}

\begin{thm}
Let $\{\Z_m = ((Z_{m; i})_{i \in I}, (Z_{m; j})_{j \in J})\}_{m = 1}^\infty$ be a sequence of  two-faced families in a Banach $\B$-$\B$-non-commutative probability space $(\A, \bE, \bF, \varepsilon)$ which are c-bi-free over $(\B, \D)$. Moreover assume
\begin{enumerate}[$\qquad(1)$]
\item $\bE(Z_{m; k}) = \bF(Z_{m; k}) = 0$ for all $m \geq 1$ and $k \in I \sqcup J$;
	
\item $\sup_{m \geq 1}\|\nu_\omega^{\Z_m}(b_1, \dots, b_{n - 1})\| < \infty$ and $\sup_{m \geq 1}\|\mu_\omega^{\Z_m}(b_1, \dots, b_{n - 1})\| < \infty$ for all $n \geq 1$, $\omega: \{1, \dots, n\} \to I \sqcup J$, and $b_1, \dots, b_{n - 1} \in \B$;

\item there are linear maps $\sigma: \B \to M_{|I \sqcup J|}(\B)$ and $\tau: \B \to M_{|I \sqcup J|}(\D)$ such that
\[\lim_{N \to \infty}\frac{1}{N}\sum_{m = 1}^N\nu_\omega^{\Z_m}(b) = \sigma_{\omega(1), \omega(2)}(b) \qand \lim_{N \to \infty}\frac{1}{N}\sum_{m = 1}^N\mu_\omega^{\Z_m}(b) = \tau_{\omega(1), \omega(2)}(b)\]
for all $\omega: \{1, 2\} \to I \sqcup J$ and $b \in \B$.
\end{enumerate}
Then the two-faced families $\{\S_N = ((S_{N ; i})_{i \in I}, (S_{N ; j})_{j \in J})\}_{N = 1}^\infty$, defined by
\[S_{N ; k} = \frac{1}{\sqrt{N}}\sum_{m = 1}^NZ_{m; k}, \quad k \in I \sqcup J,\]
converges in distribution to a two-faced family $\Y = ((Y_i)_{i \in I}, (Y_j)_{j \in J})$ which has a centred $(\B, \D)$-valued c-bi-free Gaussian distribution with covariance matrices $(\sigma, \tau)$.
\end{thm}

\begin{proof}
Since the cumulant series uniquely determine the joint distributions, it suffices to show that
\[\lim_{N \to \infty}\rho_\omega^{\S_N}(b_1, \dots, b_{n - 1}) = \rho_\omega^\Y(b_1, \dots, b_{n - 1}) \qand \lim_{N \to \infty}\eta_\omega^{\S_N}(b_1, \dots, b_{n - 1}) = \eta_\omega^\Y(b_1, \dots, b_{n - 1})\]
for all $n \geq 1$, $\omega: \{1, \dots, n\} \to I \sqcup J$, and $b_1, \dots, b_{n - 1} \in \B$. By definitions, this means
\[\lim_{N \to \infty}\rho_\omega^{\S_N}(b_1, \dots, b_{n - 1}) = \lim_{N \to \infty}\eta_\omega^{\S_N}(b_1, \dots, b_{n - 1}) = 0\]
for all $\omega: \{1, \dots, n\} \to I \sqcup J$ such that $n \neq 2$,
\[\lim_{N \to \infty}\rho_\omega^{\S_N}(b) = \sigma_{\omega(1), \omega(2)}(b) \qand \lim_{N \to \infty}\eta_\omega^{\S_N}(b) = \tau_{\omega(1), \omega(2)}(b)\]
for all $\omega: \{1, 2\} \to I \sqcup J$ and $b \in \B$. 

For fixed $n \geq 1$, $\omega: \{1, \dots, n\} \to I \sqcup J$, and $b_1, \dots, b_{n - 1} \in \B$, by the additive and multilinear properties of cumulants, we have
\begin{align*}
E &:= \lim_{N \to \infty}\rho_\omega^{\S_N}(b_1, \dots, b_{n - 1}) = \lim_{N \to \infty}\frac{1}{N^{n/2}}\sum_{m = 1}^N\rho_\omega^{\Z_m}(b_1, \dots, b_{n - 1}),\\
F &:= \lim_{N \to \infty}\eta_\omega^{\S_N}(b_1, \dots, b_{n - 1}) = \lim_{N \to \infty}\frac{1}{N^{n/2}}\sum_{m = 1}^N\eta_\omega^{\Z_m}(b_1, \dots, b_{n - 1}).
\end{align*} 
If $n = 1$, then  
\begin{align*}
E &= \lim_{N \to \infty}\frac{1}{\sqrt{N}}\sum_{m = 1}^N\bE(Z_{m; \omega(1)}) = 0,\\
F &= \lim_{N \to \infty}\frac{1}{\sqrt{N}}\sum_{m = 1}^N\bF(Z_{m; \omega(1)}) = 0
\end{align*}
by assumption $(1)$. If $n \geq 3$, then assumption $(2)$ and operator-valued conditionally bi-multiplicativity imply
\[\sup_{m \geq 1}\|\rho_\omega^{\Z_m}(b_1, \dots, b_{n - 1})\| := B < \infty \qand \sup_{m \geq 1}\|\eta_\omega^{\Z_m}(b_1, \dots, b_{n - 1})\| := D < \infty,\]
hence
\[\|E\| \leq \lim_{N \to \infty}\frac{B}{N^{(n - 2)/2}} = 0 \qand \|F\| \leq \lim_{N \to \infty}\frac{D}{N^{(n - 2)/2}} = 0.\]
Otherwise $n = 2$ and
\[E = \lim_{N \to \infty}\frac{1}{N}\sum_{m = 1}^N\rho_\omega^{\Z_m}(b) = \lim_{N \to \infty}\frac{1}{N}\sum_{m = 1}^N\nu_\omega^{\Z_m}(b) = \sigma_{\omega(1), \omega(2)}(b),\]
and similarly $F = \tau_{\omega(1), \omega(2)}(b)$, for all $\omega: \{1, 2\} \to I \sqcup J$ and $b \in \B$ by assumptions $(1)$ and $(3)$.
\end{proof}

\subsection{The operator-valued compound c-bi-free Poisson limit theorem}

The next result is a Poisson type limit theorem in the operator-valued c-bi-free setting. In what follows, all two-faced families are assumed to have non-empty disjoint left and right index sets $I$ and $J$, respectively. To formulate the statement, we introduce the following notation.

Let $(\nu_1, \mu_1)$ and $(\nu_2, \mu_2)$ be the moment series of two-faced families. For $\lambda \in \bR$, denote by
\[(\lambda\nu_1 + (1 - \lambda)\nu_2, \lambda\mu_1 + (1 - \lambda)\mu_2)\]
the moment series of some two-faced family such that
\[(\lambda\nu_1 + (1 - \lambda)\nu_2)_\omega(b_1, \dots, b_{n - 1}) = \lambda(\nu_1)_\omega(b_1, \dots, b_{n - 1}) + (1 - \lambda)(\nu_2)_\omega(b_1, \dots, b_{n - 1})\]
and
\[(\lambda\mu_1 + (1 - \lambda)\mu_2)_\omega(b_1, \dots, b_{n - 1}) = \lambda(\mu_1)_\omega(b_1, \dots, b_{n - 1}) + (1 - \lambda)(\mu_2)_\omega(b_1, \dots, b_{n - 1})\]
for all $n \geq 1$, $\omega: \{1, \dots, n\} \to I \sqcup J$, and $b_1, \dots, b_{n - 1} \in \B$. Such a realization always exists by similar (and simpler) constructions as in the proofs of \cite{S2015}*{Lemma 3.8} and Lemma \ref{Existence}. Moreover, let $(\nu^\delta, \mu^\delta)$ be the special moment series such that
\[\nu_\omega^\delta(b_1, \dots, b_{n - 1}) = \mu_\omega^\delta(b_1, \dots, b_{n - 1}) = 0\]
for all $n \geq 1$, $\omega: \{1, \dots, n\} \to I \sqcup J$, and $b_1, \dots, b_{n - 1} \in \B$.

\begin{defn}
Let $(\nu, \mu)$ be the moment series of some two-faced family and let $\lambda \in \bR$. A two-faced family $\Z = ((Z_i)_{i \in I}, (Z_j)_{j \in J})$ in a $\B$-$\B$-non-commutative probability space with a pair of $(\B, \D)$-valued expectations $(\A, \bE, \bF, \varepsilon)$ is said to have a \textit{$(\B, \D)$-valued compound c-bi-free Poisson distribution with rate $\lambda$ and jump distribution $(\nu, \mu)$} if
\[\rho_\omega^\Z(b_1, \dots, b_{n - 1}) = \lambda\nu_\omega(b_1, \dots, b_{n - 1}) \qand \eta_\omega^\Z(b_1, \dots, b_{n - 1}) = \lambda\mu_\omega(b_1, \dots, b_{n - 1})\]
for all $n \geq 1$, $\omega: \{1, \dots, n\} \to I \sqcup J$, and $b_1, \dots, b_{n - 1} \in \B$.
\end{defn}

\begin{thm}
Let $(\nu, \mu)$ be the moment series of some two-faced family, let $\lambda \in \bR$, and consider the sequence $\{(\nu_N, \mu_N)\}_{N = 1}^\infty$ of moment series defined by
\[\nu_N = \left(1 - \frac{\lambda}{N}\right)\nu^\delta + \frac{\lambda}{N}\nu \qand \mu_N = \left(1 - \frac{\lambda}{N}\right)\mu^\delta + \frac{\lambda}{N}\mu.\]
If $\{\Z_{N; m} = ((Z_{N; m; i})_{i \in I}, (Z_{N; m; j})_{j \in J})\}_{m = 1}^N$ is a sequence of identically distributed two-faced families in a $\B$-$\B$-non-commutative probability space with a pair of $(\B, \D)$-valued expectations $(\A, \bE, \bF, \varepsilon)$ which are c-bi-free over $(\B, \D)$ with moment series $(\nu_N, \mu_N)$, then the two-faced families $\{\S_N = ((S_{N ; i})_{i \in I}, (S_{N ; j})_{j \in J})\}_{N = 1}^\infty$, defined by
\[S_{N ; k} = \sum_{m = 1}^NZ_{N; m; k}, \quad k \in I \sqcup J,\]
converges in distribution to a two-faced family $\Z = ((Z_i)_{i \in I}, (Z_j)_{j \in J})$ which has a $(\B, \D)$-valued compound c-bi-free Poisson distribution with rate $\lambda$ and jump distribution $(\nu, \mu)$.
\end{thm}

\begin{proof}
For each $N \geq 1$, let $(\rho_N, \eta_N)$ be the cumulant series corresponding to $(\nu_N, \mu_N)$. For $n \geq 1$, $\omega: \{1, \dots, n\} \to I \sqcup J$, and $b_1, \dots, b_{n - 1} \in \B$, we have
\begin{align*}
(\rho_N)_\omega(b_1, \dots, b_{n - 1}) &= \sum_{\pi \in \B\N\C(\chi_\omega)}(\nu_N)_\pi(b_1, \dots, b_{n - 1})\mu_{\B\N\C}(\pi, 1_{\chi_\omega})\\
&= (\nu_N)_\omega(b_1, \dots, b_{n - 1}) + O(1/N^2)\\
&= \frac{\lambda}{N}\nu_\omega(b_1, \dots, b_{n - 1}) + O(1/N^2),
\end{align*}
and thus
\begin{align*}
\rho_\omega^\Z(b_1, \dots, b_{n - 1}) &= \lim_{N \to \infty}\rho_\omega^{\S_N}(b_1, \dots, b_{n - 1})\\
&= \lim_{N \to \infty}\left(\lambda\nu_\omega(b_1, \dots, b_{n - 1}) + O(1/N)\right)\\
&= \lambda\nu_\omega(b_1, \dots, b_{n - 1}).
\end{align*}
Similarly, we have $(\eta_N)_\pi(b_1, \dots, b_{n - 1}) = O(1/N^2)$ for $\pi \in \B\N\C(\chi_\omega)$ with at least two blocks, therefore
\[(\eta_N)_\omega(b_1, \dots, b_{n - 1}) = (\mu_N)_\omega(b_1, \dots, b_{n - 1}) + O(1/N^2) = \frac{\lambda}{N}\mu_\omega(b_1, \dots, b_{n - 1}) + O(1/N^2),\]
and thus
\[\eta_\omega^\Z(b_1, \dots, b_{n - 1}) = \lim_{N \to \infty}\eta_\omega^{\S_N}(b_1, \dots, b_{n - 1}) = \lambda\mu_\omega(b_1, \dots, b_{n - 1})\]
as required.
\end{proof}

\subsection{A general operator-valued c-bi-free limit theorem}

We finish this section with an operator-valued analogue of \cite{GS2016}*{Theorem 6.8}.

\begin{lem}\label{LimitMC}
For every $N \in \mathbb{N}$, let $\Z_N = ((Z_{N; i})_{i \in I}, (Z_{N; j})_{j \in J})$ be a two-faced family in a Banach $\mathcal{B}$-$\mathcal{B}$-non-commutative-probability space with a pair of $(\B, \D)$-valued expectations $(\mathcal{A}_N, \bE_{\A_N}, \bF_{\A_N}, \varepsilon_N)$. The following assertions are equivalent.
\begin{enumerate}[$\qquad(1)$]
\item For all $n \geq 1$, $\omega: \{1, \dots, n\} \to I \sqcup J$, and $b_1, \dots, b_{n - 1} \in \mathcal{B}$, the limits
\[\lim_{N \to \infty}N\nu_\omega^{\Z_n}(b_1, \dots, b_{n - 1}) \in \B \qand \lim_{N \to \infty}N\mu_\omega^{\Z_n}(b_1, \dots, b_{n - 1}) \in \D\]
exist.

\item For all $n \geq 1$, $\omega: \{1, \dots, n\} \to I \sqcup J$, and $b_1, \dots, b_{n - 1} \in \mathcal{B}$, the limits
\[\lim_{N \to \infty}N\rho_\omega^{\Z_n}(b_1, \dots, b_{n - 1}) \in \B \qand \lim_{N \to \infty}N\eta_\omega^{\Z_n}(b_1, \dots, b_{n - 1}) \in \D\]
exist.
\end{enumerate}
Moreover, if these assertions hold, then
\begin{align*}
&\lim_{N \to \infty}N\nu_\omega^{\Z_n}(b_1, \dots, b_{n - 1}) = \lim_{N \to \infty}N\rho_\omega^{\Z_n}(b_1, \dots, b_{n - 1}) \qand  \\
&\lim_{N \to \infty}N\mu_\omega^{\Z_n}(b_1, \dots, b_{n - 1}) = \lim_{N \to \infty}N\eta_\omega^{\Z_n}(b_1, \dots, b_{n - 1})
\end{align*}
for all $n \geq 1$, $\omega: \{1, \dots, n\} \to I \sqcup J$, and $b_1, \dots, b_{n - 1} \in \mathcal{B}$.
\end{lem}

\begin{proof}
Suppose assertion $(2)$ holds. Since $(\kappa_{\A_N}, \K_{\A_N})$ is operator-valued conditionally bi-multiplicative, we have
\[(\rho_\omega^{\Z_n})_\pi(b_1, \dots, b_{n - 1}) = O\left(1/N^2\right) \qand (\eta_\omega^{\Z_n})_\pi(b_1, \dots, b_{n - 1}) = O\left(1/N^2\right)\]
for $\pi \in \B\N\C(\chi_\omega)$ with at least two blocks. Hence
\begin{align*}
\lim_{N \to \infty}N\nu_\omega^{\Z_n}(b_1, \dots, b_{n - 1}) &= \lim_{N \to \infty}N\sum_{\pi \in \B\N\C(\chi_\omega)}(\rho_\omega^{\Z_n})_\pi(b_1, \dots, b_{n - 1})\\
&= \lim_{N \to \infty}\left(N\rho_\omega^{\Z_n}(b_1, \dots, b_{n - 1}) + O\left(1/N\right)\right),
\end{align*}
and similarly
\[\lim_{N \to \infty}N\mu_\omega^{\Z_n}(b_1, \dots, b_{n - 1}) = \lim_{N \to \infty}\left(N\eta_\omega^{\Z_n}(b_1, \dots, b_{n - 1}) + O\left(1/N\right)\right)\]
for all $n \geq 1$, $\omega: \{1, \dots, n\} \to I \sqcup J$, and $b_1, \dots, b_{n - 1} \in \mathcal{B}$.

The proof for the other direction is analogous by the operator-valued conditionally bi-multiplicativity of $(\E_{\A_N}, \F_{\A_N})$ and the moment-cumulant formulae from Definitions \ref{MomentCumulant} and \ref{OpVCBFCumulants}.
\end{proof}

\begin{thm}\label{LimitThm}
For every $N \in \bN$, let $(\A_N, \bE_N, \bF_N, \varepsilon_N)$ be a Banach $\B$-$\B$-non-commutative-probability space with a pair of $(\B, \D)$-valued expectations and let $\{\Z_{N; m} = ((Z_{N; m; i})_{i \in I}, (Z_{N; m; j})_{j \in J})\}_{m = 1}^N$ be a sequence of identically distributed two-faced families in $(\A_N, \bE_N, \bF_N, \varepsilon_N)$ which are c-bi-free over $(\B, \D)$. Furthermore, let $\S_N = ((S_{N; i})_{i \in I}, (S_{N; j})_{j \in J})$ be the two-faced family in $(\A_N, \bE_N, \bF_N, \varepsilon_N)$ defined by
\[S_{N; k} = \sum_{m = 1}^NZ_{N; m; k},\,\,\,\,\,k \in I \sqcup J.\]
The following assertions are equivalent.
\begin{enumerate}[$\qquad(1)$]
\item There exists a two-faced family $\Y = ((Y_i)_{i \in I}, (Y_j)_{j \in J})$ in a Banach $\B$-$\B$-non-commutative-probability space with a pair of $(\B, \D)$-valued expectations $(\A, \bE, \bF, \varepsilon)$ such that $\S_N$ converges in distribution to $\Y$ as $N \to \infty$.

\item For all $n \geq 1$, $\omega: \{1, \dots, n\} \to I \sqcup J$, and $b_1, \dots, b_{n - 1} \in \B$, the limits
\[\lim_{N \to \infty}N\nu_\omega^{\Z_{N; m}}(b_1, \dots, b_{n - 1}) \qand \lim_{N \to \infty}N\mu_\omega^{\Z_{N; m}}(b_1, \dots, b_{n - 1})\]
exist and are independent of $m$.
\end{enumerate}
Moreover, if these assertions hold, then the operator-valued bi-free and conditionally bi-free cumulants of $\Y$ are given by
\[\rho_\omega^{\Y}(b_1, \dots, b_{n - 1}) = \lim_{N \to \infty}N\nu_\omega^{\Z_{N; m}}(b_1, \dots, b_{n - 1}) \qand \eta_\omega^{\Y}(b_1, \dots, b_{n - 1}) = \lim_{N \to \infty}N\mu_\omega^{\Z_{N; m}}(b_1, \dots, b_{n - 1})\]
for all $n \geq 1$, $\omega: \{1, \dots, n\} \to I \sqcup J$, and $b_1, \dots, b_{n - 1} \in \B$.
\end{thm}

\begin{proof}
Suppose assertion $(1)$ holds. For $n \geq 1$, $\omega: \{1, \dots, n\} \to I \sqcup J$, and $b_1, \dots, b_{n - 1} \in \B$, we have
\begin{align*}
&\nu_\omega^{\Y}(b_1, \dots, b_{n - 1}) \\
&= \lim_{N \to \infty}\nu_\omega^{\S_N}(b_1, \dots, b_{n - 1})\\
&= \lim_{N \to \infty}(\E_{\A_N})_{1_{\chi_\omega}}\left(S_{N; \omega(1)}, C^{\omega(2)}_{b_1}S_{N; \omega(2)}, \dots, C^{\omega(n)}_{b_{n - 2}}S_{N; \omega(n)}C^{\omega(n)}_{b_{n - 1}}\right)\\
&= \lim_{N \to \infty}\sum_{t(1), \dots, t(n) = 1}^N(\E_{\A_N})_{1_{\chi_\omega}}\left(Z_{N; t(1); \omega(1)}, C^{\omega(2)}_{b_1}Z_{N; t(2); \omega(2)}, \dots, C^{\omega(n)}_{b_{n - 2}}Z_{N; t(n); \omega(n)}C^{\omega(n)}_{b_{n - 1}}\right)\\
&= \lim_{N \to \infty}\sum_{t(1), \dots, t(n) = 1}^N\sum_{\pi \in \B\N\C(\chi_\omega)}(\kappa_{\A_N})_{\pi}\left(Z_{N; t(1); \omega(1)}, C^{\omega(2)}_{b_1}Z_{N; t(2); \omega(2)}, \dots, C^{\omega(n)}_{b_{n - 2}}Z_{N; t(n); \omega(n)}C^{\omega(n)}_{b_{n - 1}}\right)\\
&= \lim_{N \to \infty}\sum_{\pi \in \B\N\C(\chi_\omega)}N^{|\pi|}(\kappa_{\A_N})_{\pi}\left(Z_{N; m; \omega(1)}, C^{\omega(2)}_{b_1}Z_{N; m; \omega(2)}, \dots, C^{\omega(n)}_{b_{n - 2}}Z_{N; m; \omega(n)}C^{\omega(n)}_{b_{n - 1}}\right)\\
&= \lim_{N \to \infty}\sum_{\pi \in \B\N\C(\chi_\omega)}N^{|\pi|}(\rho_\omega^{\Z_{N; m}})_\pi(b_1, \dots, b_{n - 1}),
\end{align*}
where the second equality follows from Notation \ref{MCSeries} by assuming $\{\omega(k)\}_{k = 1}^n$ intersects both $I$ and $J$ (the special cases that $\omega(k) \in I$ or $\omega(k) \in J$ for all $k$ can be checked similarly), and the fifth equality, which is independent of $m$, follows from the assumptions of c-bi-free independence over $(\B, \D)$ and identical distribution. Since $\nu_\omega^{\Y}(b_1, \dots, b_{n - 1})$ exist for all $n \geq 1$ and $\omega: \{1, \dots, n\} \to I \sqcup J$, it can be shown by induction on $n$ that the limits
\[\lim_{N \to \infty}N^{|\pi|}(\rho_\omega^{\Z_{N; m}})_\pi(b_1, \dots, b_{n - 1})\]
exist for all $n \geq 1$, $\omega: \{1, \dots, n\} \to I \sqcup J$, $\pi \in \B\N\C(\chi_\omega)$, and $b_1, \dots, b_{n - 1} \in \B$. Indeed, the base case $n = 1$ follows from the assumption that
\[\lim_{N \to \infty}N(\kappa_{\A_n})_{1_{\chi_\omega}}\left(Z_{N; m; \omega(1)}\right) = \lim_{N \to \infty}(\kappa_{\A_N})_{1_{\chi_\omega}}\left(S_{N; \omega(1)}\right) = \lim_{N \to \infty}(\E_{\A_N})_{1_{\chi_\omega}}\left(S_{N; \omega(1)}\right) = \E_{1_{\chi_\omega}}(Y_{\omega(1)})\]
exist for all $\omega: \{1\} \to I \sqcup J$. For the inductive step, the limit
\[\nu_\omega^{\Y}(b_1, \dots, b_{n - 1})\]
exists by assumption, and the limit
\[\lim_{N \to \infty}\sum_{\substack{\pi \in \B\N\C(\chi_\omega)\\\pi \neq 1_{\chi_\omega}}}N^{|\pi|}(\rho_\omega^{\Z_{N; m}})_\pi(b_1, \dots, b_{n - 1})\]
exists by induction hypothesis with operator-valued conditionally bi-multiplicativity, thus the limit
\[\lim_{N \to \infty}N\rho_\omega^{\Z_{N; m}}(b_1, \dots, b_{n - 1}),\]
exists, and equals
\[\lim_{N \to \infty}N\nu_\omega^{\Z_{N; m}}(b_1, \dots, b_{n - 1})\]
by Lemma \ref{LimitMC}. On the other hand, it follows from a similar calculation as above that
\begin{align*}
\mu_\omega^{\Y}(b_1, \dots, b_{n - 1}) &= \lim_{N \to \infty}\mu_\omega^{\S_N}(b_1, \dots, b_{n - 1})\\
&= \lim_{N \to \infty}\sum_{\pi \in \B\N\C(\chi_\omega)}N^{|\pi|}(\eta_\omega^{\Z_{N; m}})_\pi(b_1, \dots, b_{n - 1})
\end{align*}
for all $n \geq 1$, $\omega: \{1, \dots, n\} \to I \sqcup J$, and $b_1, \dots, b_{n - 1} \in \B$, and a similar induction argument on $n$ shows that the limits
\[\lim_{N \to \infty}N^{|\pi|}(\eta_\omega^{\Z_{N; m}})_\pi(b_1, \dots, b_{n - 1})\]
exist for all $n \geq 1$, $\omega: \{1, \dots, n\} \to I \sqcup J$, $\pi \in \B\N\C(\chi_\omega)$, and $b_1, \dots, b_{n - 1} \in \B$. In particular, choose $\pi = 1_{\chi_\omega}$ and apply Lemma \ref{LimitMC}, we obtain the existence of the limit
\[\lim_{N \to \infty}N\mu_\omega^{\Z_{N; m}}(b_1, \dots, b_{n - 1}).\]
	
Conversely, suppose assertion $(2)$ holds. By Lemma \ref{LimitMC} and operator-valued conditionally bi-multiplicativity, the limits
\[\lim_{N \to \infty}N^{|\pi|}(\rho_\omega^{\Z_{N; m}})_\pi(b_1, \dots, b_{n - 1}) \qand \lim_{N \to \infty}N^{|\pi|}(\eta_\omega^{\Z_{N; m}})_\pi(b_1, \dots, b_{n - 1})\]
exist for all $n \geq 1$, $\omega: \{1, \dots, n\} \to I \sqcup J$, $\pi \in \B\N\C(\chi_\omega)$, and $b_1, \dots, b_{n - 1} \in \B$. Therefore, by the calculations above,
\begin{equation}\label{LimitingMomentsNu}
\lim_{N \to \infty}\nu_\omega^{\S_N}(b_1, \dots, b_{n - 1}) = \sum_{\pi \in \B\N\C(\chi_\omega)}\lim_{N \to \infty}N^{|\pi|}(\rho_\omega^{\Z_{N; m}})_\pi(b_1, \dots, b_{n - 1})
\end{equation}
and
\begin{equation}\label{LimitingMomentsMu}
\lim_{N \to \infty}\mu_\omega^{\S_N}(b_1, \dots, b_{n - 1}) = \sum_{\pi \in \B\N\C(\chi_\omega)}\lim_{N \to \infty}N^{|\pi|}(\eta_\omega^{\Z_{N; m}})_\pi(b_1, \dots, b_{n - 1}),
\end{equation}
and hence these limits exist. One concludes assertion $(1)$ by using Lemma \ref{Existence} to construct a two-faced family $\Y = ((Y_i)_{i \in I}, (Y_j)_{j \in J})$ in a Banach $\B$-$\B$-non-commutative-probability space with a pair of $(\B, \D)$-valued expectations $(\A, \bE, \bF, \varepsilon)$ and define $\nu_\omega^{\Y}(b_1, \dots, b_{n - 1})$ and $\mu_\omega^{\Y}(b_1, \dots, b_{n - 1})$ to be the corresponding limit in equations \eqref{LimitingMomentsNu} and \eqref{LimitingMomentsMu} respectively.
	
Finally, for $n \geq 1$, $\omega: \{1, \dots, n\} \to I \sqcup J$, and $b_1, \dots, b_{n - 1} \in \B$, we have
\begin{align*}
\sum_{\pi \in \B\N\C(\chi_\omega)}(\rho_\omega^{\Y})_\pi(b_1, \dots, b_{n - 1}) &= \nu_\omega^{\Y}(b_1, \dots, b_{n - 1})\\
&= \lim_{N \to \infty}\nu_\omega^{\S_N}(b_1, \dots, b_{n - 1})\\
&= \sum_{\pi \in \B\N\C(\chi_\omega)}\lim_{N \to \infty}N^{|\pi|}(\rho_\omega^{\Z_{N; m}})_\pi(b_1, \dots, b_{n - 1}),
\end{align*}
and similarly
\[\sum_{\pi \in \B\N\C(\chi_\omega)}(\eta_\omega^{\Y})_\pi(b_1, \dots, b_{n - 1}) = \sum_{\pi \in \B\N\C(\chi_\omega)}\lim_{N \to \infty}N^{|\pi|}(\eta_\omega^{\Z_{N; m}})_\pi(b_1, \dots, b_{n - 1}).\]
A similar induction argument on $n$ shows that
\begin{align*}
&(\rho_\omega^{\Y})_\pi(b_1, \dots, b_{n - 1}) = \lim_{N \to \infty}N^{|\pi|}(\rho_\omega^{\Z_{N; m}})_\pi(b_1, \dots, b_{n - 1}) \qand \\
& (\eta_\omega^{\Y})_\pi(b_1, \dots, b_{n - 1}) = \lim_{N \to \infty}N^{|\pi|}(\eta_\omega^{\Z_{N; m}})_\pi(b_1, \dots, b_{n - 1})
\end{align*}
for all $\pi \in \B\N\C(\chi_\omega)$, from which the last claims follow from Lemma \ref{LimitMC} applied to $\pi = 1_{\chi_\omega}$.
\end{proof}

\end{document}